\documentclass[11pt, reqno]{amsart}
\usepackage{amsmath, amsthm, amscd, amsfonts, amssymb, graphicx, color, mathtools, mathrsfs}
\usepackage[bookmarksnumbered, colorlinks, plainpages]{hyperref}
\usepackage{enumerate}
\usepackage{setspace}
\usepackage{multicol}
\usepackage[margin=1in]{geometry}
\usepackage{comment}
\usepackage{booktabs}
\usepackage{caption}
\usepackage{subcaption}
\usepackage{pgfplots}
\usepackage[normalem]{ulem}
\newcommand\tsout{\bgroup\markoverwith{\textcolor{red}{\rule[0.5ex]{2pt}{1.4pt}}}\ULon}
\newcommand{\stkout}[1]{\ifmmode\text{\tsout{\ensuremath{#1}}}\else\tsout{#1}\fi}

\theoremstyle{definition}
\newtheorem{theorem}{Theorem}[section]
\newtheorem{lemma}[theorem]{Lemma}
\newtheorem{proposition}[theorem]{Proposition}

\newtheorem{definition}[theorem]{Definition}
\newtheorem{algorithm}[theorem]{Algorithm}
\newtheorem{assumption}[theorem]{Assumption}

\newtheorem{remark}[theorem]{Remark}
\numberwithin{equation}{section}

\newcommand{\bff}{\boldsymbol}
\newcommand{\bb}{\mathbb}

\newcommand{\dt}{\mathrm{d}t}
\newcommand{\dtt}{\mathrm{d}_\tau}

\newcommand{\ddt}{\frac{\mathrm{d}}{\mathrm{d}t}}
\newcommand{\dx}{\mathrm{d}x}
\newcommand{\ds}{\mathrm{d}s}

\newcommand{\norm}[2]{\left\|{#1}\right\|_{#2}}
\newcommand{\inpro}[2]{\left\langle#1,#2\right\rangle}
\newcommand{\abs}[1]{\left|{#1}\right|}

\allowdisplaybreaks

\begin{document}
	\setcounter{page}{1}
	
	\title[Error analysis of SAV-FEM for tumour growth model]
	{Error analysis of a fully discrete structure-preserving finite element scheme for a diffuse-interface model of tumour growth}
	
	\author[Agus L. Soenjaya]{Agus L. Soenjaya}
	\address{School of Mathematics and Statistics, The University of New South Wales, Sydney 2052, Australia}
	\email{\textcolor[rgb]{0.00,0.00,0.84}{a.soenjaya@unsw.edu.au}}
	
	\author[Ping Lin]{Ping Lin}
	\address{Division of Mathematics, University of Dundee, Dundee, DD1 4HN, UK}
	\email{\textcolor[rgb]{0.00,0.00,0.84}{p.lin@dundee.ac.uk}}
	
	\author[Thanh Tran]{Thanh Tran}
	\address{School of Mathematics and Statistics, The University of New South Wales, Sydney 2052, Australia}
	\email{\textcolor[rgb]{0.00,0.00,0.84}{thanh.tran@unsw.edu.au}}
	
	\date{\today}
	
	\keywords{}
	\subjclass{}
	
	\begin{abstract}
	   We develop a linear fully discrete structure-preserving finite element method for a diffuse-interface model of tumour growth. The system couples a Cahn--Hilliard type equation with a nonlinear reaction-diffusion equation for nutrient concentration and admits a dissipative energy law at the continuous level. For the discretisation, we employ a scalar auxiliary variable (SAV) formulation together with a mixed finite element method for the Cahn--Hilliard part and standard conforming finite elements for the reaction-diffusion equation in space, combined with a first-order time-stepping scheme. The resulting method is unconditionally energy-stable, mass-preserving, and inherits a discrete energy dissipation law associated with the SAV-based approximate energy functional, while requiring the solution of only linear systems at each time step. Under suitable regularity assumptions on the exact solution, we derive rigorous error estimates in the $L^2$, $H^1$, and $L^\infty$ norms, establishing first-order accuracy in time and optimal-order accuracy in space. A key step in this analysis is the proof of boundedness of the numerical solutions in $L^\infty$. Numerical experiments validate the theoretical convergence rates and demonstrate the robustness of the method in capturing characteristic phenomena such as aggregation and chemotactic tumour growth.
	\end{abstract}
	\maketitle
	\tableofcontents
    
\section{Introduction}

Based on the continuum theory of mixtures~\cite{OdeHawPru10}, Hawkins-Daarud, van der Zee, and Oden develop a thermodynamically consistent four-species model of tumour growth~\cite{HawZeeOde12} (see also~\cite{CriLiLowWis09, Fri23} for related models and \cite{WanLinYan25} for a simplified derivation). The formulation introduces several unknowns, which we now describe. {Let $\mathscr{D}\subset \bb{R}^3$ be an open, bounded domain, and $T>0$. Let $u:(0,T)\times\mathscr{D}\to\bb{R}$ and $n:(0,T)\times\mathscr{D}\to\bb{R}$ denote the tumour cell volume fraction and the nutrient-rich extracellular water volume fraction, respectively.} Physical consideration dictates that the Helmholtz free energy of the system is given by
\begin{align}\label{equ:free ener}
	\mathcal{E}[u,n]:= \int_\mathscr{D} \left( \frac{\epsilon^2}{2} \abs{\nabla u}^2 + \frac{\lambda}{2} u^2 + g(u) + \frac{1}{2\delta} n^2  - \chi_0 un\right) \dx,
\end{align}
where $g(u)$ is the classical Cahn--Hilliard free-energy density function.
Here, $\epsilon^2$ is a small diffusivity parameter related to the thickness of interfacial layers, $\lambda >0$ is a constant to be specified later, $\chi_0>0$ is a chemotaxis parameter, and $\delta\in (0,1)$ is a small parameter that determines the relative interaction strength between the cells and nutrient species.
The chemical potentials associated to the variables $u$ and $n$, respectively denoted by $\mu$ and $\sigma$, are given by the variational derivative of the Helmholtz free energy~\eqref{equ:free ener} with respect to each variable. 

The unknowns described above satisfy a system of coupled nonlinear PDEs which can be formulated as follows: Find {$(u,\mu,n):(0,T)\times\mathscr{D}\to\bb{R}^3$} satisfying
\begin{subequations}\label{equ:tum}
	\begin{alignat}{2}
		\label{equ:tum eq u}
		&\partial_t u
		=
		\Delta \mu + P(u)\cdot (\sigma-\mu)
		\qquad && \text{for $(t,\bff{x})\in(0,T)\times\mathscr{D}$,}
		\\
		\label{equ:tum eq mu}
		&\mu = -\epsilon^2 \Delta u + \lambda u - \chi_0 n + g'(u)
		\qquad && \text{for $(t,\bff{x})\in(0,T)\times\mathscr{D}$,}
		\\
		\label{equ:tum eq n}
		&\partial_t n = \Delta \sigma - P(u)\cdot (\sigma-\mu)
		\qquad && \text{for $(t,\bff{x})\in(0,T)\times\mathscr{D}$,}
		\\
		\label{equ:tum eq sig}
		&\sigma = \delta^{-1} n- \chi_0 u
		\qquad && \text{for $(t,\bff{x})\in(0,T)\times\mathscr{D}$,}
		\\
		\label{equ:tum init}
		&u(0,\bff{x})= u_0(\bff{x}),\quad n(0,\bff{x})= n_0(\bff{x})
		\qquad && \text{for } \bff{x}\in \mathscr{D},
		\\
		\label{equ:tum bound}
		&\partial_{\bff{\nu}} u= \partial_{\bff{\nu}} \mu = \partial_{\bff{\nu}} n = 0
		\qquad && \text{for } (t,\bff{x})\in (0,T) \times \partial \mathscr{D},
	\end{alignat}
\end{subequations}
where $\bff{\nu}$ is the outward pointing normal vector and $P$ is a non-negative proliferation function. The exact assumptions on the functions $P$ and $g$ will be given later. As described in~\cite{HawZeeOde12}, the model \eqref{equ:tum} has a natural four-constituent interpretation: a tumorous phase $u\approx 1$, a healthy cell phase $u\approx 0$, a nutrient-rich extracellular water phase $\sigma\approx 1$, and a nutrient-poor extracellular water phase $\sigma \approx 0$.

Relevant mathematical results on the problem \eqref{equ:tum} pertinent to the present work will be reviewed next. {Unique existence of weak and strong solutions, as well as the existence of global attractor are established in~\cite{GarYay20}; see also~\cite{FriGraRoc15} for a simplified model.} A characterisation of the singular limit of the problem to the corresponding sharp-interface model is discussed in~\cite{HilKamNguVan15}, with a rigorous proof given more recently in~\cite{RivRoc25}. An optimal control problem for \eqref{equ:tum} is studied in \cite{KadLouTra26}. On the numerical side, several nonlinear schemes to solve \eqref{equ:tum} are described in \cite{HawZeeOde12} based on the convex splitting approach. A linear second-order scheme is later proposed in~\cite{WuZwiZee14}, however no error analysis is performed for the schemes in any of these papers. A Fourier-spectral first-order scheme utilising the scalar auxiliary variable (SAV) approach is later proposed in~\cite{SheWuWenZha23}. While this scheme is linear, mass-preserving, and energy-dissipative, the method is applicable mainly to rectangular domains with periodic boundary conditions, and again no error analysis is given. 

More recently, numerical analysis of a first-order scheme to solve \eqref{equ:tum} employing a mass-lumping technique is performed in \cite{GarTra22}, where convergence along a subsequence to the exact solution is shown, albeit without a convergence rate. However, this scheme is nonlinear and does not necessarily guarantee mass conservation or energy dissipation at the discrete level. In contrast, some structure-preserving time-semidiscrete numerical schemes employing the SAV approach are proposed in~\cite{WanLinYan25} to solve the problem under periodic boundary conditions. For this setting, an error analysis of the time-semidiscrete case for the first-order scheme is performed, establishing an optimal \emph{temporal} order of convergence, although it is still limited to rectangular domains $\mathscr{D}\subset \bb{R}^2$ with periodic boundary conditions. To the best of our knowledge, a \emph{fully discrete}, \emph{linear}, and \emph{structure-preserving} numerical scheme for \eqref{equ:tum}, together with a rigorous error analysis, has not yet been developed in the literature.

In this work, we aim to fill this gap by proposing a fully discrete, linear, and structure-preserving numerical method for \eqref{equ:tum}, and providing its rigorous error analysis. Specifically, we employ an SAV formulation~\cite{SheXuYan19} together with a first-order semi-implicit time-stepping scheme. {In space, we use a mixed finite element method for the Cahn--Hilliard part and standard conforming finite elements for the reaction-diffusion part. 

This choice is motivated by the structure of the system. The Cahn--Hilliard equation is fourth order when written solely in terms of the phase-field variable, and the introduction of the chemical potential reduces it to a coupled system of second-order equations that can be discretised using standard $C^0$ conforming finite elements. The reaction-diffusion equation, on the other hand, is already second order and can therefore be discretised directly by standard conforming finite elements without introducing additional auxiliary variables.} The resulting mass-preserving linear scheme inherits a discrete energy dissipation law from the SAV formulation, but differs from the scheme introduced in~\cite{WanLinYan25} to facilitate the analysis in the fully discrete setting considered here. 

The convergence analysis carried out in the present paper, and hence the final results, differ substantially from those in~\cite{WanLinYan25}. While \cite{WanLinYan25} studies a time-semidiscrete scheme in a two-dimensional domain, we analyse a fully discrete finite element scheme on a three-dimensional convex domain. This distinction is significant, since the fully discrete solutions possess only $H^1$ spatial regularity, and the three-dimensional setting makes the control of nonlinear terms substantially more delicate. In particular, the $H^2$ regularity estimate that plays a central role in the time-semidiscrete analysis of~\cite{WanLinYan25} is no longer available in our setting.

To overcome this difficulty, we derive a number of delicate estimates for the nonlinear terms, relying on discrete Sobolev embeddings and discrete Gagliardo--Nirenberg inequalities. A key ingredient is the derivation of $L^\infty$ bounds for the finite element solution, together with estimates for its discrete time derivative in negative-order Sobolev spaces. Furthermore, we employ the Ritz projection of the exact solution and establish superconvergence estimates for the difference between the finite element solution and its Ritz projection. These tools allow us to obtain optimal error estimates not only in $H^1$, but also in $L^2$ and $L^\infty$.

The remainder of the paper is organised as follows. In Section~\ref{sec:prelim}, we introduce the notations and assumptions, present a formulation of the problem based on the SAV approach, and collect various preliminary results on finite element approximation. Section~\ref{sec:savfemeuler} describes the numerical scheme, establishes its mass conservation and energy stability properties, and provides the error estimates. In particular, the scheme is detailed in Algorithm~\ref{alg:sav}, and the main convergence result is stated in Theorem~\ref{the:main}. Section~\ref{sec:exp} reports several numerical experiments that validate the theoretical results. Finally, Section~\ref{sec:conclusion} concludes the paper with some final remarks.

\section{Preliminaries}\label{sec:prelim}

Commonly used notations, standing assumptions, as well as preliminary results on the finite element approximation and SAV formulation of \eqref{equ:tum} are collected in this section.

\subsection{Notations}
Let $\mathscr{D}\subset \bb{R}^d$ for $d=1,2,3$
be an open and bounded convex polytopal domain. Let $\Delta$ denote the Neumann Laplacian operator. The function space $L^p := L^p(\mathscr{D})$ denotes the usual Lebesgue space of $p$-th integrable functions and $W^{s,p} := W^{s,p}(\mathscr{D})$ denotes the usual Sobolev space of 
functions on $\mathscr{D}$. We
write $H^s := W^{s,2}$ and set $W^{0,p}=L^p$.

{If $X$ is a Banach space, the spaces $L^p(0,T;X)$ and $W^{s,p}(0,T;X)$ denote, respectively, the usual Lebesgue--Bochner and Sobolev--Bochner spaces of functions on $(0,T)$ taking values in $X$. We shall refer to these vector-valued function spaces collectively as Bochner spaces.} The space $\mathcal{C}([0,T];X)$ denotes the space of continuous functions on $[0,T]$ taking values in $X$. For simplicity, we write $W^{s,p}_T(W^{m,r}) := W^{s,p}(0,T; W^{m,r})$ and $L^p_T(L^q) := L^p(0,T; L^q)$. Throughout this paper, we denote the scalar product in a Hilbert space $H$ by $\inpro{\cdot}{\cdot}_H$ and its corresponding norm by $\|\cdot\|_H$. The scalar product in $L^2$ will simply be denoted by $\inpro{\cdot}{\cdot}$.

For any positive integer $s$, the dual space of $H^s$ with respect to the dual pairing which is the extension of the $L^2$ inner product will be denoted by $\widetilde{H}^{-s}$. This space is equipped with the norm
\begin{align*}
	\norm{v}{\widetilde{H}^{-s}}:= \sup_{0\neq \varphi\in H^s}\, \frac{\inpro{v}{\varphi}}{\norm{\varphi}{H^s}}.
\end{align*}

Finally, the constant $C$ in the estimate denotes a generic constant which takes different values at different occurrences. If
the dependence of $C$ on some variable, e.g.~$T$, is highlighted, we will write $C_T$ or
$C(T)$.

\subsection{Assumptions}

Assumptions on the model used in this paper are stated in this subsection.

\begin{assumption}\label{ass:assum P}
We assume that the proliferation function $P:\bb{R}\to \bb{R}_{\geq 0}$ lies in $W^{1,\infty}$ and is non-negative. In particular, $P$ is Lipschitz continuous, namely there exists a constant $C>0$ such that for any $v_1,v_2\in \bb{R}$,
\begin{align}\label{equ:P lipschitz}
	\abs{P(v_1)-P(v_2)} \leq C\abs{v_1-v_2}.
\end{align}
This implies a sublinear growth for $P$:
\begin{align}\label{equ:P sublin}
	\abs{P(v)} \leq C\left(1+\abs{v}\right).
\end{align}
\end{assumption}

\begin{assumption}\label{ass:assum f}
	{Let $g$ denote the Cahn--Hilliard free-energy density. We assume that
	$g\in C^3(\bb{R})$ and introduce the SAV splitting
	\[
	f(s):=g(s)-\frac{\lambda}{2}s^2,
	\]
	where $\lambda>0$ is the stabilisation parameter used in the scheme.} We assume
	that $f$ is bounded below and that there exists a constant $C>0$ such that, for
	$m\in\{0,1,2,3\}$ and all $s\in\bb{R}$,
	\begin{align}
		\label{equ:f der bound}
		\abs{f^{(m)}(s)}
		\leq
		C\left(1+\abs{s}^{4-m}\right).
	\end{align}
	Here, $f^{(m)}$ denotes the $m$-th derivative of $f$ and $f^{(0)}=f$.
\end{assumption}

\begin{remark}\label{rem:func f P}
	The commonly used double-well quartic potential and the linear proliferation function satisfy the above assumptions. For concreteness, as in~\cite{HawZeeOde12,WanLinYan25}, we take the original Cahn--Hilliard free-energy density to be
	\[
	g(u)=\kappa u^2(1-u)^2,
	\]
	where $\kappa>0$. {In the SAV formulation, we use the splitting
	$f(u)=\kappa u^2(1-u)^2-\frac{\lambda}{2}u^2$.
	Thus $\lambda>0$ is a stabilisation parameter introduced by the splitting, not a physical parameter of the model. The term $\frac{\lambda}{2}u^2$ cancels with the corresponding term in $f$ in the original free energy \eqref{equ:free ener}, but it appears in the numerical scheme (Algorithm~\ref{alg:sav}) and is useful for stability. Since the quartic term dominates the negative quadratic term, $f$ remains bounded from below for every fixed $\lambda>0$.
	
	In the analysis below, we impose the sufficient condition
	\begin{align}\label{equ:ass lambda begin}
	\lambda>\chi_0^2\delta .
	\end{align}
	This condition is not claimed to be sharp. It implies the weaker condition
	$4\lambda\ge\chi_0^2\delta$ used in the well-posedness proof (see Proposition~\ref{pro:well-posed}) and ensures the
	coercivity of the quadratic part of the modified energy (see Lemma~\ref{lem:stab u H1}). In computations,
	smaller choices of $\lambda$ may still be effective, but such choices are not
	covered by the present coercivity argument.} We do not investigate the optimal
	choice of $\lambda$ here. Similar stabilisation approaches are used in~\cite{HeLiuTan07,SheYan10} for phase-field models; see also~\cite{XuLiWu19} for a detailed study of choices of $\lambda$ and comparisons with other methods.
	
	The function $P$ is defined by
	\begin{equation}\label{equ:P u def}
		P(u)=
		\begin{cases}
			\delta p_0 u, &\text{if $u\geq 0$},
			\\
			0, &\text{otherwise},
		\end{cases}
	\end{equation}
	where $p_0$ is a positive constant denoting the proliferation growth parameter and $\delta\in(0,1)$.
\end{remark}

\begin{remark}
As a consequence of Assumption~\ref{ass:assum f}, by the mean value theorem we have for $m\in \{0,1,2\}$ any $x,y\in\bb{R}$,
\begin{align}
	\label{equ:f prime xy}
	\abs{f^{(m)}(x)-f^{(m)}(y)}
	&\leq
	\left(\sup_{t\in [0,1]} \abs{f^{(m+1)}\big((1-t)x+ty\big)}\right) \abs{x-y}
    \nonumber \\
	&\leq
	C\left(1+\abs{x}^{3-m}+\abs{y}^{3-m}\right) \abs{x-y},
\end{align}
where $C$ is a constant independent of $x$ or $y$.
\end{remark}

\subsection{SAV formulations}

We now present a reformulation of \eqref{equ:tum} that introduces an auxiliary variable, which will facilitate the development of numerical schemes subsequently.
{Recalling the introduction of the free-energy density function $f$ in \eqref{equ:free ener}, we define}
\begin{align}\label{equ:E1 u}
	\mathcal{E}_1[u] := \int_\mathscr{D} f(u) \,\dx.
\end{align}
Note that $\mathcal{E}_1[u]> -B$ for some {positive} constant $B$ by Assumption~\ref{ass:assum f}. We introduce a scalar variable $r(t):= \sqrt{\mathcal{E}_1[u]+B}>0$, and so $r(0)= \sqrt{\mathcal{E}_1[u_0]+B}$. 

{The constant $B$ is not a physical parameter; its role is only to ensure positivity in the SAV reformulation. In practice, $B$ is chosen with a moderate safety margin so that $\mathcal{E}_1[u_h^k]+B$ remains away from zero, where $u_h^k$ denotes the fully discrete approximation of $u$ to be introduced in Section~\ref{sec:savfemeuler}. In particular, it is not necessarily desirable to choose $B$ as small as possible, since this may enlarge the factor $\big(\mathcal{E}_1[u_h^k]+B\big)^{-\frac12}$ appearing in the scheme and may affect conditioning or amplify errors. On the other hand, excessively large values of $B$ are also unnecessary: it does not change the continuous SAV formulation, but makes the auxiliary variable artificially large and may introduce avoidable cancellation in the modified energy or influence the size of constants in the numerical error estimates. In practice, $B$ is therefore chosen as a moderate positive shift.}

With these functions, we can rewrite the problem~\eqref{equ:tum} following the SAV approach as follows:
\begin{subequations}\label{equ:tum sav}
	\begin{alignat}{2}
		\label{equ:tum eq u sav}
		&\partial_t u
		=
		\Delta \mu + P(u)\cdot (\sigma-\mu)
		\qquad && \text{for $(t,\bff{x})\in(0,T)\times\mathscr{D}$,}
		\\
		\label{equ:tum eq mu sav}
		&\mu = -\epsilon^2 \Delta u + \lambda u - \chi_0 n + \frac{r(t)}{\sqrt{\mathcal{E}_1[u]+B}}\, f'(u)
		\qquad && \text{for $(t,\bff{x})\in(0,T)\times\mathscr{D}$,}
		\\
		\label{equ:tum eq n sav}
		&\partial_t n = \Delta \sigma - P(u)\cdot (\sigma-\mu)
		\qquad && \text{for $(t,\bff{x})\in(0,T)\times\mathscr{D}$,}
		\\
		\label{equ:tum eq sig sav}
		&\sigma = \delta^{-1} n- \chi_0 u
		\qquad && \text{for $(t,\bff{x})\in(0,T)\times\mathscr{D}$,}
		\\
		\label{equ:tum eq r sav}
		&{\ddt r} = \frac{1}{2\sqrt{\mathcal{E}_1[u]+B}} \inpro{f'(u)}{\partial_t u}
		\qquad && \text{for {$t\in(0,T)$},}
		\\
		\label{equ:tum init sav}
		&u(0,\bff{x})= u_0(\bff{x}),\quad n(0,\bff{x})= n_0(\bff{x})
		\qquad && \text{for } \bff{x}\in \mathscr{D},
		\\
		\label{equ:tum bound sav}
		&\partial_{\bff{\nu}} u= \partial_{\bff{\nu}} \mu = \partial_{\bff{\nu}} n = 0
		\qquad && \text{for } (t,\bff{x})\in (0,T) \times \partial \mathscr{D},
	\end{alignat}
\end{subequations}
The system \eqref{equ:tum sav} consists of a fourth-order Cahn--Hilliard-type PDE with nonlinear source terms, coupled with a second-order nonlinear reaction-diffusion equation and an ODE.
Taking the inner product of \eqref{equ:tum eq u sav}, \eqref{equ:tum eq mu sav}, \eqref{equ:tum eq n sav}, and \eqref{equ:tum eq sig sav} with $\mu$, $-\partial_t u$, $\sigma$, and $-\partial_t n$, respectively, and multiplying \eqref{equ:tum eq r sav} by $2r(t)$, we obtain
\begin{align}\label{equ:ene law sav form}
	\ddt \mathcal{E}[u,n] \leq 0,	
\end{align}
which is the same energy dissipation law obeyed by~\eqref{equ:tum}.

A weak formulation of the system \eqref{equ:tum sav} can now be defined.

\begin{definition}\label{def:weak-solution}
Let $T>0$ and $(u_0,n_0)\in H^1\times L^2$ be given. A weak solution to the system \eqref{equ:tum sav} is $(u,\mu,n,r)$ with regularity
\begin{align*}
	u &\in H^1_T(L^2) \cap L^\infty_T(H^1) \cap L^2_T(H^2),
	\\
	\mu &\in L^2_T(H^1),
	\\
	n &\in H^1_T(L^2) \cap L^2_T(H^1),
	\\
	r &\in H^1(0,T),
\end{align*}
such that $\big(u(0),n(0)\big)=(u_0,n_0)$, and that for all $(\phi, \psi, \varphi) \in H^1\times H^1\times H^1$ and {a.e.} $t\in (0,T)$,
\begin{subequations}\label{equ:weak tum}
	\begin{alignat}{1}
		\label{equ:weak tum u}
		\inpro{\partial_t u(t)}{\phi}
		&=
		-\inpro{\nabla \mu(t)}{\nabla \phi}
		+
		\inpro{P(u(t))\cdot (\sigma(t) -\mu(t))}{\phi}
		\\
		\label{equ:weak tum mu}
		\inpro{\mu(t)}{\psi}
		&=
		\epsilon^2 \inpro{\nabla u(t)}{\nabla \psi}
		+
		\lambda \inpro{u(t)}{\psi}
		-
		\chi_0 \inpro{n(t)}{\psi}
		+
		\frac{r(t)}{\sqrt{\mathcal{E}_1[u(t)]+B}} \inpro{f'(u(t))}{\psi},
		\\
		\label{equ:weak tum n}
		\inpro{\partial_t n(t)}{\varphi}
		&=
		-\inpro{\nabla \sigma(t)}{\nabla \varphi}
		-
		\inpro{P(u(t))\cdot (\sigma(t) -\mu(t))}{\varphi},
		\\
		\label{equ:weak tum sigma}
		\sigma(t)
		&=
		\delta^{-1} n(t)
		-
		\chi_0 u(t), 
		\\
		\label{equ:weak tum r}
		\ddt r
		&=
		\frac{1}{2\sqrt{\mathcal{E}_1[u(t)]+B}}
		\inpro{f'(u(t))}{\partial_t u(t)}.
	\end{alignat}
\end{subequations}
\end{definition}
The assumed regularity on weak solution implies that $u\in \mathcal{C}([0,T];H^1)$ and $n\in \mathcal{C}([0,T];L^2)$ by interpolation theorems, thus the requirement $\big(u(0),n(0)\big)=(u_0,n_0)\in H^1\times L^2$ is meaningful. {The existence of a unique weak solution to the problem is established in~\cite{FriGraRoc15}.}

\subsection{Finite element approximation}

Let $\mathscr{D}\subset \mathbb{R}^d$, $d=1,2,3$, be a convex polytopal domain. Let $\{\mathcal{T}_h\}_{h>0}$ be a shape-regular and quasi-uniform family of simplicial triangulations of $\mathscr{D}$, consisting of intervals if $d=1$, triangles if $d=2$, and tetrahedra if $d=3$, with maximal mesh size $h$.
To discretise the problem \eqref{equ:tum sav}, we
introduce the conforming finite element space $V_h \subset H^1$ given by
\begin{equation}\label{equ:Vh}
	V_h := \{\phi_h \in \mathcal{C}(\overline{\mathscr{D}}): \phi_h|_K \in \mathcal{P}_q(K), \; \forall K \in \mathcal{T}_h\},
\end{equation}
where $\mathcal{P}_q(K)$ denotes the space of polynomials of degree less than or equal to $q$ on $K$, {$q\ge1$.}

As a consequence of the standard finite element approximation theory, for any $p\in [1,\infty]$, there exists a constant $C>0$, independent of $h$, such that for any $v\in W^{q+1,p}$,
\begin{align}\label{equ:fin approx}
	\inf_{\chi \in V_h} \left\{ \norm{v - \chi}{L^p} 
	+ 
	h \norm{\nabla v-\nabla \chi}{L^p} 
	\right\} 
	\leq 
	C h^{q+1} \norm{v}{W^{q+1,p}}.
\end{align}

In the analysis, we shall use several projection and interpolation operators. The existence of such operators and the properties that they possess will be described below (also see~\cite{BreSco08}).
Firstly, there exists an orthogonal projection operator $\Pi_h: L^2 \to V_h$ defined by
\begin{align}\label{equ:orth proj}
	\inpro{\Pi_h v-v}{\chi}=0,
	\quad
	\forall \chi \in V_h.
\end{align}
The projection operator $\Pi_h$ has the following boundedness and approximation properties~\cite{CroTho87, DouDupWah74}: for any $p\in [1,\infty]$, there exists a constant $C$ independent of $h$ such that
\begin{align}
	\label{equ:proj stable}
	\norm{\Pi_h v}{W^{s,p}}
	&\leq
	C \norm{v}{W^{s,p}}, \quad \forall v\in W^{s,p},\; s\in\{0,1\},
	\\
	\label{equ:proj approx}
	\norm{v- \Pi_h v}{L^p}
	+
	h \norm{\nabla( v-\Pi_h v)}{L^p}
	&\leq
	C h^{q+1} \norm{v}{W^{q+1,p}}, \quad \forall v\in W^{q+1,p}.
\end{align}
It is well-known that $\Pi_h$ is also bounded in the $\widetilde{H}^{-1}$ norm~\cite{CheMaoShe20}.

{
Next, we introduce the discrete Laplacian operator $\Delta_h:V_h\to V_h$ as the
	$L^2$-Riesz representative of the weak Laplacian restricted to $V_h$. More precisely, for each fixed $v_h\in V_h$, the map $\chi\mapsto -\inpro{\nabla v_h}{\nabla \chi}$
	is a linear functional on the finite-dimensional space $V_h$. Hence, by the
	finite-dimensional Riesz representation theorem with respect to the $L^2$ inner
	product, there exists a unique element $\Delta_h v_h\in V_h$ such that
	\begin{align}\label{equ:disc laplacian}
		\inpro{\Delta_h v_h}{\chi}
		=
		-\inpro{\nabla v_h}{\nabla \chi},
		\qquad
		\forall \chi\in V_h.
	\end{align}
	More concretely, if $\{\varphi_i\}_{i=1}^{N_h}$ is a basis of $V_h$, and
	\[
	v_h=\sum_{j=1}^{N_h} V_j\varphi_j,
	\qquad
	\Delta_h v_h=\sum_{j=1}^{N_h} D_j\varphi_j,
	\]
	then the coefficient vector $D=(D_j)_{j=1}^{N_h}$ is determined by $MD=-AV$, where
	$M_{ij}:=\inpro{\varphi_j}{\varphi_i}$ and $A_{ij}:=\inpro{\nabla\varphi_j}{\nabla\varphi_i}$, and $V=(V_j)_{j=1}^{N_h}$.
	Since the mass matrix $M$ is symmetric positive definite, $D$ is uniquely defined.
	Thus $\Delta_h v_h$ is a well-defined element of $V_h$, not merely a functional on
	$V_h$. This is the standard Galerkin realisation of the Neumann Laplacian on the
	finite element space; see~\cite[Chapter~1]{Tho06}. Since $\Delta_h:V_h\to V_h$ is well-defined, the operator $\Delta_h^2$ used in our analysis is simply the composition of $\Delta_h$ with itself. Consequently, for all $v_h, \chi\in V_h$,
	\begin{align}\label{equ:disc lap sq}
		\inpro{\Delta_h^2 v_h}{\chi}= -\inpro{\nabla\Delta_h v_h}{\nabla \chi}= \inpro{\Delta_h v_h}{\Delta_h \chi}.
	\end{align}
	The identities \eqref{equ:disc laplacian} and \eqref{equ:disc lap sq} are discrete analogues of Green's formula. They allow the weak Laplacian and its discrete iterate to be represented through $L^2$ inner products on $V_h$, which is the key mechanism used later in deriving higher-order stability and error estimates.
}

Taking $\chi=v_h$ in \eqref{equ:disc laplacian} gives
\begin{align}
	\label{equ:nab disc lap}
	\norm{\nabla v_h}{L^2}^2
	=
	-\inpro{\Delta_h v_h}{v_h}
	\leq
	\norm{\Delta_h v_h}{L^2}\norm{v_h}{L^2}.
\end{align}
Moreover, the discrete Laplacian satisfies the following estimates
(see~\cite[Lemma~2 and Lemma~9]{CheMaoShe20}):
\begin{align}
	\label{equ:disc lap H-1}
	\norm{\Delta_h v_h}{\widetilde{H}^{-1}}
	&\leq C \norm{\nabla v_h}{L^2},
	\\
	\label{equ:disc lap H-2}
	\norm{\Delta_h v_h}{\widetilde{H}^{-2}}
	&\leq C \norm{v_h}{L^2}.
\end{align}

We also define the Ritz projection $R_h:H^1\to V_h$ as follows: for each
$v\in H^1$, $R_h v\in V_h$ is the unique function satisfying
\begin{align}\label{equ:Ritz}
	\inpro{\nabla R_h v- \nabla v}{\nabla \chi}=0,
	\quad
	\forall \chi\in V_h, \quad \text{such that }\; \inpro{R_h v-v}{1}=0.
\end{align}
The last condition fixes the constant part and ensures uniqueness, since the
Neumann bilinear form $\inpro{\nabla \cdot}{\nabla \cdot}$ is coercive only
modulo constants. The Ritz projection is closely related to the discrete Laplacian introduced
above. Indeed, if $v$ is sufficiently smooth and satisfies the homogeneous
Neumann boundary condition, then $\Delta_h R_h v=\Pi_h\Delta v$, where $\Pi_h$ denotes the $L^2$ projection onto $V_h$ introduced earlier in~\eqref{equ:orth proj}; see~\cite[Chapter~1]{Tho06}.

The Ritz projection satisfies the following approximation properties
\cite{BreSco08,LeyLi21,LinThoWah91,RanSco82}: for any $p\in(1,\infty)$, there
exists a constant $C>0$, independent of $h$ and $v$, such that
\begin{align}\label{equ:Ritz ineq}
	\norm{v-R_h v}{W^{s,p}}
	&\leq
	C h^{q+1-s} \norm{v}{W^{q+1,p}},
	\qquad  s\in \{0,1\},
	\\
	\label{equ:Ritz infty}
	\norm{v-R_h v}{L^\infty}
	&\leq
	Ch^{q+1} \abs{\ln h}^{\frac12}
	\norm{v}{W^{q+1,\infty}} .
\end{align}

We also need the following approximation property of the Ritz projection $R_h$ in Sobolev spaces of negative indices~\cite{CheMaoShe20, Tho06}: for $0\leq s\leq q-1$ and $1\leq p\leq q+1$, there exists a constant $C$ independent of $h$ and $v$ such that
\begin{align}\label{equ:Ritz approx neg}
	{\norm{v-R_h v}{\widetilde{H}^{-s}}} 
	\leq Ch^{s+p} \norm{v}{H^p}.
\end{align}

In two-dimensional polygonal domains, the discrete Sobolev inequality~\cite{BreSco08} holds: there exists a constant $C$ independent of $h$ such that for all $v_h\in V_h$,
\begin{align}\label{equ:disc sob 2d}
    \norm{v_h}{L^\infty} \leq C \abs{\ln h}^{\frac12} \norm{v_h}{H^1}.
\end{align}

Finally, in convex polytopal domains with quasi-uniform triangulations, the following discrete Gagliardo--Nirenberg inequalities~\cite{LiuCheWanWis17} hold for any $v_h\in V_h$:
\begin{align}
	\label{equ:disc lapl L infty}
	\norm{v_h}{L^\infty}
	&\leq
	C \norm{\Delta_h v_h}{L^2}^{\frac{d}{2(6-d)}}
	\norm{v_h}{L^6}^{\frac{3(4-d)}{2(6-d)}}
	+
	C\norm{v_h}{L^6},
	\\
	\label{equ:nab vh L3}
	\norm{\nabla v_h}{L^3}
	&\leq
	C\norm{\Delta_h v_h}{L^2}^{\frac{d}{6}} \norm{\nabla v_h}{L^2}^{\frac{6-d}{6}}
	+
	C\norm{\nabla v_h}{L^2}.
\end{align}
Here, the constant $C$ depends only on $\mathscr{D}$ and $d$.

\section{SAV-FEM with first-order time discretisation}\label{sec:savfemeuler}

\subsection{The scheme and its well-posedness}
{For any function $v$, we write $v^k:= v(t_k)$, and define the discrete time derivative of $v^{k+1}$ by
\begin{align*}
	\dtt v^{k+1} := \frac{1}{\tau} (v^{k+1}-v^k),
\end{align*}
where $t_k:=k\tau \in [0,T]$ for $k=0,1,\ldots,N$ and $N := \lfloor T/\tau \rfloor$.

Let $(u_h^k, \mu_h^k, n_h^k, \sigma_h^k, r_h^k)\in [V_h]^4 \times \bb{R}$ be an approximation of $(u^k,\mu^k,n^k,\sigma^k,r^k)$ by a linear fully-discrete SAV-FEM with first-order time discretisation defined as follows.}
We start with $(u_h^0, n_h^0)=(R_h u_0, R_h n_0) \in V_h\times V_h$ and $r_h^0=\sqrt{\mathcal E_1[u_h^0]+B} \in \bb{R}$ for simplicity. 

Given $(u_h^k, \mu_h^k, n_h^k, r_h^k)\in V_h\times V_h\times V_h\times \bb{R}$, we find $(u_h^{k+1}, \mu_h^{k+1}, n_h^{k+1}, r_h^{k+1})$ satisfying
\begin{subequations}\label{equ:fem euler}
	\begin{alignat}{2}
		\label{equ:fem euler u}
		\inpro{\dtt u_h^{k+1}}{\phi}
		&=
		-\inpro{\nabla \mu_h^{k+1}}{\nabla \phi}
		+
		\inpro{P(u_h^k)\cdot (\sigma_h^{k+1}-\mu_h^{k+1})}{\phi}, \quad &&\forall \phi \in V_h,
		\\
		\inpro{\mu_h^{k+1}}{\psi}
		&=
		\epsilon^2 \inpro{\nabla u_h^{k+1}}{\nabla \psi}
		+
		\lambda \inpro{u_h^{k+1}}{\psi}
		-
		\chi_0 \inpro{n_h^{k+1}}{\psi}
		\nonumber\\
		\label{equ:fem euler mu}
		&\qquad
		+
		\frac{r_h^{k+1}}{\sqrt{\mathcal{E}_1[u_h^k]+B}} \inpro{f'(u_h^k)}{\psi}, \quad &&\forall \psi\in V_h,
		\\
		\label{equ:fem euler n}
		\inpro{\dtt n_h^{k+1}}{\varphi}
		&=
		-\inpro{\nabla \sigma_h^{k+1}}{\nabla \varphi}
		-
		\inpro{P(u_h^k)\cdot (\sigma_h^{k+1}-\mu_h^{k+1})}{\varphi}, \quad &&\forall \varphi \in V_h,
		\\
		\label{equ:fem euler sigma}
		\sigma_h^{k+1}
		&=
		\delta^{-1} n_h^{k+1}
		-
		\chi_0 u_h^k,
		\\
		\label{equ:fem euler r}
		\dtt r_h^{k+1}
		&=
		\frac{1}{2\sqrt{\mathcal{E}_1[u_h^k]+B}}
		\inpro{f'(u_h^k)}{\dtt u_h^{k+1}}.
	\end{alignat}
\end{subequations}
The energy functional $\mathcal{E}_1$ is defined by \eqref{equ:E1 u}, while the functions $P$ and $f$ are given in Remark~\ref{rem:func f P}. 

For ease of presentation, we will subsequently write:	
\begin{align}\label{equ:phk bhk}
	p_h^k:=P(u_h^k),
	\qquad
	\beta_h^k := \mathcal E_1[u_h^k]+B,
	\qquad
	b_h^k:=
	\frac{f'(u_h^k)}{\sqrt{\beta_h^k}},
\end{align}
as well as
\begin{align}\label{equ:beta k bk}
	p^k :=P(u^k),
	\qquad
	\beta^k := \mathcal E_1[u^k]+B = (r^k)^2,
	\qquad
	b^k:=
	\frac{f'(u^k)}{\sqrt{\beta^k}}.
\end{align}
The numerical scheme given by \eqref{equ:fem euler} is made precise in the following algorithm.

\begin{algorithm}[SAV-FEM with first-order time discretisation]\label{alg:sav}
	Let $h>0$, $\tau>0$ be given. 
	\\
	\textbf{Input}: Set $(u_h^0,n_h^0)=(R_hu_0,R_hn_0)\in V_h\times V_h$, and $r_h^0=\sqrt{\mathcal E_1[u_h^0]+B} \in \bb{R}$. Choose a stabilisation parameter $\lambda>0$, for instance according to \eqref{equ:ass lambda begin}.
	\\
	\textbf{For} $k=0,1,\ldots,N-1$, \textbf{do}:
	\begin{enumerate}
		\item Using \eqref{equ:fem euler r}, write
		\[
		r_h^{k+1}
		=
		r_h^k
		+
		\frac12\inpro{b_h^k}{u_h^{k+1}-u_h^k}.
		\]
		
		\item Find $(u_h^{k+1},\mu_h^{k+1},n_h^{k+1})\in V_h\times V_h\times V_h$ such that, for all
		$(\phi,\psi,\varphi)\in V_h\times V_h\times V_h$,
		\begin{align}
			\inpro{\dtt u_h^{k+1}}{\phi}
			&=
			-\inpro{\nabla\mu_h^{k+1}}{\nabla\phi}
			+
			\inpro{
				p_h^k
				\left(
				\delta^{-1}n_h^{k+1}
				-\chi_0u_h^k
				-\mu_h^{k+1}
				\right)
			}{\phi},
			\label{equ:alg-elim-u}
			\\
			\inpro{\mu_h^{k+1}}{\psi}
			&=
			\epsilon^2\inpro{\nabla u_h^{k+1}}{\nabla\psi}
			+
			\lambda\inpro{u_h^{k+1}}{\psi}
			-
			\chi_0\inpro{n_h^{k+1}}{\psi}
			\nonumber\\
			&\quad
			+
			\left(
			r_h^k-\frac12\inpro{b_h^k}{u_h^k}
			\right)
			\inpro{b_h^k}{\psi}
			+
			\frac12
			\inpro{b_h^k}{u_h^{k+1}}
			\inpro{b_h^k}{\psi},
			\label{equ:alg-elim-mu}
			\\
			\inpro{\dtt n_h^{k+1}}{\varphi}
			&=
			-\inpro{
				\nabla\left(
				\delta^{-1}n_h^{k+1}
				-\chi_0 u_h^k
				\right)
			}{\nabla\varphi}
			-
			\inpro{
				p_h^k
				\left(
				\delta^{-1}n_h^{k+1}
				-\chi_0u_h^k
				-\mu_h^{k+1}
				\right)
			}{\varphi}.
			\label{equ:alg-elim-n}
		\end{align}
		
		\item Recover
		\[
		\sigma_h^{k+1}
		=
		\delta^{-1}n_h^{k+1}
		-
		\chi_0u_h^k,
		\]
		and
		\[
		r_h^{k+1}
		=
		r_h^k
		+
		\frac12\inpro{b_h^k}{u_h^{k+1}-u_h^k}.
		\]
	\end{enumerate}
	
	\textbf{Output}: the sequence $\{u_h^k,\mu_h^k,n_h^k,\sigma_h^k,r_h^k\}_{k=1}^{N}$.
\end{algorithm}

We first establish the well-posedness of Algorithm~\ref{alg:sav}. 
The condition \eqref{equ:ass lambda} below does not impose a genuine restriction on the model parameters, since $\lambda>0$ is a stabilisation parameter introduced in the SAV splitting and may be chosen sufficiently large (see Remark~\ref{rem:func f P}).

{
	\begin{proposition}\label{pro:well-posed}
		Let $\tau>0$ and suppose that $u_h^k,n_h^k\in V_h$, and $r_h^k\in\mathbb{R}$ are given. 
		Assume that
		\begin{align}\label{equ:ass lambda}
			4\lambda \ge \chi_0^2\delta .
		\end{align}
		Then the scheme given by Algorithm~\ref{alg:sav} admits a unique solution
		\[
		(u_h^{k+1},\mu_h^{k+1},n_h^{k+1},\sigma_h^{k+1},r_h^{k+1})
		\in V_h\times V_h\times V_h\times V_h\times \mathbb R .
		\]
	\end{proposition}
	
	\begin{proof}
		At the time level $k$, all coefficients evaluated at $u_h^k$ are known. Recall $p_h^k$ and $b_h^k$ defined in \eqref{equ:phk bhk}.
		We have $p_h^k\ge 0$ in $\mathscr D$ by Assumption~\ref{ass:assum P}.
		After choosing a basis of $V_h$, the scheme is a square finite-dimensional linear system. Uniqueness of solution implies existence, thus it is enough to prove uniqueness.
		
		Suppose that $(u_{h,1}^{k+1},\mu_{h,1}^{k+1},n_{h,1}^{k+1},\sigma_{h,1}^{k+1},r_{h,1}^{k+1})$
		and
		$(u_{h,2}^{k+1},\mu_{h,2}^{k+1},n_{h,2}^{k+1},\sigma_{h,2}^{k+1},r_{h,2}^{k+1})$
		are two solutions of \eqref{equ:fem euler} with the same data $(u_{h}^{k},\mu_{h}^{k},n_{h}^{k},\sigma_{h}^{k},r_{h}^{k})$ at time level $k$. Set
		\[
		U:=u_{h,1}^{k+1}-u_{h,2}^{k+1},
		\qquad
		M:=\mu_{h,1}^{k+1}-\mu_{h,2}^{k+1},
		\]
		\[
		N:=n_{h,1}^{k+1}-n_{h,2}^{k+1},
		\qquad
		S:=\sigma_{h,1}^{k+1}-\sigma_{h,2}^{k+1},
		\]
		and
		\[
		R:=r_{h,1}^{k+1}-r_{h,2}^{k+1}.
		\]
		Subtracting the two systems gives
		\begin{subequations}\label{equ:lin sys}
			\begin{alignat}{2}
				\label{eq:diff-u-wellposed}
				&\inpro{U}{\phi}
				+
				\tau \inpro{\nabla M}{\nabla \phi}
				-
				\tau\inpro{p_h^k(S-M)}{\phi}
				=0,
				\qquad
				&&\forall \phi\in V_h,
				\\
				\label{eq:diff-mu-wellposed}
				&\inpro{M}{\psi}
				=
				\epsilon^2\inpro{\nabla U}{\nabla \psi}
				+
				\lambda\inpro{U}{\psi}
				-
				\chi_0\inpro{N}{\psi}
				+
				R\inpro{b_h^k}{\psi},
				\qquad
				&&\forall \psi\in V_h,
				\\
				\label{eq:diff-n-wellposed}
				&\inpro{N}{\varphi}
				+
				\tau\inpro{\nabla S}{\nabla \varphi}
				+
				\tau\inpro{p_h^k(S-M)}{\varphi}
				=0,
				\qquad
				&&\forall \varphi\in V_h,
			\end{alignat}
		\end{subequations}
		together with
		\begin{equation}
			\label{eq:diff-sigma-r-wellposed}
			S=\delta^{-1}N,
			\qquad
			R=\frac12\inpro{b_h^k}{U}.
		\end{equation}
		
		Taking $\phi=M$ in \eqref{eq:diff-u-wellposed} and $\varphi=S$ in \eqref{eq:diff-n-wellposed}, and adding the resulting identities, we obtain
		\begin{align}
			\label{eq:first-energy-wellposed}
			0
			&=
			\inpro{U}{M}
			+
			\inpro{N}{S}
			+
			\tau\norm{\nabla M}{L^2}^2
			+
			\tau\norm{\nabla S}{L^2}^2
			- 
			\tau\inpro{p_h^k(S-M)}{M}
			+
			\tau\inpro{p_h^k(S-M)}{S}
			\nonumber\\
			&=
			\inpro{U}{M}
			+
			\inpro{N}{S}
			+
			\tau\norm{\nabla M}{L^2}^2
			+
			\tau\norm{\nabla S}{L^2}^2
			+
			\tau\norm{(p_h^k)^{\frac12} (M-S)}{L^2}^2.
		\end{align}
		Next, choosing $\psi=U$ in \eqref{eq:diff-mu-wellposed} gives
		\[
		\inpro{U}{M}
		=
		\epsilon^2\|\nabla U\|_{L^2}^2
		+
		\lambda\|U\|_{L^2}^2
		-
		\chi_0\inpro{N}{U}
		+
		R\inpro{b_h^k}{U}.
		\]
		By \eqref{eq:diff-sigma-r-wellposed}, we have $R\inpro{b_h^k}{U}=2R^2$ and
		$\inpro{N}{S}
		=
		\delta^{-1}\|N\|_{L^2}^2$.
		Substituting these identities into \eqref{eq:first-energy-wellposed} yields
		\begin{align}
			\label{eq:wellposed-coercive-identity}
			0
			&=
			\epsilon^2 \norm{\nabla U}{L^2}^2
			+
			\lambda \norm{U}{L^2}^2
			-
			\chi_0\inpro{N}{U}
			+
			\delta^{-1} \norm{N}{L^2}^2
			+
			2R^2
			\nonumber\\
			&\qquad
			+
			\tau \norm{\nabla M}{L^2}^2
			+
			\tau \norm{\nabla S}{L^2}^2
			+
			\tau \norm{(p_h^k)^{\frac12} (M-S)}{L^2}^2.
		\end{align}
		Now, by assumption \eqref{equ:ass lambda}, the quadratic form
		$(a,b)\mapsto \lambda a^2-\chi_0 ab+\delta^{-1}b^2$
		is nonnegative on $\mathbb R^2$. Thus,
		\[
		\lambda \norm{U}{L^2}^2
		-
		\chi_0\inpro{N}{U}
		+
		\delta^{-1} \norm{N}{L^2}^2
		\ge0,
		\]
		while other terms on the right-hand side of
		\eqref{eq:wellposed-coercive-identity} are non-negative. This implies
		\begin{align}\label{equ:UMS}
			\nabla U=0,
			\qquad
			\nabla M=0,
			\qquad
			\nabla S=0,
			\qquad
			(p_h^k)^{\frac12}(M-S)=0,
			\qquad
			R=0.
		\end{align}
		Since $1\in V_h$, we may choose $\phi=1$ in \eqref{eq:diff-u-wellposed} to obtain
		\[
		\inpro{U}{1}
		+
		\tau\inpro{p_h^k(M-S)}{1}
		=0.
		\]
		Because $(p_h^k)^{\frac12}(M-S)=0$ a.e. and $p_h^k\ge0$, we have $p_h^k(M-S)=0$ a.e. in $\mathscr D$. Thus, $\inpro{U}{1}=0$.
		Together with $\nabla U=0$, this implies $U=0$. By \eqref{eq:diff-sigma-r-wellposed}, we also have $R=0$.
		
		Similarly, choosing $\varphi=1$ in \eqref{eq:diff-n-wellposed} yields
		\[
		\inpro{N}{1}
		+
		\tau\inpro{p_h^k(S-M)}{1}
		=0.
		\]
		By the same argument as before, $p_h^k(S-M)=0$ a.e. in $\mathscr{D}$, and therefore $\inpro{N}{1}=0$.
		Since $S=\delta^{-1}N$ and $\nabla S=0$ by \eqref{eq:diff-sigma-r-wellposed} and \eqref{equ:UMS}, we also have $\nabla N=0$. Consequently, $N=0$ and $S=\delta^{-1}N=0$. 
		Finally, substituting $U=N=R=0$ into \eqref{eq:diff-mu-wellposed} gives $M=0$.
		
		We infer that there exists a unique solution to the system \eqref{equ:lin sys} and \eqref{eq:diff-sigma-r-wellposed}. The proof is now complete.
	\end{proof}

\subsection{Implementation of the SAV term}

We now describe in more detail the implementation of the SAV term in Algorithm~\ref{alg:sav}.
Although Algorithm~\ref{alg:sav} contains the scalar variable $r_h^{k+1}$ implicitly, this does \emph{not} lead to a nonlinear or genuinely dense system. We explain below how $r_h^{k+1}$ can be eliminated so that the resulting linear system consists of the usual sparse finite element block system plus a rank-one SAV correction, which can be treated efficiently.

Recall that $p_h^k$ and $b_h^k$ were defined in \eqref{equ:phk bhk}.
Then \eqref{equ:fem euler r} can be written as
\[
r_h^{k+1}
=
r_h^k
+
\frac12\inpro{b_h^k}{u_h^{k+1}-u_h^k}.
\]
Substituting this expression into the chemical-potential equation
\eqref{equ:fem euler mu} yields equation \eqref{equ:alg-elim-mu}.
All terms in \eqref{equ:alg-elim-mu} are linear in the unknowns
$(u_h^{k+1},\mu_h^{k+1},n_h^{k+1})$. The only nonlocal term is
\begin{align}\label{equ:nonloc}
\frac12
\inpro{b_h^k}{u_h^{k+1}}
\inpro{b_h^k}{\psi},
\end{align}
which is a product of two global $L^2$ inner products.

We describe the resulting algebraic system. Let
$\{\varphi_i\}_{i=1}^{N_h}$ be a basis of $V_h$, and write
\[
u_h^{k+1}=\sum_{j=1}^{N_h}U_j^{k+1}\varphi_j,
\qquad
\mu_h^{k+1}=\sum_{j=1}^{N_h}\mu_j^{k+1}\varphi_j,
\qquad
n_h^{k+1}=\sum_{j=1}^{N_h}N_j^{k+1}\varphi_j.
\]
Let $\mathbf U^{k+1}$, $\boldsymbol{\mu}^{k+1}$, and $\mathbf N^{k+1}$ denote the corresponding coefficient vectors. Define the mass, stiffness, and weighted mass matrices by
\[
\mathsf M_{ij}:=\inpro{\varphi_j}{\varphi_i},
\qquad
\mathsf S_{ij}:=\inpro{\nabla\varphi_j}{\nabla\varphi_i},
\qquad
\mathsf Q^k_{ij}:=\inpro{p_h^k \varphi_j}{\varphi_i}.
\]
Also define
\[
\gamma_i^k:=\inpro{b_h^k}{\varphi_i},
\qquad
\boldsymbol{\gamma}^k:=(\gamma_i^k)_{i=1}^{N_h}.
\]
After eliminating $\sigma_h^{k+1}$ through $\sigma_h^{k+1}
=
\delta^{-1}n_h^{k+1}-\chi_0u_h^k$, the sparse part of the linear system (not including the term \eqref{equ:nonloc}) for
\[
\mathbf X^{k+1}:=
\begin{pmatrix}
	\mathbf U^{k+1}\\
	\boldsymbol{\mu}^{k+1}\\
	\mathbf N^{k+1}
\end{pmatrix}
\]
is $\mathsf A_0^k\mathbf X^{k+1}
=
\mathbf F^k$,
where
\[
\mathsf A_0^k
=
\begin{pmatrix}
	\tau^{-1}\mathsf M
	&
	\mathsf S+\mathsf Q^k
	&
	-\delta^{-1}\mathsf Q^k
	\\[2mm]
	-(\epsilon^2\mathsf S+\lambda\mathsf M)
	&
	\mathsf M
	&
	\chi_0\mathsf M
	\\[2mm]
	0
	&
	-\mathsf Q^k
	&
	\tau^{-1}\mathsf M+\delta^{-1}(\mathsf S+\mathsf Q^k)
\end{pmatrix}.
\]
The right-hand side $\mathbf F^k$ contains only known quantities from time level $k$:
\[
\mathbf F^k
=
\begin{pmatrix}
	\tau^{-1}\mathsf M\mathbf U^k-\chi_0\mathsf Q^k\mathbf U^k
	\\[1mm]
	\left(
	r_h^k-\frac12(\boldsymbol{\gamma}^k)^\top \mathbf U^k
	\right)\boldsymbol{\gamma}^k
	\\[1mm]
	\tau^{-1}\mathsf M\mathbf N^k
	+
	\chi_0(\mathsf S+\mathsf Q^k)\mathbf U^k
\end{pmatrix}.
\]
The \emph{sparse block matrix} $\mathsf A_0^k$ is nonsingular by the same homogeneous
energy argument used in Proposition~\ref{pro:well-posed}, with the SAV rank-one
term omitted. Equivalently, the argument in Proposition~\ref{pro:well-posed}
applies with the scalar difference $R$ set to zero.

The term \eqref{equ:nonloc} is the only part of \eqref{equ:alg-elim-mu} that
contains the unknown $u_h^{k+1}$ through the SAV update. When this term is moved
to the left-hand side of the chemical-potential equation, it gives the rank-one
contribution
\[
-\frac12
\begin{pmatrix}
	0\\
	\boldsymbol{\gamma}^k\\
	0
\end{pmatrix}
\begin{pmatrix}
	(\boldsymbol{\gamma}^k)^\top & 0 & 0
\end{pmatrix}
\mathbf X^{k+1}.
\]
Hence the full algebraic system can be written as
\begin{equation}\label{equ:rank-one-sav-system}
	\left(
	\mathsf A_0^k
	-
	\frac12
	\mathbf v_{\rm sav}^k
	(\mathbf w_{\rm sav}^k)^\top
	\right)
	\mathbf X^{k+1}
	=
	\mathbf F^k,
\end{equation}
where
\[
\mathbf v_{\rm sav}^k
:=
\begin{pmatrix}
	0\\
	\boldsymbol{\gamma}^k\\
	0
\end{pmatrix},
\qquad
\mathbf w_{\rm sav}^k
:=
\begin{pmatrix}
	\boldsymbol{\gamma}^k\\
	0\\
	0
\end{pmatrix}.
\]
Thus the SAV term is nonlocal at the variational level, but it appears
algebraically only as a rank-one correction to the sparse finite element block
matrix $\mathsf A_0^k$. Therefore, this term need not be assembled as a dense
matrix.

The rank-one correction can be treated efficiently by the Sherman--Morrison
formula~\cite{SheMor50}. First solve the two sparse systems
\[
\mathsf A_0^k\mathbf Y^k=\mathbf F^k,
\qquad
\mathsf A_0^k\mathbf Z^k=\mathbf v_{\rm sav}^k.
\]
Then set $\alpha_{\rm sav}^k
:=
1-\frac12(\mathbf w_{\rm sav}^k)^\top\mathbf Z^k$.
The scalar $\alpha^k$ is nonzero. Indeed, if $\alpha^k=0$, then the rank-one
perturbed matrix in \eqref{equ:rank-one-sav-system} would be singular by the
rank-one determinant identity. This would contradict the unique solvability of
Algorithm~\ref{alg:sav} established in Proposition~\ref{pro:well-posed}.
Thus the following quantity is well-defined:
\[
\vartheta_{\rm sav}^k
:=
\frac{(\mathbf w_{\rm sav}^k)^\top\mathbf Y^k}{\alpha^k}.
\]
The solution of \eqref{equ:rank-one-sav-system} is then
\[
\mathbf X^{k+1}
=
\mathbf Y^k
+
\frac12\mathbf Z^k \vartheta_{\rm sav}^k.
\]
Finally, after extracting $\mathbf U^{k+1}$ from $\mathbf X^{k+1}$, the scalar
auxiliary variable is recovered from
\[
r_h^{k+1}
=
r_h^k
+
\frac12\inpro{b_h^k}{u_h^{k+1}-u_h^k}.
\]
Thus, at each time step, after assembling the sparse block matrix $\mathsf A_0^k$, the computation only requires two sparse linear solves with the same matrix $\mathsf A_0^k$, together with vector inner products and scalar operations. Since $\mathsf A_0^k$ depends on $u_h^k$ through $\mathsf Q^k$, the matrix is generally updated from one time step to the next, but the two solves within each time step share the same coefficient matrix.
This is a low-rank implementation principle used in standard SAV schemes, adapted here to the coupled finite element system.
}

\subsection{Stability and error analysis}

In the remainder of this section, we carry out the stability and error analysis of Algorithm \ref{alg:sav}. The weak solution regularity stated in Definition~\ref{def:weak-solution} is sufficient for the variational formulation of the continuous problem, but the optimal-order error estimates below require additional smoothness of the exact solution. We therefore assume that the exact solution satisfies
\begin{equation}\label{equ:reg u euler}
	u,\mu,n \in L^\infty_T(W^{q+1,\infty}) \cap W^{1,\infty}_T(H^{q+1}) \cap W^{2,\infty}_T(H^2),
\end{equation}
where $q\geq 1$ is the polynomial degree of the finite element space.

{This regularity assumption \eqref{equ:reg u euler} is used only as a sufficient hypothesis to carry out the fully discrete error analysis in $\mathscr{D}$. We do not claim that the above regularity assumption is optimal for this purpose. Such assumptions are standard in optimal-order finite element error estimates for nonlinear parabolic systems, and are compatible with the available strong-solution theory for related diffuse-interface tumour-growth models~\cite{FriGraRoc15,GarYay20}. Under smoother data, compatible boundary conditions, and sufficiently regular domains, higher regularity may be obtained by combining parabolic regularity with elliptic bootstrapping. Since the present paper focuses on the numerical analysis, we do not pursue minimal continuous regularity assumptions here.}

Let us define the modified energy functional $\widetilde{\mathcal{E}}[u_h^k,n_h^k,r_h^k]$ of the finite element solution as
\begin{align}\label{equ:ener modified}
	\widetilde{\mathcal{E}}[u_h^k,n_h^k,r_h^k] 
	:=
	\int_\mathscr{D} \left( \frac{\epsilon^2}{2} \abs{\nabla u_h^k}^2 + \frac{\lambda}{2} \abs{u_h^k}^2 + \frac{1}{2\delta} \abs{n_h^k}^2 - \chi_0 u_h^k n_h^k\right) \dx
    + |r_h^k|^2 -B.
\end{align}
This modified energy functional is an approximation of the original energy $\mathcal{E}[u^k,n^k]$.
We first prove the mass conservation and unconditional energy stability properties of scheme~\eqref{equ:fem euler} in the following proposition.

\begin{proposition}
Let $(u_h^k, \mu_h^k, n_h^k, \sigma_h^k, r_h^k)$ be defined by~\eqref{equ:fem euler}. The following mass conservation law holds, namely for $k=1,2,\ldots,\lfloor T/\tau \rfloor$,
\begin{align}\label{equ:mass cons}
	\int_{\mathscr{D}} (u_h^k+n_h^k) \,\dx = \int_{\mathscr{D}} (u_h^0+n_h^0) \,\dx.
\end{align}
Furthermore, energy dissipation law holds unconditionally for scheme~\eqref{equ:fem euler}:
	\begin{align}\label{equ:disc ener ineq}
		\widetilde{\mathcal{E}}[u_h^{k+1}, n_h^{k+1}, r_h^{k+1}] \leq \widetilde{\mathcal{E}}[u_h^k, n_h^k, r_h^k],
	\end{align}
where $\widetilde{\mathcal{E}}[u_h^k, n_h^k, r_h^k]$ is defined in \eqref{equ:ener modified}.
\end{proposition}

\begin{proof}
To show \eqref{equ:mass cons}, take $\phi=\varphi=1$ {in \eqref{equ:fem euler u} and \eqref{equ:fem euler n}}, then add the resulting equations.	

Next, we prove \eqref{equ:disc ener ineq}. We set $\phi=\mu_h^{k+1}$, $\psi=-\dtt u_h^{k+1}$, and $\varphi=\sigma_h^{k+1}$. We also multiply \eqref{equ:fem euler sigma} and \eqref{equ:fem euler r} with $-\dtt n_h^{k+1}$ and $2r_h^{k+1}$, respectively, then add the resulting equations. Noting the identity
\begin{equation}\label{equ:ab ide}
	2a(a-b)=\abs{a}^2-\abs{b}^2 + \abs{a-b}^2, \quad \forall a,b\in \bb{R},
\end{equation}
we obtain
\begin{align}\label{equ:Ek1 min Ek}
	&\widetilde{\mathcal{E}}[u_h^{k+1}, n_h^{k+1}, r_h^{k+1}] - \widetilde{\mathcal{E}}[u_h^k, n_h^k, r_h^k]
	\nonumber\\
	&\quad
	+
	\frac{\epsilon^2}{2} \norm{\nabla u_h^{k+1}-\nabla u_h^k}{L^2}^2
	+
	\frac{\lambda}{2} \norm{u_h^{k+1}- u_h^k}{L^2}^2
	+
	\frac{1}{2\delta} \norm{n_h^{k+1}-n_h^k}{L^2}^2
	+
	\abs{r_h^{k+1}-r_h^k}^2
	\nonumber\\
	&\quad
	+
	\tau \norm{\nabla \mu_h^{k+1}}{L^2}^2
	+
	\tau \norm{\nabla \sigma_h^{k+1}}{L^2}^2
	+
	{\tau} \norm{\sqrt{P(u_h^k)} \left(\sigma_h^{k+1}-\mu_h^{k+1}\right)}{L^2}^2
	= 0,
\end{align}
where $P$ is a non-negative function by Assumption~\ref{ass:assum P}.
This implies \eqref{equ:disc ener ineq}.
\end{proof}

We note that \eqref{equ:disc ener ineq} does not directly yield a bound on $\norm{u_h^k}{H^1}$ or $\norm{n_h^k}{L^2}$ since the term $\chi_0 n_h^k u_h^k$ {in \eqref{equ:ener modified}} does not have a definite sign. The following lemma takes care of this issue.

\begin{lemma}\label{lem:stab u H1}
Assume that
\begin{align}\label{equ:ass lambda ener}
	\lambda> \chi_0^2 \delta.
\end{align}
Then, for $k=1,2,\ldots, \lfloor T/\tau \rfloor$,
\begin{align}\label{equ:stab u H1 n L2}
	\norm{u_h^k}{H^1}^2 + \norm{n_h^k}{L^2}^2 + \norm{\sigma_h^k}{L^2}^2 + |r_h^k|^2 \leq C.
\end{align}
Furthermore,
\begin{align}\label{equ:stab sum mu sigma H1}
	\tau \sum_{j=1}^k \norm{\nabla \mu_h^j}{L^2}^2
	+
	\tau \sum_{j=1}^k \norm{\nabla \sigma_h^j}{L^2}^2
	\leq C.
\end{align}
Here, $C$ depends on $T$ (as well as $\epsilon, \lambda, \delta, \chi_0$, and the initial data), but is independent of $k$, $h$, and $\tau$.
\end{lemma}

\begin{proof}
{First, we note that by Young's inequality, for any $\eta>0$,
\[
\chi_0\abs{\inpro{u_h^k}{n_h^k}}
\le
\frac{\eta}{2}\norm{u_h^k}{L^2}^2
+
\frac{\chi_0^2}{2\eta}\norm{n_h^k}{L^2}^2 .
\]
Choosing $\eta$ such that $\chi_0^2\delta<\eta<\lambda$, we obtain
\begin{align*}
	&\frac{\lambda}{2}\norm{u_h^k}{L^2}^2
	+
	\frac{1}{2\delta}\norm{n_h^k}{L^2}^2
	-
	\chi_0\inpro{u_h^k}{n_h^k}
	\ge
	\frac{\lambda-\eta}{2}\norm{u_h^k}{L^2}^2
	+
	\frac12
	\left(
	\frac{1}{\delta}
	-
	\frac{\chi_0^2}{\eta}
	\right)
	\norm{n_h^k}{L^2}^2 .
\end{align*}
Consequently, with this fixed $\eta$, there exist constants $c_u,c_n>0$, depending only on $\lambda,\delta,\chi_0$, such that
\begin{align}\label{equ:eta ineq}
\frac{\lambda}{2}\norm{u_h^k}{L^2}^2
+
\frac{1}{2\delta}\norm{n_h^k}{L^2}^2
-
\chi_0\inpro{u_h^k}{n_h^k}
\ge
c_u\norm{u_h^k}{L^2}^2
+
c_n\norm{n_h^k}{L^2}^2 .
\end{align}
Now, from \eqref{equ:disc ener ineq}, we have for all $k\leq \lfloor T/\tau \rfloor$,
\begin{align*}
	\widetilde{\mathcal{E}}[u_h^k, n_h^k, r_h^k] \leq \widetilde{\mathcal{E}}[u_h^0, n_h^0, r_h^0],
\end{align*}
which implies
\[
\begin{aligned}
	&\frac{\epsilon^2}{2}\norm{\nabla u_h^k}{L^2}^2
	+
	c_u\norm{u_h^k}{L^2}^2
	+
	c_n\norm{n_h^k}{L^2}^2
	+
	|r_h^k|^2
	\\
	&\leq
	\frac{\epsilon^2}{2} \norm{\nabla u_h^k}{L^2}^2
	+
	\frac{\lambda}{2} \norm{u_h^k}{L^2}^2
	+
	\frac{1}{2\delta} \norm{n_h^k}{L^2}^2
	-
	\chi_0\inpro{u_h^k}{n_h^k}
	+
	|r_h^k|^2
	\\
	&\le
	\widetilde{\mathcal E}[u_h^0,n_h^0,r_h^0]+B.
\end{aligned}
\]
The bound for $\sigma_h^k$ follows from the above inequality and \eqref{equ:fem euler sigma}.
This shows~\eqref{equ:stab u H1 n L2}.}

Next, note that \eqref{equ:stab u H1 n L2} in particular implies for all $k\leq \lfloor T/\tau \rfloor$,
\begin{align}\label{equ:abs E bdd}
	\abs{\widetilde{\mathcal{E}}[u_h^k, n_h^k, r_h^k]} \leq C.
\end{align}
Then summing \eqref{equ:Ek1 min Ek} over $j\in \{0,1,\ldots, k-1\}$ yields
\begin{align*}
	\tau \sum_{j=1}^k \norm{\nabla \mu_h^j}{L^2}^2
	+
	\tau \sum_{j=1}^k \norm{\nabla \sigma_h^j}{L^2}^2
	\leq
	\abs{\widetilde{\mathcal{E}}[u_h^k, n_h^k, r_h^k]} 
	+ \widetilde{\mathcal{E}}[u_h^0, n_h^0, r_h^0]
	\leq C,
\end{align*}
by \eqref{equ:abs E bdd}. This shows \eqref{equ:stab sum mu sigma H1}, thus completing the proof.
\end{proof}

\begin{lemma}
For $k=1,2,\ldots, \lfloor T/\tau \rfloor$,
\begin{align}\label{equ:stab sum mu L2}
	\tau \sum_{j=1}^k \norm{\mu_h^j}{L^2}^2 + \tau \sum_{j=1}^k \norm{\Delta_h u_h^j}{L^2}^2 \leq C,
\end{align}
where $C$ depends on $T$, but is independent of $k$, $h$, and $\tau$.
\end{lemma}

\begin{proof}
Recall the notations introduced in \eqref{equ:phk bhk} and \eqref{equ:beta k bk}.
Taking $\phi=u_h^{k+1}$ and $\psi=\Delta_h u_h^{k+1}$ in \eqref{equ:fem euler}, then summing the resulting equations yield
\begin{align}\label{equ:sum u I1 I4}
	&\frac{1}{2\tau} \left(\norm{u_h^{k+1}}{L^2}^2 - \norm{u_h^k}{L^2}^2 \right)
	+
	\frac{1}{2\tau} \norm{u_h^{k+1}-u_h^k}{L^2}^2
	+
	\epsilon^2 \norm{\Delta_h u_h^{k+1}}{L^2}^2
	+
	\lambda \norm{\nabla u_h^{k+1}}{L^2}^2
	\nonumber\\
	&=
	\inpro{p_h^k \sigma_h^{k+1}}{u_h^{k+1}}
	-
	\inpro{p_h^k \mu_h^{k+1}}{u_h^{k+1}}
	-
	\chi_0 \inpro{n_h^{k+1}}{\Delta_h u_h^{k+1}}
	+
	\frac{r_h^{k+1}}{\sqrt{\beta_h^k}} \inpro{f'(u_h^k)}{\Delta_h u_h^{k+1}}
	\nonumber\\
	&=:I_1+I_2+I_3+I_4.
\end{align}
We will estimate each term on the last line. For the first term, by \eqref{equ:P sublin}, H\"older's inequality, and the Sobolev embedding $H^1\hookrightarrow L^4$, we have
\begin{align*}
	\abs{I_1} \leq \norm{p_h^k}{L^4} \norm{\sigma_h^{k+1}}{L^2} \norm{u_h^{k+1}}{L^4}
	\leq
	C\left(1+ \norm{u_h^k}{H^1}\right) \norm{\sigma_h^{k+1}}{L^2} \norm{u_h^{k+1}}{H^1}
	\leq C,
\end{align*}
where in the last step we used \eqref{equ:stab u H1 n L2}. Next, for the second term, by H\"older's inequality, \eqref{equ:stab u H1 n L2}, and the definition of $\mu_h^{k+1}$ in \eqref{equ:fem euler}, we obtain
\begin{align*}
	\abs{I_2} 
	&\leq
	\norm{p_h^k}{L^4} \norm{\mu_h^{k+1}}{L^2} \norm{u_h^{k+1}}{L^4}
	\leq
	C\norm{\mu_h^{k+1}}{L^2}
	\\
	&\leq
	C\norm{\Delta_h u_h^{k+1}}{L^2}
	+
	C\norm{u_h^{k+1}}{L^2}
	+
	C\norm{n_h^{k+1}}{L^2}
	+
	C \norm{f'(u_h^k)}{L^2} \norm{\Delta_h u_h^{k+1}}{L^2}
	\\
	&\leq
	C + \frac{\epsilon^2}{4} \norm{\Delta_h u_h^{k+1}}{L^2}^2
	+
	C\norm{u_h^k}{H^1}^6,
\end{align*}
where in the last step we applied \eqref{equ:f der bound}, \eqref{equ:stab u H1 n L2}, Young's inequality, and the Sobolev embedding $H^1\hookrightarrow L^6$. For the term $I_3$, by Young's inequality and \eqref{equ:stab u H1 n L2}, we have
\begin{align*}
	\abs{I_3} 
	&\leq
	C\norm{n_h^{k+1}}{L^2}^2
	+
	\frac{\epsilon^2}{4} \norm{\Delta_h u_h^{k+1}}{L^2}^2
	\leq
	C+ \frac{\epsilon^2}{4} \norm{\Delta_h u_h^{k+1}}{L^2}^2.
\end{align*}
Finally, for the last term, by \eqref{equ:f der bound}, \eqref{equ:stab u H1 n L2}, Young's inequality, and the Sobolev embedding $H^1\hookrightarrow L^6$, we infer
\begin{align*}
	\abs{I_4}
	\leq
	C\norm{f'(u_h^k)}{L^2} \norm{{\Delta_h u_h^{k+1}}}{L^2}
	\leq
	C\norm{u_h^k}{H^1}^6 + \frac{\epsilon^2}{4} \norm{\Delta_h u_h^{k+1}}{L^2}^2
	\leq
	C+ \frac{\epsilon^2}{4} \norm{\Delta_h u_h^{k+1}}{L^2}^2.
\end{align*}
Altogether, continuing from \eqref{equ:sum u I1 I4}, rearranging the terms, and summing over $j\in \{0,1,\ldots, k-1\}$ yield the bound \eqref{equ:stab sum mu L2} for the second term on the left-hand side of \eqref{equ:stab sum mu L2}. Since
\begin{align*}
	\mu_h^{k+1}= -\epsilon^2 \Delta_h u_h^{k+1}+ \lambda u_h^{k+1} - \chi_0 n_h^{k+1} +  \frac{r_h^{k+1}}{\sqrt{\beta_h^k}} \Pi_h \big[f'(u_h^k)\big],
\end{align*}
we can obtain a bound $\norm{\mu_h^{k+1}}{L^2}^2\leq C+ C\norm{\Delta_h u_h^{k+1}}{L^2}^2$ in a straightforward manner. This implies the bound for the first term in \eqref{equ:stab sum mu L2}, thus completing the proof.
\end{proof}

We also obtain uniform bounds for $\norm{n_h^k}{H^1}$ and $\norm{\sigma_h^k}{H^1}$.

\begin{lemma}
For $k=1,2,\ldots, \lfloor T/\tau \rfloor$,
\begin{align}\label{equ:stab n H1}
	\norm{n_h^k}{H^1}^2 + \norm{\sigma_h^k}{H^1}^2 + \tau \sum_{j=1}^k \left(\norm{\Delta_h n_h^j}{L^2}^2 + \norm{\Delta_h \sigma_h^j}{L^2}^2 \right)
	\leq C,
\end{align}
where $C$ depends on $T$, but is independent of $k$, $h$, and $\tau$.
\end{lemma}

\begin{proof}
Setting $\varphi= -\Delta_h n_h^{k+1}$ in \eqref{equ:fem euler n} and taking the inner product of \eqref{equ:fem euler sigma} with $\Delta_h^2 n_h^{k+1}$ in \eqref{equ:fem euler}, then adding the resulting equations (noting the definition of $\Delta_h$ in \eqref{equ:disc laplacian}), we obtain
\begin{align}\label{equ:sum nab n J1}
	&\frac{1}{2\tau} \left(\norm{\nabla n_h^{k+1}}{L^2}^2 - \norm{\nabla n_h^k}{L^2}^2 \right)
	+
	\frac{1}{2\tau} \norm{\nabla n_h^{k+1}-\nabla n_h^k}{L^2}^2
	+
	\delta^{-1} \norm{\Delta_h n_h^{k+1}}{L^2}^2
	\nonumber\\
	&=
	\chi_0 \inpro{\Delta_h u_h^k}{\Delta_h n_h^{k+1}}
	-
	\inpro{P(u_h^k)\sigma_h^{k+1}}{\Delta_h n_h^{k+1}}
	+
	\inpro{P(u_h^k) \mu_h^{k+1}}{\Delta_h n_h^{k+1}}
	\nonumber\\
	&=:J_1+J_2+J_3.
\end{align}
Each term on the last line will be estimated as follows. For the first term, by Young's inequality,
\begin{align*}
	\abs{J_1}
	&\leq
	\frac{1}{4\delta} \norm{\Delta_h n_h^{k+1}}{L^2}^2
	+
	C\norm{\Delta_h u_h^k}{L^2}^2.
\end{align*}
Similarly, for the terms $J_2$ and $J_3$, noting \eqref{equ:P sublin}, the Sobolev embedding, and \eqref{equ:stab u H1 n L2}, we have
\begin{align*}
	\abs{J_2}
	&\leq
	\frac{1}{4\delta} \norm{\Delta_h n_h^{k+1}}{L^2}^2
	+
	C\left(1+\norm{u_h^k}{L^4}^2\right) \norm{\sigma_h^{k+1}}{L^4}^2
	\leq  
	\frac{1}{4\delta} \norm{\Delta_h n_h^{k+1}}{L^2}^2
	+
	C\norm{\sigma_h^{k+1}}{H^1}^2,
	\\
	\abs{J_3}
	&\leq
	\frac{1}{4\delta} \norm{\Delta_h n_h^{k+1}}{L^2}^2
	+
	C\left(1+\norm{u_h^k}{L^4}^2\right) \norm{\mu_h^{k+1}}{L^4}^2
	\leq  
	\frac{1}{4\delta} \norm{\Delta_h n_h^{k+1}}{L^2}^2
	+
	C\norm{\mu_h^{k+1}}{H^1}^2.
\end{align*}
Substituting these estimates into \eqref{equ:sum nab n J1}, summing over $j\in \{0,1,\ldots,k-1\}$, and rearranging the terms, we infer
\begin{align*}
	\norm{\nabla n_h^k}{L^2}^2 + \tau \sum_{j=1}^k \norm{\Delta_h n_h^j}{L^2}^2
	\leq
	C\tau \sum_{j=1}^k \left( \norm{\Delta_h u_h^j}{L^2}^2 + \norm{\sigma_h^j}{H^1}^2 + \norm{\mu_h^j}{H^1}^2 \right)
	\leq
	C,
\end{align*}
where in the last step we used \eqref{equ:stab u H1 n L2}, \eqref{equ:stab sum mu sigma H1}, and \eqref{equ:stab sum mu L2}. From this, noting the definition of $\sigma_h^k$ and \eqref{equ:stab sum mu L2}, we deduce \eqref{equ:stab n H1}.
\end{proof}

The following lemma will be needed to establish a bound for $\norm{u_h^k}{L^\infty}$ later.

\begin{lemma}
Let $k \in \{1,2,\ldots,\lfloor T/\tau \rfloor\}$. Let $b_h^k$ and $\beta_h^k$ be defined by \eqref{equ:phk bhk}. For any $\alpha>0$ and $\psi\in V_h$, we have
\begin{align}\label{equ:ineq rhk f}
	\frac{1}{\tau} \left|r_h^{k+1} \inpro{b_h^k}{\psi}
	-
	r_h^k \inpro{b_h^{k-1}}{\psi} \right| 
	&\leq 
	C\norm{\psi}{H^1}^2 + \alpha \norm{\dtt u_h^k}{L^2}^2
	+ \alpha \norm{\dtt u_h^{k+1}}{L^2}^2,
\end{align}
where $C$ depends on $\alpha$ and $T$, but is independent of $k$, $h$, and $\tau$.
\end{lemma}

\begin{proof}
Noting the definition of $b_h^k$ and $\beta_h^k$, we write
\begin{align}\label{equ:1tau k I}
	&\frac{1}{\tau} \left(r_h^{k+1} \inpro{b_h^k}{\psi}
	-
	r_h^k \inpro{b_h^{k-1}}{\psi} \right)
	\nonumber\\
	&=
	\frac{r_h^{k+1}}{\sqrt{\beta_h^k}} \inpro{\frac{f'(u_h^k)-f'(u_h^{k-1})}{\tau}}{\psi}
	+
	\left(\frac{\dtt r_h^{k+1}}{\sqrt{\beta_h^k}} \right)  \inpro{f'(u_h^{k-1})}{\psi}
	\nonumber\\
	&\quad
	+
	\left(\frac{r_h^k}{\sqrt{\beta_h^{k-1} \beta_h^k} \left(\sqrt{\beta_h^{k-1}} + \sqrt{\beta_h^k} \right)} \right)
	\left(\frac{\mathcal{E}_1[u_h^{k-1}]- \mathcal{E}_1[u_h^k]}{\tau}\right) 
	\inpro{f'(u_h^{k-1})}{\psi}
	\nonumber\\
	&=:I_1+I_2+I_3.
\end{align}
For the first term, by H\"older's inequality, \eqref{equ:stab u H1 n L2}, and \eqref{equ:f prime xy}, we have for any $\alpha>0$,
\begin{align}\label{equ:I1 tau}
	\abs{I_1}
	&\leq
	C \left(1+\norm{u_h^k}{L^6}^2 + \norm{u_h^{k-1}}{L^6}^2\right) \norm{\dtt u_h^k}{L^2} \norm{\psi}{L^6}
	\leq
	C\norm{\psi}{H^1}^2 + \alpha \norm{\dtt u_h^k}{L^2}^2,
\end{align}
where in the last step we used Young's inequality, the embedding $H^1\hookrightarrow L^6$, and \eqref{equ:stab u H1 n L2}. Similarly for the term $I_2$, we use the expression for $\dtt r_h^k$ in \eqref{equ:fem euler r} and Young's inequality to obtain
\begin{align}\label{equ:I2 tau}
	\abs{I_2}
	&\leq
	C \norm{f'(u_h^{k-1})}{L^2}^2 \norm{\dtt u_h^{k+1}}{L^2} \norm{\psi}{L^2}
	\leq 
	C\norm{\psi}{L^2}^2 + \alpha \norm{\dtt u_h^{k+1}}{L^2}^2.
\end{align}
Finally, for the last term, applying the definition of $\mathcal{E}_1$, employing \eqref{equ:stab u H1 n L2}, H\"older's inequality, and the Sobolev embedding, we have
\begin{align}\label{equ:I3 tau}
	\abs{I_3}
	&\leq
	C \left(\int_{\mathscr{D}} \abs{\frac{f'(u_h^k)-f'(u_h^{k-1})}{\tau}} \,\dx \right) \norm{f'(u_h^{k-1})}{L^2} \norm{\psi}{L^2}
	\nonumber\\
	&\leq
	C \left(1+ \norm{u_h^k}{L^4}^2 + \norm{u_h^{k-1}}{L^4}^2 \right) \norm{\dtt u_h^k}{L^2} \left(1+ \norm{u_h^{k-1}}{L^6}^3\right) \norm{\psi}{L^2}
	\nonumber\\
	&\leq
	C\norm{\psi}{L^2}^2 + \alpha \norm{\dtt u_h^k}{L^2}^2.
\end{align}
Substituting these estimates back into \eqref{equ:1tau k I}, we obtain \eqref{equ:ineq rhk f}.
\end{proof}

An essential uniform bound on the $L^\infty$ norm of $u_h^k$ can now be shown. In the following, we set $u_h^{-1}=u_h^0$ and $b_h^{-1}=b_h^0$ when these quantities are needed only for notational convenience, and define $\mu_h^0$ by \eqref{equ:fem euler mu} with $k=-1$.

\begin{lemma}
For $k=1,2,\ldots, \lfloor T/\tau \rfloor$,
\begin{align}\label{equ:stab Delta u L2}
	\norm{\mu_h^k}{L^2}^2 + \norm{\Delta_h u_h^k}{L^2}^2 + \tau \sum_{j=1}^k \left(\norm{\dtt u_h^j}{L^2}^2 + \norm{\Delta_h \mu_h^j}{L^2}^2 \right)
	\leq C.
\end{align}
Consequently, we have
\begin{align}\label{equ:stab u Linfty}
	\norm{u_h^k}{L^\infty}^2
    +
    \tau \sum_{j=1}^k \left(\norm{\mu_h^j}{L^\infty}^2 + \norm{n_h^j}{L^\infty}^2 + \norm{\sigma_h^j}{L^\infty}^2 \right) \leq C,
\end{align}
where $C$ depends on $T$, but is independent of $k$, $h$, and $\tau$.
\end{lemma}

\begin{proof}
Recall the notations in \eqref{equ:phk bhk}.
First, we subtract \eqref{equ:fem euler mu} at time step $k$ from the same equation at time step $k+1$ to obtain
\begin{align*}
	\inpro{\mu_h^{k+1}-\mu_h^k}{\psi}
	&=
	\epsilon^2 \inpro{\nabla u_h^{k+1}-\nabla u_h^k}{\nabla \psi}
	+
	\lambda \inpro{u_h^{k+1}-u_h^k}{\psi}
	-
	\chi_0 \inpro{n_h^{k+1}-n_h^k}{\psi}
	\nonumber\\
	&\quad
	+
	r_h^{k+1} \inpro{b_h^k}{\psi}
	-
	r_h^k\inpro{b_h^{k-1}}{\psi}
	, \qquad \forall \psi\in V_h.
\end{align*}
Next, we set $\psi=\mu_h^{k+1}/\tau$ in the above equation and take $\phi=\epsilon^2 \dtt u_h^{k+1}$ in \eqref{equ:fem euler} to obtain
\begin{align}
	\label{equ:eps dtuh L2}
	\epsilon^2 \norm{\dtt u_h^{k+1}}{L^2}^2
	&=
	-\epsilon^2 \inpro{\nabla \mu_h^{k+1}}{\nabla \dtt u_h^{k+1}}
	+
	\epsilon^2 \inpro{p_h^k (\sigma_h^{k+1}-\mu_h^{k+1})}{\dtt u_h^{k+1}},
\end{align}
and
\begin{align} 
	\label{equ:mu k mu k1}
	&\frac{1}{2\tau} \left(\norm{\mu_h^{k+1}}{L^2}^2 - \norm{\mu_h^k}{L^2}^2\right) 
	+
	\frac{1}{2\tau} \norm{\mu_h^{k+1}-\mu_h^k}{L^2}^2
	\nonumber\\
	&=
	\epsilon^2 \inpro{\nabla\dtt u_h^{k+1}}{\nabla \mu_h^{k+1}}
	+
	\lambda \inpro{\dtt u_h^{k+1}}{\mu_h^{k+1}}
	-
	\chi_0 \inpro{\dtt n_h^{k+1}}{\mu_h^{k+1}}
	\nonumber\\
	&\quad
	+
	\frac{1}{\tau} \left(r_h^{k+1} \inpro{b_h^k}{\mu_h^{k+1}}
	-
	r_h^k \inpro{b_h^{k-1}}{\mu_h^{k+1}}\right).
\end{align}
Furthermore, taking $\varphi=-\chi_0 \mu_h^{k+1}$ gives
\begin{align}\label{equ:chi 0 dtn mu}
	-\chi_0 \inpro{\dtt n_h^{k+1}}{\mu_h^{k+1}}
	&=
	\chi_0 \inpro{\nabla \sigma_h^{k+1}}{\nabla \mu_h^{k+1}}
	+
	\chi_0 \inpro{p_h^k \sigma_h^{k+1}}{\mu_h^{k+1}}
	-
	\chi_0 \norm{\sqrt{p_h^k} \mu_h^{k+1}}{L^2}^2.
\end{align}
Adding \eqref{equ:eps dtuh L2}, \eqref{equ:mu k mu k1}, and \eqref{equ:chi 0 dtn mu} yields
\begin{align}\label{equ:mu k1 mu k}
	&\frac{1}{2\tau} \left(\norm{\mu_h^{k+1}}{L^2}^2 - \norm{\mu_h^k}{L^2}^2\right) 
	+
	\frac{1}{2\tau} \norm{\mu_h^{k+1}-\mu_h^k}{L^2}^2
	+
	\epsilon^2 \norm{\dtt u_h^{k+1}}{L^2}^2
	+
	\chi_0 \norm{\sqrt{p_h^k} \mu_h^{k+1}}{L^2}^2
	\nonumber\\
	&=
	\epsilon^2 \inpro{p_h^k (\sigma_h^{k+1}-\mu_h^{k+1})}{\dtt u_h^{k+1}}
	+
	\lambda \inpro{\dtt u_h^{k+1}}{\mu_h^{k+1}}
	+
	\chi_0 \inpro{\nabla \sigma_h^{k+1}}{\nabla \mu_h^{k+1}}
	\nonumber\\
	&\quad
	+
	\chi_0 \inpro{p_h^k \sigma_h^{k+1}}{\mu_h^{k+1}}
	+
	\frac{1}{\tau} \left(r_h^{k+1} \inpro{b_h^k}{\mu_h^{k+1}}
	-
	r_h^k \inpro{b_h^{k-1}}{\mu_h^{k+1}}\right)
	\nonumber\\
	&=:J_1+J_2+J_3+J_4+J_5.
\end{align}
Each term on the last line will be estimated as follows. For the first term, by Young's inequality and \eqref{equ:P sublin}, we have
\begin{align*}
	\abs{J_1}
	&\leq
	C\left(1+\norm{u_h^k}{L^4}^2 \right) \norm{\sigma_h^{k+1}-\mu_h^{k+1}}{L^4}^2
	+
	\frac{\epsilon^2}{4} \norm{\dtt u_h^{k+1}}{L^2}^2
	\\
	&\leq
	C\left(1+\norm{\mu_h^{k+1}}{H^1}^2\right) + \frac{\epsilon^2}{4} \norm{\dtt u_h^{k+1}}{L^2}^2,
\end{align*}
where in the last step we also used \eqref{equ:stab n H1} and \eqref{equ:stab u H1 n L2}. For the term $J_2$, by Young's inequality,
\begin{align*}
	\abs{J_2} \leq C\norm{\mu_h^{k+1}}{L^2}^2 +\frac{\epsilon^2}{4} \norm{\dtt u_h^{k+1}}{L^2}^2. 
\end{align*}
For the terms $J_3$ and $J_4$, we apply Young's inequality to obtain
\begin{align*}
	\abs{J_3}
	&\leq
	C\norm{\nabla \sigma_h^{k+1}}{L^2}^2
	+
	C\norm{\nabla \mu_h^{k+1}}{L^2}^2,
	\\
	\abs{J_4}
	&\leq
	C \left(1+\norm{u_h^{k+1}}{L^4}^2\right) \norm{\sigma_h^{k+1}}{L^4}^2
	+
	C\norm{\mu_h^{k+1}}{L^2}^2
	\leq
	C\left(1+ \norm{\mu_h^{k+1}}{L^2}^2 \right),
\end{align*}
where in the last step for $J_4$ we also used \eqref{equ:stab n H1} and \eqref{equ:stab u H1 n L2}. Finally, for the term $J_5$, we apply \eqref{equ:ineq rhk f} with $\psi=\mu_h^{k+1}$ and $\alpha=\epsilon^2/4$ to infer
\begin{align*}
	\abs{J_5}
	\leq
	C\norm{\mu_h^{k+1}}{H^1}^2 + \frac{\epsilon^2}{4} \norm{\dtt u_h^k}{L^2}^2 + \frac{\epsilon^2}{4} \norm{\dtt u_h^{k+1}}{L^2}^2.
\end{align*}
Altogether, substituting these estimates back into \eqref{equ:mu k1 mu k} and rearranging the terms, we obtain
\begin{align*}
	\left(\norm{\mu_h^{k+1}}{L^2}^2 - \norm{\mu_h^k}{L^2}^2\right) 
	+
	\frac{\epsilon^2 \tau}{2} \left(\norm{\dtt u_h^{k+1}}{L^2}^2- \norm{\dtt u_h^k}{L^2}^2 \right)
	\leq
	C\tau\left(1+ \norm{\mu_h^{k+1}}{H^1}^2 \right).
\end{align*}
Summing this over $j\in \{0,1,\ldots,k-1\}$, setting $u_h^{-1}=u_h^0$ so that $\dtt u_h^0=0$, and noting \eqref{equ:stab sum mu sigma H1} and \eqref{equ:stab sum mu L2}, we deduce
\begin{align}\label{equ:mu dt 1}
	\norm{\mu_h^k}{L^2}^2 + \tau \sum_{j=1}^k \norm{\dtt u_h^j}{L^2}^2 \leq C.
\end{align}
Now, by \eqref{equ:fem euler mu}, we have
\begin{align*}
	\epsilon^2 \Delta_h u_h^{k+1}= \mu_h^{k+1}-\lambda u_h^{k+1} + \chi_0 n_h^{k+1} -  r_h^{k+1}
	\Pi_h \big[b_h^k \big],
\end{align*}
which implies by \eqref{equ:stab u H1 n L2} and \eqref{equ:mu dt 1},
\begin{align}\label{equ:Delta uk L2}
	\norm{\Delta_h u_h^{k+1}}{L^2}^2
	&\leq
	C\norm{\mu_h^{k+1}}{L^2}^2
	+
	C\norm{u_h^{k+1}}{L^2}^2
	+
	C\norm{n_h^{k+1}}{L^2}^2
	+
	C \left(1+ \norm{u_h^k}{L^6}^6\right) \leq C.
\end{align}
By a similar argument, but using \eqref{equ:fem euler u} instead, we obtain
\begin{align*}
	\tau\Delta_h \mu_h^j= \tau \dtt u_h^j+ \tau \Pi_h \left[p_h^{j-1} (\sigma_h^j-\mu_h^j)\right].
\end{align*}
Summing this over $j\in \{1,2,\ldots,k\}$ and using \eqref{equ:stab n H1} and \eqref{equ:mu dt 1}, we infer that
\begin{align}\label{equ:sum Delta h2}
	\tau \sum_{j=1}^k \norm{\Delta_h \mu_h^j}{L^2}^2
	\leq C.
\end{align}
Estimate \eqref{equ:stab Delta u L2} then follows from \eqref{equ:mu dt 1}, \eqref{equ:Delta uk L2}, and \eqref{equ:sum Delta h2}. 

Finally, \eqref{equ:stab u Linfty} follows from \eqref{equ:disc lapl L infty}, \eqref{equ:stab Delta u L2}, and \eqref{equ:stab n H1}, noting the embedding $H^1\hookrightarrow L^6$.
\end{proof}

The following lemma will be needed to derive further uniform estimates.

\begin{lemma}
Let $k \in \{1,2,\ldots,\lfloor T/\tau \rfloor\}$. {Let $b_h^k$ and $\beta_h^k$ be defined by \eqref{equ:phk bhk}.} For any $\alpha>0$, there exists a constant $C>0$ such that for all $\psi\in V_h$,
\begin{align}\label{equ:ineq 2 rhk f}
	\frac{1}{\tau} \left|r_h^{k+1} \inpro{b_h^k}{\psi}
	-
	r_h^k \inpro{b_h^{k-1}}{\psi} \right| 
	&\leq 
	C\norm{\dtt u_h^k}{L^2}^2 + C\norm{\dtt u_h^{k+1}}{L^2}^2
	+ \alpha \norm{\psi}{L^2}^2.
\end{align}
Here, $C$ depends on $\alpha$ and $T$, but is independent of $k$, $h$, and $\tau$.
\end{lemma}

\begin{proof}
We employ the decomposition \eqref{equ:1tau k I}, but we will estimate $I_1$, $I_2$, and $I_3$ slightly differently here, in light of the estimate \eqref{equ:stab u Linfty}. Similarly to \eqref{equ:I1 tau}, for the term $I_1$, we have for any $\alpha>0$,
\begin{align*}
	\abs{I_1}
	&\leq
	C \left(1+\norm{u_h^k}{L^\infty}^2 + \norm{u_h^{k-1}}{L^\infty}^2\right) \norm{\dtt u_h^k}{L^2} \norm{\psi}{L^2}
	\leq
	C\norm{\dtt u_h^k}{L^2}^2 + \alpha \norm{\psi}{L^2}^2.
\end{align*}
For the terms $I_2$ and $I_3$, by the same argument as in \eqref{equ:I2 tau} and \eqref{equ:I3 tau}, but swapping the roles of $C$ and $\alpha$ there, we obtain
\begin{align*}
	\abs{I_2}+\abs{I_3}
	&\leq
	C\norm{\dtt u_h^k}{L^2}^2 + C\norm{\dtt u_h^{k+1}}{L^2}^2 + \alpha \norm{\psi}{L^2}^2.
\end{align*}
This completes the proof of \eqref{equ:ineq 2 rhk f}.
\end{proof}

From the above lemma, we deduce the following bounds on the discrete time derivatives of $u_h^k$ and $n_h^k$.

\begin{lemma}
For $k=1,2,\ldots, \lfloor T/\tau \rfloor$,
	\begin{align}\label{equ:stab dtu L2}
		\norm{\dtt u_h^k}{L^2}^2
		+
		\norm{\dtt n_h^k}{L^2}^2
		+
		\tau \sum_{j=1}^k \left(\norm{\nabla \dtt u_h^j}{L^2}^2 + \norm{\Delta_h \dtt u_h^j}{L^2}^2 + \norm{\dtt \mu_h^j}{L^2}^2 + \norm{\nabla \dtt n_h^j}{L^2}^2 \right) 
		\leq C.
	\end{align}
	In particular, this implies
	\begin{align}\label{equ:stab uk1 uk L2}
		\norm{u_h^{k+1}-u_h^k}{L^2} + \norm{n_h^{k+1}-n_h^k}{L^2} \leq C\tau,
	\end{align}
	where $C$ depends on $T$, but is independent of $k$, $h$, and $\tau$.
\end{lemma}

\begin{proof}
We subtract \eqref{equ:fem euler u}, \eqref{equ:fem euler mu}, \eqref{equ:fem euler n} at time step $k$ from the same equation at time step $k+1$ to obtain
\begin{align}
	\label{equ:u k1 min u k}
	&\inpro{\dtt u_h^{k+1}-\dtt u_h^k}{\phi}
	+
	\inpro{\nabla \mu_h^{k+1}-\nabla \mu_h^k}{\nabla \phi}
	\nonumber\\
	&=
	\inpro{p_h^k \big(\delta^{-1} n_h^{k+1}-\chi_0 u_h^k - \mu_h^{k+1}\big)- p_h^{k-1} \big(\delta^{-1} n_h^k-\chi_0 u_h^{k-1} - \mu_h^k\big)}{\phi}, \qquad \forall \phi \in V_h,
\end{align}
and
\begin{align}
	\label{equ:mu k1 min mu k}
	\inpro{\mu_h^{k+1}-\mu_h^k}{\psi}
	&=
	\epsilon^2 \inpro{\nabla u_h^{k+1}-\nabla u_h^k}{\nabla \psi}
	+
	\lambda \inpro{u_h^{k+1}-u_h^k}{\psi}
	-
	\chi_0 \inpro{n_h^{k+1}-n_h^k}{\psi}
	\nonumber\\
	&\quad
	+
	r_h^{k+1} \inpro{b_h^k}{\psi}
	-
	r_h^k \inpro{b_h^{k-1}}{\psi}
	, \qquad \forall \psi\in V_h,
\end{align}
as well as
\begin{align}\label{equ:n k1 min n}
	&\inpro{\dtt n_h^{k+1} -\dtt n_h^k}{\varphi}
	\nonumber\\
	&=
	-
	\delta^{-1} \inpro{\nabla n_h^{k+1}-\nabla n_h^k}{\nabla \varphi}
	-
	\chi_0 \inpro{\nabla u_h^k-\nabla u_h^{k-1}}{\nabla \varphi}
	\nonumber\\
	&\quad
	-
	\inpro{p_h^k \big(\delta^{-1} n_h^{k+1}-\chi_0 u_h^k - \mu_h^{k+1}\big)- p_h^{k-1}  \big(\delta^{-1} n_h^k-\chi_0 u_h^{k-1} - \mu_h^k\big)}{\varphi}, \qquad \forall \varphi\in V_h.
\end{align}

Now, taking $\phi=2\epsilon^2 \dtt u_h^{k+1}$ in \eqref{equ:u k1 min u k} and $\psi=\Delta_h \dtt u_h^{k+1}$ in \eqref{equ:mu k1 min mu k}, we obtain
\begin{align}\label{equ:dt min dt}
	&\epsilon^2 \left(\norm{\dtt u_h^{k+1}}{L^2}^2 - \norm{\dtt u_h^k}{L^2}^2 \right)
	+
	\epsilon^2 \norm{\dtt u_h^{k+1}-\dtt u_h^k}{L^2}^2
	+
	2\epsilon^2 \inpro{\nabla \mu_h^{k+1}-\nabla \mu_h^k}{\nabla \dtt u_h^{k+1}}
	\nonumber\\
	&=
	\epsilon^2 \inpro{p_h^k\big(\delta^{-1} n_h^{k+1}-\chi_0 u_h^k - \mu_h^{k+1}\big)- p_h^{k-1} \big(\delta^{-1} n_h^k-\chi_0 u_h^{k-1} - \mu_h^k\big)}{\dtt u_h^{k+1}}
	\nonumber\\
	&=
	\epsilon^2 \inpro{\big(p_h^k-p_h^{k-1}\big) \big(\delta^{-1} n_h^{k+1}-\chi_0 u_h^k - \mu_h^{k+1}\big)}{\dtt u_h^{k+1}}
	\nonumber\\
	&\quad
	+
	\epsilon^2 \tau 
	\inpro{p_h^{k-1} \big(\delta^{-1} \dtt n_h^{k+1}- \chi_0 \dtt u_h^k - \dtt \mu_h^{k+1} \big)}{\dtt u_h^{k+1}}
	\nonumber\\
	&=: A_1+A_2,
\end{align}
and
\begin{align}\label{equ:nab mu min nab mu}
	&\epsilon^2 \tau \norm{\Delta_h \dtt u_h^{k+1}}{L^2}^2 
	+
	\lambda \tau \norm{\nabla \dtt u_h^{k+1}}{L^2}^2
	-
	\epsilon^2
	\inpro{\nabla \mu_h^{k+1}-\nabla \mu_h^k}{\nabla \dtt u_h^{k+1}}
	\nonumber\\
	&=
	-\chi_0 \tau \inpro{\dtt n_h^{k+1}}{\Delta_h \dtt u_h^{k+1}}
	+
	r_h^{k+1} \inpro{b_h^k}{\Delta_h \dtt u_h^{k+1}}
	-
	r_h^k \inpro{b_h^{k-1}}{\Delta_h \dtt u_h^{k+1}}
	\nonumber\\
	&=: B_1+B_2+B_3.
\end{align}
Setting $\psi=\dtt \mu_h^{k+1}$ in \eqref{equ:mu k1 min mu k}, we have
\begin{align}\label{equ:mu min mu}
	&\tau \norm{\dtt \mu_h^{k+1}}{L^2}^2
	-
	\epsilon^2 \inpro{\nabla \dtt u_h^{k+1}}{\nabla \mu_h^{k+1} -\nabla\mu_h^k}
	\nonumber\\
	&=
	\lambda \tau \inpro{\dtt u_h^{k+1}}{\dtt \mu_h^{k+1}}
	-
	\chi_0 \inpro{\dtt n_h^{k+1}}{\dtt \mu_h^{k+1}}
	+
	r_h^{k+1} \inpro{b_h^k}{\dtt \mu_h^{k+1}}
	-
	r_h^k \inpro{b_h^{k-1}}{\dtt \mu_h^{k+1}}
	\nonumber\\
	&=:D_1+D_2+D_3+D_4.
\end{align}
Finally, taking $\varphi=2\dtt n_h^{k+1}$ in \eqref{equ:n k1 min n} yields
\begin{align}\label{equ:dtn min dtn}
	&\norm{\dtt n_h^{k+1}}{L^2}^2 - \norm{\dtt n_h^k}{L^2}^2 + \norm{\dtt n_h^{k+1}-\dtt n_h^k}{L^2}^2
	+
	2\delta^{-1} \tau \norm{\nabla\dtt n_h^{k+1}}{L^2}^2
	\nonumber\\
	&=
	-\tau\chi_0 \inpro{\Delta_h \dtt u_h^k}{\dtt n_h^{k+1}}
	- \inpro{\big(p_h^k-p_h^{k-1}\big)\big(\delta^{-1} n_h^{k+1}-\chi_0 u_h^k - \mu_h^{k+1}\big)}{\dtt n_h^{k+1}}
	\nonumber\\
	&\quad
	-
	\tau 
	\inpro{p_h^{k-1} \big(\delta^{-1} \dtt n_h^{k+1}- \chi_0 \dtt u_h^k - \dtt \mu_h^{k+1} \big)}{\dtt n_h^{k+1}}
	\nonumber\\
	&=:E_1+E_2+E_3.
\end{align}
We now add equations \eqref{equ:dt min dt}, \eqref{equ:nab mu min nab mu}, \eqref{equ:mu min mu}, and \eqref{equ:dtn min dtn}, then estimate each term coming from the right-hand side of each equation as follows. Let $\alpha>0$ be a number to be chosen later.
\\[2ex]
\underline{Estimates for $A_1$ and $A_2$}: By the Young and Gagliardo--Nirenberg inequalities, noting \eqref{equ:P lipschitz} as well as the uniform estimates \eqref{equ:stab u H1 n L2} and \eqref{equ:stab n H1}, we have
\begin{align*}
	\abs{A_1}
	&\leq
	C\norm{u_h^k-u_h^{k-1}}{L^2} \left(1+\norm{\mu_h^{k+1}}{L^4}\right) \norm{\dtt u_h^{k+1}}{L^4}
	\\
	&\leq
	C\tau \left(1+\norm{\mu_h^{k+1}}{H^1}^2 \right) \norm{\dtt u_h^k}{L^2}^2
	+
	C\tau \norm{\dtt u_h^{k+1}}{L^2}^2
	+
	\alpha \tau \norm{\nabla \dtt u_h^{k+1}}{L^2}^2.
\end{align*}
For the term $A_2$, by \eqref{equ:P sublin} and \eqref{equ:stab u Linfty}, and Young's inequality, we have
\begin{align*}
	\abs{A_2}
	&\leq
	C\tau \left(1+ \norm{u_h^{k-1}}{L^\infty}\right) \left(\norm{\dtt n_h^{k+1}}{L^2} + \norm{\dtt u_h^k}{L^2} + \norm{\dtt \mu_h^{k+1}}{L^2} \right) \norm{\dtt u_h^{k+1}}{L^2}
	\\
	&\leq
	C\tau \norm{\dtt u_h^k}{L^2}^2 + C\tau \norm{\dtt u_h^{k+1}}{L^2}^2 + \alpha\tau \norm{\dtt n_h^{k+1}}{L^2}^2 + \alpha\tau \norm{\dtt \mu_h^{k+1}}{L^2}^2.
\end{align*}
\\
\underline{Estimates for $B_1$ to $B_3$}: We estimate the term $B_1$ by Young's inequality:
\begin{align*}
	\abs{B_1}
	&\leq
	C \tau \norm{\dtt n_h^{k+1}}{L^2}^2
	+
	\alpha\tau \norm{\Delta_h \dtt u_h^{k+1}}{L^2}^2.
\end{align*}
For $B_2$ and $B_3$, we apply \eqref{equ:ineq 2 rhk f} to obtain
\begin{align*}
	\abs{B_2+B_3}
	&\leq
	C\tau \norm{\dtt u_h^k}{L^2}^2
	+
	C\tau \norm{\dtt u_h^{k+1}}{L^2}^2
	+
	\alpha \tau \norm{\Delta_h \dtt u_h^{k+1}}{L^2}^2.
\end{align*}
\\
\underline{Estimates for $D_1$ to $D_4$}: For the terms $D_1$ and $D_2$, by Young's inequality we have
\begin{align*}
	\abs{D_1}+\abs{D_2}
	&\leq
	C\tau \norm{\dtt u_h^{k+1}}{L^2}^2 + C\tau \norm{\dtt n_h^{k+1}}{L^2}^2 + \alpha \tau \norm{\dtt \mu_h^{k+1}}{L^2}^2.
\end{align*}
For the final two terms, we apply \eqref{equ:ineq 2 rhk f} to obtain
\begin{align*}
	\abs{D_3+D_4}
	&\leq
	C\tau \norm{\dtt u_h^k}{L^2}^2
	+
	C\tau \norm{\dtt u_h^{k+1}}{L^2}^2
	+
	\alpha\tau \norm{{\dtt \mu_h^{k+1}}}{L^2}^2.
\end{align*}
\\
\underline{Estimates for $E_1$ to $E_2$}: For the term $E_1$, by Young's inequality we have
\begin{align*}
	\abs{E_1}
	&\leq
	C\tau \norm{\dtt n_h^{k+1}}{L^2}^2 + \alpha\tau \norm{\Delta_h \dtt u_h^k}{L^2}^2.
\end{align*}
The terms $E_2$ and $E_3$ can be estimated analogously to $A_1$ and $A_2$ in \eqref{equ:dt min dt}, yielding
\begin{align*}
	\abs{E_2}
	&\leq
	C\tau \left(1+\norm{\mu_h^{k+1}}{H^1}^2 \right) \norm{\dtt u_h^k}{L^2}^2
	+
	C\tau \norm{\dtt n_h^{k+1}}{L^2}^2
	+
	\alpha \tau \norm{\nabla \dtt n_h^{k+1}}{L^2}^2,
	\\
	\abs{E_3}
	&\leq
	C\tau \norm{\dtt u_h^k}{L^2}^2 + C\tau \norm{\dtt n_h^{k+1}}{L^2}^2 + \alpha\tau \norm{\dtt \mu_h^{k+1}}{L^2}^2.
\end{align*}
\\
\underline{Conclusion of the proof}: Altogether, after summing the above estimates over $j\in \{0,1,\ldots,k-1\}$, choosing $\alpha>0$ sufficiently small so as to absorb relevant terms to the left-hand side, we deduce
\begin{align*}
	&\norm{\dtt u_h^k}{L^2}^2
	+
	\norm{\dtt n_h^k}{L^2}^2
	+
	\tau \sum_{j=1}^k \left(\norm{\nabla \dtt u_h^j}{L^2}^2 + \norm{\Delta_h \dtt u_h^j}{L^2}^2 + \norm{\dtt \mu_h^j}{L^2}^2 + \norm{\nabla \dtt n_h^j}{L^2}^2 \right) 
	\\
	&\leq 
	C\tau \sum_{j=1}^k \left(1+\norm{\mu_h^j}{H^1}^2 \right) \norm{\dtt u_h^{j-1}}{L^2}^2
	+
	C\tau \sum_{j=1}^k \left(\norm{\dtt u_h^j}{L^2}^2+ \norm{\dtt n_h^j}{L^2}^2 \right).
\end{align*}
Noting \eqref{equ:stab sum mu sigma H1}, we conclude \eqref{equ:stab dtu L2} for sufficiently small $\tau>0$ by the discrete Gronwall lemma. Inequality \eqref{equ:stab uk1 uk L2} then follows, since by definition $u_h^{k+1}-u_h^k= \tau \dtt u_h^{k+1}$ and $n_h^{k+1}-n_h^k= \tau \dtt n_h^{k+1}$.
\end{proof}

To facilitate the error analysis of the numerical scheme, we decompose the approximation error as follows:
\begin{align}
	\label{equ:split uk}
	u_h^k-u^k &= (u_h^k -R_h u^k) + (R_h u^k- u^k) =: \theta^k + \rho^k,
	\\
	\label{equ:split mu}
	\mu_h^k-\mu^k &= (\mu_h^k -R_h \mu^k) + (R_h \mu^k- \mu^k) =: \xi^k + {\eta^k},
	\\
	\label{equ:split n}
	n_h^k-n^k &= (n_h^k -R_h n^k) + (R_h n^k- n^k) =: \zeta^k + \omega^k,
	\\
	\label{equ:split rk}
	r_h^k-r^k &= e^k,
\end{align}
where $R_h$ denotes the Ritz projection defined in \eqref{equ:Ritz}. By construction, we have
\begin{align}\label{equ:Ritz zero}
	\inpro{\nabla \rho^k}{\nabla {\phi}} = \inpro{\nabla \eta^k}{\nabla {\phi}} = 
	\inpro{\nabla \omega^k}{\nabla {\phi}} = 0, \quad \forall {\phi}\in V_h.
\end{align}
We begin with some preliminary estimates. For any $p\in [1,\infty]$, the discrete time derivative satisfies
\begin{align}\label{equ:dt vn Lp}
	\norm{\dtt v^k}{L^p}
	&=
	\norm{\frac{1}{\tau} \int_{t_{k-1}}^{t_k} \partial_t v(t) \,\dt}{L^p}
	\leq
	\norm{v}{W^{1,\infty}_T(L^p)},
\end{align}
while the difference between the discrete and the continuous time derivatives is bounded by
\begin{align}
	\label{equ:dt vn min par v}
	\norm{\dtt v^k- \partial_t v^k}{L^p}
	&=
	\norm{\frac{1}{2\tau} \int_{t_{k-1}}^{t_k} (t-t_{k-1}) \partial_t^2 v(t)\,\dt}{L^p}
	\leq 
	C_p \tau \norm{v}{W^{2,\infty}_T(L^p)}.
\end{align}

To control the difference between expressions involving $\mathcal{E}_1[\,\cdot\,]+B$, we require the following lemma.

\begin{lemma}\label{lem:rf frac est}
{Let $\beta_h^k, b_h^k, \beta^k, b^k$ be defined by \eqref{equ:phk bhk} and \eqref{equ:beta k bk}.} For $k=1,2,\ldots, \lfloor T/\tau \rfloor$, 
\begin{align}
	\label{equ:rhk rk frac}
	\abs{\frac{r_h^k}{\sqrt{\beta_h^{k-1}}} - \frac{r^{k+1}}{\sqrt{\beta^{k+1}}}}
	&\leq
	C\abs{e^k} + C\norm{\theta^{k-1}}{L^2} + Ch^{q+1} + C\tau,
	\\[1ex]
	\label{equ:frac rk f min rk1 f}
	\norm{b^k - b^{k+1}}{H^s}
	&\leq
	C\tau, \qquad s\in \{0,1\},
	\\[1ex]
	\label{equ:frac L2 uhk uk}
	\norm{b_h^k - b^k}{H^s}
	&\leq
	C\norm{\theta^k}{H^s} + Ch^{q+1-s}, \qquad s\in \{0,1\},
\end{align}
where $C$ depends on $T$, but is independent of $k$, $h$, and $\tau$.
\end{lemma}

\begin{proof}
Recall that $\beta_h^k$ and $\beta^k$ are bounded below by a strictly positive constant. First, we prove \eqref{equ:rhk rk frac}. We have by the triangle inequality and \eqref{equ:dt vn Lp},
\begin{align}\label{equ:rhk steps}
	&\abs{\frac{r_h^k}{\sqrt{\beta_h^{k-1}}} - \frac{r^{k+1}}{\sqrt{\beta^{k+1}}}}
	\nonumber\\
	&\leq
	\abs{\frac{r_h^k}{\sqrt{\beta_h^{k-1}}} - \frac{r^k}{\sqrt{\beta_h^{k-1}}}}
	+
	\abs{\frac{r^k}{\sqrt{\beta_h^{k-1}}} - \frac{r^{k+1}}{\sqrt{\beta_h^{k-1}}}}
	+
	\abs{r^{k+1}} \abs{\frac{1}{\sqrt{\beta_h^{k-1}}} - \frac{1}{\sqrt{\beta^{k+1}}}}
	\nonumber\\
	&\leq
	C\abs{e^k} + C\tau 
	+
	C \abs{\mathcal{E}_1[u_h^{k-1}]-\mathcal{E}_1[u^{k+1}]}
	\nonumber\\
	&\leq
	C\abs{e^k} +C\tau+ C\norm{f(u_h^{k-1})-f(u^{k-1})}{L^1}
	+
	C\norm{f(u^{k-1})-f(u^{k+1})}{L^1}
	\nonumber\\
	&\leq
	C\abs{e^k} +C\tau + C \left(1+\norm{u^{k-1}}{L^\infty}^3 + \norm{u_h^{k-1}}{L^\infty}^3 \right) \norm{u^{k-1}-u_h^{k-1}}{L^2}
	+
	C \norm{u^{k-1}-u^{k+1}}{L^2}
	\nonumber\\
	&\leq
	C\abs{e^k} + C\tau+ C\norm{\theta^{k-1}}{L^2} + Ch^{q+1},
\end{align}
where in the penultimate step we used \eqref{equ:f prime xy}, and in the last step we used \eqref{equ:split uk}, \eqref{equ:dt vn Lp}, and \eqref{equ:Ritz ineq}.

Next, we show \eqref{equ:frac rk f min rk1 f}. By splitting the terms on the left-hand side, we have for $s\in\{0,1\}$,
\begin{align}\label{equ:rkfp steps}
	\norm{b^k - b^{k+1}}{H^s}
	&\leq
	\frac{1}{\sqrt{\beta^k}} \norm{f'(u^k)-f'(u^{k+1})}{H^s}
	+
	\abs{\frac{\mathcal{E}_1[u^{k+1}]- \mathcal{E}_1 [u^k]}{\sqrt{\beta^k \beta^{k+1}} \left(\sqrt{\beta^k} +\sqrt{\beta^{k+1}}\right)}}  \norm{f'(u^{k+1})}{H^s}
	\nonumber\\
	&\leq
	C\tau + C\norm{f'(u^k)-f'(u^{k+1})}{H^s}
	+
	C\norm{f(u^{k+1})-f(u^k)}{L^1}.
\end{align}
Now, the last term can be estimated as in \eqref{equ:rhk steps}:
\begin{align*}
	\norm{f(u^{k+1})-f(u^k)}{L^1}
	&\leq
	C\left(1+ \norm{u^k}{L^\infty}^3 + \norm{u^{k+1}}{L^\infty}^3 \right) \norm{u^{k+1}-u^k}{L^2}
	\leq C\tau,
\end{align*}
while the second term on the right-hand side of \eqref{equ:rkfp steps} can be estimated by using \eqref{equ:f prime xy} as follows. If $s=0$, then
\begin{align*}
	\norm{f'(u^k)-f'(u^{k+1})}{L^2}
	&\leq
	C\left(1+ \norm{u^k}{L^\infty}^2 + \norm{u^{k+1}}{L^\infty}^2 \right) \norm{u^{k+1}-u^k}{L^2}
	\leq C\tau.
\end{align*}
If $s=1$, then we have the estimate
\begin{align*}
	&\norm{\nabla \left[f'(u^k)\right] - \nabla\left[f'(u^{k+1})\right]}{L^2}
	\\
	&=
	\norm{f''(u^k)\nabla u^k- f''(u^{k+1}) \nabla u^{k+1}}{L^2}
	\\
	&\leq
	\norm{f''(u^k)}{L^\infty} \norm{\nabla u^k-\nabla u^{k+1}}{L^2}
	+
	\norm{f''(u^k)-f''(u^{k+1})}{L^2} \norm{\nabla u^{k+1}}{L^\infty}
	\\
	&\leq
	C\tau,
\end{align*}
where we used the regularity assumption on $u$ in \eqref{equ:reg u euler}, as well as \eqref{equ:f prime xy}.
Altogether, substituting these back into \eqref{equ:rkfp steps} yields \eqref{equ:frac rk f min rk1 f}.

Next, we show \eqref{equ:frac L2 uhk uk} for $s=0$. The left-hand side of \eqref{equ:frac L2 uhk uk} can be split into
\begin{align}\label{equ:fp steps uk}
	\norm{b_h^k-b^k}{L^2}
	&\leq
	\norm{\frac{f'(u_h^k)-f'(u^k)}{\sqrt{\beta_h^k}}}{L^2}
	+
	\abs{\frac{\mathcal{E}_1[u^k]- \mathcal{E}_1[u_h^k]}{\sqrt{\beta^k \beta_h^k} \Big(\sqrt{\beta^k}+ \sqrt{\beta_h^k}\Big)}} \norm{f'(u^k)}{L^2}
	\nonumber\\
	&\leq
	C\norm{f'(u_h^k)-f'(u^k)}{L^2} + C\norm{f(u^k)-f(u_h^k)}{L^1}
	\nonumber\\
	&\leq
	C\norm{\theta^k}{L^2} + Ch^{q+1},
\end{align}
by applying \eqref{equ:f prime xy} and arguing similarly to \eqref{equ:rkfp steps}.

Finally, to prove \eqref{equ:frac L2 uhk uk} for $s=1$, we apply similar reasoning to that used in \eqref{equ:fp steps uk} and \eqref{equ:rkfp steps}:
\begin{align*}
	\norm{b_h^k-b^k}{H^1}
	&\leq
	C\norm{f'(u_h^k)-f'(u^k)}{H^1} + C\norm{f(u^k)-f(u_h^k)}{L^1}
	\\
	&\leq
	C \norm{f''(u_h^k)}{L^\infty} \norm{\nabla u_h^k-\nabla u^k}{L^2}
	\\
	&\quad
	+
	C \norm{f''(u_h^k)-f''(u^k)}{L^2} \norm{\nabla u^k}{L^\infty}
	+
	C\norm{\theta^k}{L^2} + Ch^{q+1}
	\\
	&\leq
	C\norm{\theta^k}{H^1} + Ch^q,
\end{align*}
as required. This completes the proof of the lemma.
\end{proof}

Recall the notations in \eqref{equ:phk bhk} and \eqref{equ:beta k bk}. Before proving the main theorem, we derive several auxiliary estimates. With this aim, we record the following identity, which will be used repeatedly in the subsequent analysis: subtracting the weak formulation~\eqref{equ:weak tum} at time $t=t_{k+1}$ from the finite element equation~\eqref{equ:fem euler}, and using \eqref{equ:Ritz zero}, we obtain
\begin{subequations}\label{equ:fem sub euler}
	\begin{alignat}{2}
		&\label{equ:fem euler sub u}
		\inpro{\dtt \theta^{k+1}+ \dtt \rho^{k+1} + \dtt u^{k+1}- \partial_t u^{k+1}}{\phi}
		\nonumber\\
		&=
		-\inpro{\nabla \xi^{k+1}}{\nabla \phi}
		+
		\inpro{\big(p_h^k- p^{k+1}\big) (\sigma_h^{k+1}-\mu_h^{k+1})}{\phi}
		\nonumber\\
		&\quad
		+
		\inpro{p^{k+1} \big(\delta^{-1}(\zeta^{k+1}+\omega^{k+1}) - \chi_0(\theta^k+\rho^k+u^k-u^{k+1})\big)}{\phi}
		\nonumber\\
		&\quad
		- \inpro{p^{k+1} (\xi^{k+1}+\eta^{k+1})}{\phi}
		, \quad &&\forall \phi \in V_h,
		\\[2ex]
		&\inpro{\xi^{k+1}+\eta^{k+1}}{\psi}
		\nonumber\\
		&=
		\epsilon^2 \inpro{\nabla \theta^{k+1}}{\nabla \psi}
		+
		\lambda \inpro{\theta^{k+1}+ \rho^{k+1}}{\psi}
		-
		\chi_0 \inpro{\zeta^{k+1}+\omega^{k+1}}{\psi}
		\nonumber\\
		&\quad
		+
		r_h^{k+1} \inpro{b_h^k-b^k}{\psi}
		+ 
		e^{k+1} \inpro{b^k}{\psi}
		+
		r^{k+1} \inpro{b^k-b^{k+1}}{\psi},
		\label{equ:fem euler sub mu}
		\quad &&\forall \psi\in V_h,
		\\[2ex]
		\label{equ:fem euler sub n}
		&\inpro{\dtt \zeta^{k+1}+ \dtt \omega^{k+1}+ \dtt n^{k+1}-\partial_t n^{k+1}}{\varphi}
		\nonumber\\
		&=
		-\delta^{-1} \inpro{\nabla \zeta^{k+1}}{\nabla \varphi}
		+\chi_0 \inpro{\nabla \theta^k}{\nabla \varphi}
		+\chi_0 \inpro{\nabla u^k-\nabla u^{k+1}}{\nabla \varphi}
		\nonumber\\
		&\quad
		-
		\inpro{\big(p_h^k- p^{k+1} \big) (\sigma_h^{k+1}-\mu_h^{k+1})}{\varphi}
		\nonumber\\
		&\quad
		-
		\inpro{p^{k+1} \big(\delta^{-1}(\zeta^{k+1}+\omega^{k+1}) - \chi_0(\theta^k+\rho^k+u^k-u^{k+1})\big)}{\varphi}
		\nonumber\\
		&\quad
		+ \inpro{p^{k+1} (\xi^{k+1}+\eta^{k+1})\big)}{\varphi}
		, \quad &&\forall \varphi \in V_h,
	\end{alignat}
\end{subequations}
as well as
\begin{align}
	\label{equ:fem euler sub r}
	&\dtt e^{k+1} + \dtt r^{k+1} - \partial_t r^{k+1}
	\nonumber\\
	&=
	\frac{1}{2\sqrt{\beta_h^k}}
	\inpro{f'(u_h^k)-f'(u^k)}{\dtt u_h^{k+1}}
	+
	\frac{1}{2\sqrt{\beta_h^k}} \inpro{f'(u^k)}{\dtt\theta^{k+1} + \dtt \rho^{k+1}}
	\nonumber\\
	&\quad
	+
	\frac{1}{2\sqrt{\beta_h^k}}
	\inpro{f'(u^k)}{\dtt u^{k+1}-\partial_t u^{k+1}}
	+
	\frac12 \left(\frac{1}{\sqrt{\beta_h^k}} - \frac{1}{\sqrt{\beta^k}}\right) \inpro{f'(u^k)}{\partial_t u^{k+1}}
	\nonumber\\
	&\quad
	+
	\inpro{b^k-b^{k+1}}{\partial_t u^{k+1}}.
\end{align}
\\
In order to obtain an optimal order of convergence, we need another lemma to bound $\norm{\dtt \theta^{k+1}}{\widetilde{H}^{-s}}$ for $s\in \{1,2\}$.

\begin{lemma}
For $k=0,1,\ldots, \lfloor T/\tau \rfloor$,
	\begin{align}
		\label{equ:dt theta H-1}
		\norm{\dtt \theta^{k+1}}{\widetilde{H}^{-1}}
		&\leq
		C\norm{\theta^k}{H^1}+ C\norm{\zeta^{k+1}}{L^2} + C\norm{\nabla \xi^{k+1}}{L^2} + Ch^{q+1} + C\tau,
		\\
		\label{equ:dt theta H-2}
		\norm{\dtt \theta^{k+1}}{\widetilde{H}^{-2}}
		&\leq
		C\norm{\theta^k}{H^1}+ C\norm{\zeta^{k+1}}{L^2} + C\norm{\xi^{k+1}}{L^2} + Ch^{q+1} + C\tau,
	\end{align}
	where $C$ depends on $T$, but is independent of $k$, $h$, and $\tau$.
\end{lemma}

\begin{proof}
Note that from \eqref{equ:fem euler sub u}, we have
\begin{align}\label{equ:dt theta k pi}
	\dtt \theta^{k+1}
	&=
	\Delta_h \xi^{k+1}
	-
	\Pi_h \left[ \dtt \rho^{k+1}+ \dtt u^{k+1}- \partial_t u^{k+1} \right] 
	+
	\Pi_h\left[ \big(p_h^k-p^{k+1}\big) (\sigma_h^{k+1}-\mu_h^{k+1}) \right]
	\nonumber\\
	&\quad
	+
	\Pi_h \left[ p^{k+1} \big(\delta^{-1}(\zeta^{k+1}+\omega^{k+1}) - \chi_0(\theta^k+\rho^k+u^k-u^{k+1})\big) \right]
	\nonumber\\
	&\quad
	- \Pi_h \left[ p^{k+1} (\xi^{k+1}+\eta^{k+1}) \right].
\end{align}
Taking the $\widetilde{H}^{-1}$ norm on both sides of \eqref{equ:dt theta k pi}, using \eqref{equ:disc lap H-1} and the boundedness of $\Pi_h$ in the $\widetilde{H}^{-1}$ norm, we obtain
\begin{align}\label{equ:steps dtt th H-1}
	\norm{\dtt \theta^{k+1}}{\widetilde{H}^{-1}}
	&\leq
	C \norm{\nabla \xi^{k+1}}{L^2}
	+
	C \norm{\dtt \rho^{k+1}+ \dtt u^{k+1}-\partial_t u^{k+1}}{\widetilde{H}^{-1}}
	\nonumber\\
	&\quad
	+
	C \norm{\big(p_h^k-p^{k+1}\big) (\sigma_h^{k+1}-\mu_h^{k+1})}{\widetilde{H}^{-1}}
	\nonumber\\
	&\quad
	+
	C \norm{ p^{k+1} \big(\delta^{-1}(\zeta^{k+1}+\omega^{k+1}) - \chi_0(\theta^k+\rho^k+u^k-u^{k+1})\big)}{\widetilde{H}^{-1}}
	\nonumber\\
	&\quad
	+
	C \norm{p^{k+1} (\xi^{k+1}+\eta^{k+1})}{\widetilde{H}^{-1}}
	\nonumber\\
	&=: C \norm{\nabla \xi^{k+1}}{L^2}
	+
	J_1+J_2+J_3+J_4.
\end{align}
It remains to estimate each term on the last line. For the term $J_1$, by the inclusion $L^2\hookrightarrow \widetilde{H}^{-1}$, \eqref{equ:dt vn min par v}, and \eqref{equ:Ritz ineq}, we infer
\begin{align*}
	\abs{J_1}
	&\leq
	Ch^{q+1} + C\tau.
\end{align*}
For the term $J_2$, by the embeddings $L^{6/5} \hookrightarrow \widetilde{H}^{-1}$ and $H^1\hookrightarrow L^3$, and assumption \eqref{equ:P lipschitz}, we have
\begin{align*}
	\abs{J_2}
	&\leq
	C\norm{\theta^k+ \rho^k+ u^k-u^{k+1}}{L^3} \norm{\sigma_h^{k+1}- \mu_h^{k+1}}{L^2}
	\\
	&\leq
	C \norm{\theta^k}{H^1} + Ch^{q+1} + C\tau,
\end{align*}
where in the last step we also used inequalities \eqref{equ:Ritz ineq}, \eqref{equ:stab n H1}, and \eqref{equ:stab Delta u L2}, noting \eqref{equ:dt vn Lp}. Similarly, for the terms $J_3$ and $J_4$, by the embedding $L^2 \hookrightarrow \widetilde{H}^{-1}$ and a standard argument, we have
\begin{align*}
	\abs{J_3}
	&\leq
	C \norm{\zeta^{k+1}}{L^2}
	+
	C\norm{\theta^k}{L^2}
	+
	Ch^{q+1} + C\tau,
	\\
	\abs{J_4}
	&\leq
	C\norm{\xi^{k+1}}{L^2}
	+
	Ch^{q+1}.
\end{align*}
Altogether, substituting these estimates into \eqref{equ:steps dtt th H-1}, we obtain \eqref{equ:dt theta H-1}.

Finally, note that by using \eqref{equ:disc lap H-2} and the inclusion $\widetilde{H}^{-1}\hookrightarrow \widetilde{H}^{-2}$, instead of \eqref{equ:steps dtt th H-1} we have
\begin{align*}
	\norm{\dtt \theta^{k+1}}{\widetilde{H}^{-2}}
	&\leq
	C \norm{\xi^{k+1}}{L^2}
	+
	J_1+J_2+J_3+J_4.
\end{align*}
The previous estimates for $J_1$ to $J_4$ hold, thus we obtain \eqref{equ:dt theta H-2}.
\end{proof}

In the following proposition, we establish a key auxiliary estimate that holds for sufficiently small $h$ and $\tau$.

\begin{proposition}\label{pro:ind}
Let $k \in \{1,2,\ldots,\lfloor T/\tau \rfloor\}$. For sufficiently small $h,\tau>0$, we have
	\begin{align}\label{equ:error theta n eul}
		&\norm{\theta^k}{L^2}^2 + \norm{\zeta^k}{L^2}^2 +
		\abs{e^k}^2
		+
		\tau \sum_{j=1}^k \left(\norm{\nabla \theta^j}{L^2}^2
		+
		\norm{\Delta_h \theta^j}{L^2}^2
		+
		\norm{\xi^j}{L^2}^2
		+
		\norm{\nabla \zeta^j}{L^2}^2 \right)
		\nonumber\\
		&\quad \leq C(h^{2(q+1)}+ \tau^2).
	\end{align}
    Consequently, we have the following estimate. 
    {Let $(u_h^k,\mu_h^k,n_h^k,\sigma_h^k,r_h^k)$ be the solution of the fully discrete scheme \eqref{equ:fem euler}, and let $(u,\mu,n,\sigma,r)$ be a solution of \eqref{equ:tum sav} satisfying the regularity assumption \eqref{equ:reg u euler}}. For sufficiently small $h,\tau >0$ and $k\in \{1,2,\ldots,\lfloor T/\tau \rfloor\}$, we have
\begin{align}\label{equ:error euler L2}
	\norm{u_h^k-u(t_k)}{L^2}
	+
	\norm{n_h^k-n(t_k)}{L^2}
	+
	\abs{r_h^k-r(t_k)}
	&\leq
	C(h^{q+1}+\tau).
\end{align}
Here, the constant $C$ depends on $T$, but is independent of $k$, $h$, and $\tau$.
\end{proposition}

\begin{proof}
We will take $\phi,\psi,\varphi$ in \eqref{equ:fem euler sub u}, \eqref{equ:fem euler sub mu}, and \eqref{equ:fem euler sub n} to be suitable functions in $V_h$. First, taking $\phi=(\epsilon^2 +1) \theta^{k+1}$ and using the identity \eqref{equ:ab ide}, we have
\begin{align}\label{equ:12 tau theta u L2}
	&\frac{\epsilon^2 +1}{2\tau} \left(\norm{\theta^{k+1}}{L^2}^2 - \norm{\theta^k}{L^2}^2 \right)
	+
	\frac{\epsilon^2 +1}{2\tau} \norm{\theta^{k+1}-\theta^k}{L^2}^2
	+
	(\epsilon^2 +1) \inpro{\nabla \xi^{k+1}}{\nabla \theta^{k+1}}
	\nonumber\\
	&=
	-
	(\epsilon^2 +1) \inpro{\dtt \rho^{k+1}+\dtt u^{k+1}-\partial_t u^{k+1}}{\theta^{k+1}}
	\nonumber\\
	&\quad
	+
	(\epsilon^2 +1) \inpro{\big(p_h^k-p^{k+1}\big) (\sigma_h^{k+1}-\mu_h^{k+1})}{\theta^{k+1}}
	\nonumber\\
	&\quad
	+
	(\epsilon^2 +1) \inpro{p^{k+1} \big(\delta^{-1}(\zeta^{k+1}+\omega^{k+1}) - \chi_0(\theta^k+\rho^k+u^k-u^{k+1})\big)}{\theta^{k+1}}
	\nonumber\\
	&\quad
	-
	(\epsilon^2 +1) \inpro{p^{k+1} (\xi^{k+1}+\eta^{k+1})\big)}{\theta^{k+1}}
	\nonumber\\
	&=:F_1+F_2+F_3+F_4.
\end{align}
Next, taking $\psi=\xi^{k+1}+ \Delta_h \theta^{k+1}$, we have
\begin{align}\label{equ:xi mu L2 eq}
	&\norm{\xi^{k+1}}{L^2}^2
	+
	\epsilon^2 \norm{\Delta_h \theta^{k+1}}{L^2}^2
	+
	\lambda \norm{\nabla \theta^{k+1}}{L^2}^2
	-
	(\epsilon^2+1) \inpro{\nabla \theta^{k+1}}{\nabla \xi^{k+1}}
	\nonumber\\
	&=
	-
	\inpro{\eta^{k+1}}{\xi^{k+1}}
	+
	\lambda \inpro{\theta^{k+1}}{\xi^{k+1}}
	+
	\lambda \inpro{\rho^{k+1}}{\xi^{k+1}+\Delta_h \theta^{k+1}}
		-
	\inpro{\eta^{k+1}}{\Delta_h \theta^{k+1}}
	\nonumber\\
	&\quad
	-
	\chi_0 \inpro{\zeta^{k+1}+\omega^{k+1}}{\xi^{k+1}+ \Delta_h \theta^{k+1}}
		+
	r_h^{k+1} \inpro{b_h^k-b^k}{\xi^{k+1}+ \Delta_h \theta^{k+1}}
	\nonumber\\
	&\quad
	+ 
	e^{k+1} \inpro{b^k}{\xi^{k+1}+ \Delta_h \theta^{k+1}}
	+
	r^{k+1} \inpro{b^k-b^{k+1}}{\xi^{k+1}+ \Delta_h \theta^{k+1}}
	\nonumber\\
	&=: 
	G_1+G_2+\ldots+G_8.
\end{align}
Setting $\varphi=\zeta^{k+1}$, we obtain
\begin{align}\label{equ:12 tau theta n}
	&\frac{1}{2\tau} \left(\norm{\zeta^{k+1}}{L^2}^2 - \norm{\zeta^k}{L^2}^2 \right)
	+
	\frac{1}{2\tau} \norm{\zeta^{k+1}-\zeta^k}{L^2}^2
	+
	\delta^{-1} \norm{\nabla \zeta^{k+1}}{L^2}^2
	\nonumber\\
	&=
	-\inpro{\dtt \omega^{k+1}+ \dtt n^{k+1}-\partial_t n^{k+1}}{\zeta^{k+1}}
	+\chi_0 \inpro{\nabla \theta^k}{\nabla \zeta^{k+1}}
	\nonumber\\
	&\quad
	+\chi_0 \inpro{\nabla u^k-\nabla u^{k+1}}{\nabla \zeta^{k+1}}
	-
	\inpro{\big(p_h^k- p^{k+1}\big) (\sigma_h^{k+1}-\mu_h^{k+1})}{\zeta^{k+1}}
	\nonumber\\
	&\quad
	-
	\inpro{p^{k+1} \big(\delta^{-1}(\zeta^{k+1}+\omega^{k+1}) - \chi_0(\theta^k+\rho^k+u^k-u^{k+1})\big)}{\zeta^{k+1}}
	\nonumber\\
	&\quad
	+ \inpro{p^{k+1} (\xi^{k+1}+\eta^{k+1})\big)}{\zeta^{k+1}}
	\nonumber\\
	&=: H_1+H_2+\ldots+H_6.
\end{align}
Finally, multiplying \eqref{equ:fem euler sub r} by $e^{k+1}$, we obtain
\begin{align}\label{equ:sub e 1}
	&\frac{1}{2\tau} \left(\abs{e^{k+1}}^2 - \abs{e^k}^2 \right)
	+
	\frac{1}{2\tau} \abs{e^{k+1}-e^k}^2
	\nonumber\\
	&=
	-
	(\dtt r^{k+1}-\partial_t r^{k+1}) e^{k+1}
	+
	\frac{e^{k+1}}{2\sqrt{\beta_h^k}}
	\inpro{f'(u_h^k)-f'(u^k)}{\dtt u_h^{k+1}}
	\nonumber\\
	&\quad
	+
	\frac{e^{k+1}}{2\sqrt{\beta_h^k}} \inpro{f'(u^k)}{\dtt\theta^{k+1} + \dtt \rho^{k+1}}
	+
	\frac{e^{k+1}}{2\sqrt{\beta_h^k}}
	\inpro{f'(u^k)}{\dtt u^{k+1}-\partial_t u^{k+1}}
	\nonumber\\
	&\quad
	+
	\frac12  e^{k+1} \left(\frac{1}{\sqrt{\beta_h^k}} - \frac{1}{\sqrt{\beta^k}}\right) \inpro{f'(u^k)}{\partial_t u^{k+1}}
	+
	e^{k+1} \inpro{b^k-b^{k+1}}{\partial_t u^{k+1}}
	\nonumber\\
	&=:I_1+I_2+\ldots+I_6.
\end{align}
Now, we add \eqref{equ:12 tau theta u L2}, \eqref{equ:xi mu L2 eq}, \eqref{equ:12 tau theta n}, and \eqref{equ:sub e 1}. We then estimate the terms coming from the right-hand side of each equation as follows. Let $\alpha>0$ be a number to be chosen later.
\\[2ex]
\underline{Estimates for $F_1$ to $F_4$}: Firstly, by Young's inequality, \eqref{equ:Ritz ineq}, and \eqref{equ:dt vn min par v}, we have
\begin{align*}
	\abs{F_1} &\leq
	C \norm{\dtt \rho^{k+1}}{L^2}^2 + C\norm{\dtt u^{k+1}-\partial_t u^{k+1}}{L^2}^2 + \alpha \norm{\theta^{k+1}}{L^2}^2
	\\
	&\leq
	Ch^{2(q+1)}+C\tau^2 + \alpha \norm{\theta^{k+1}}{L^2}^2.
\end{align*}
 For the term $F_2$, by \eqref{equ:P lipschitz}, \eqref{equ:split uk}, the Sobolev embedding, and Young's inequality, noting the estimate \eqref{equ:stab n H1}, we have
\begin{align*}
	\abs{F_2} &\leq
	C \norm{\theta^k+\rho^k+u^k-u^{k+1}}{L^2} \norm{\sigma_h^{k+1}+\mu_h^{k+1}}{L^4} \norm{\theta^{k+1}}{L^4}
	\\
	&\leq
	C\left(1+\norm{\mu_h^{k+1}}{H^1}^2\right) \left(h^{2(q+1)}+ \tau^2+ \norm{\theta^k}{L^2}^2 \right)
	+
	\alpha \norm{\theta^{k+1}}{H^1}^2.
\end{align*}
For $F_3$, noting \eqref{equ:P sublin}, we have
\begin{align*}
	\abs{F_3}
	&\leq
	C\norm{p^{k+1}}{L^\infty} \left(\norm{\zeta^{k+1}+\omega^{k+1}}{L^2} + \norm{\theta^k+\rho^k+u^k-u^{k+1}}{L^2} \right) \norm{\theta^{k+1}}{L^2}
	\\
	&\leq
	C \norm{\zeta^{k+1}}{L^2}^2 + C \norm{\theta^k}{L^2}^2 + Ch^{2(q+1)} + C\tau^2 + \alpha \norm{\theta^{k+1}}{L^2}^2.
\end{align*}
Similarly, for the term $F_4$ we infer
\begin{align*}
	\abs{F_4}
	&\leq
	C\norm{\theta^{k+1}}{L^2}^2 + Ch^{2(q+1)} + \alpha \norm{\xi^{k+1}}{L^2}^2.
\end{align*}
Letting $F := \sum_{i=1}^4 F_i$, we deduce from the above estimates that for any $\alpha>0$,
\begin{align}\label{equ:est F}
	\abs{F}
	&\leq
	C\left(1+\norm{\mu_h^{k+1}}{H^1}^2\right) \left(h^{2(q+1)}+ \tau^2+ \norm{\theta^k}{L^2}^2 \right)
	\nonumber\\
	&\quad
	+
	C \norm{\theta^{k+1}}{L^2}^2
	+
	C \norm{\zeta^{k+1}}{L^2}^2
	+
	\alpha \norm{\nabla \theta^{k+1}}{L^2}^2.
\end{align}
\\
\underline{Estimates for $G_1$ to $G_8$}: We estimate the terms $G_1$ to $G_5$ in a straightforward manner by Young's inequality and \eqref{equ:Ritz ineq}, leading to
\begin{align*}
	\abs{G_1}
	&\leq
	{C\norm{\eta^{k+1}}{L^2}^2 + \alpha \norm{\xi^{k+1}}{L^2}^2}
	\leq
	{Ch^{2(q+1)}} + \alpha \norm{\xi^{k+1}}{L^2}^2,
	\\
	\abs{G_2}
	&\leq
	C\norm{\theta^{k+1}}{L^2}^2 + \alpha \norm{\xi^{k+1}}{L^2}^2,
	\\
	\abs{G_3}
	&\leq
	{C\norm{\rho^{k+1}}{L^2}^2 + \alpha \norm{\xi^{k+1}}{L^2}^2 + \alpha \norm{\Delta_h \theta^{k+1}}{L^2}^2}
	\leq
	{Ch^{2(q+1)}} + \alpha \norm{\xi^{k+1}}{L^2}^2 + \alpha \norm{\Delta_h \theta^{k+1}}{L^2}^2,
	\\
	\abs{G_4}
	&\leq
	{C \norm{\eta^{k+1}}{L^2}^2 + \alpha \norm{\Delta_h \theta^{k+1}}{L^2}^2}
	\leq
	{Ch^{2(q+1)}} + \alpha \norm{\Delta_h \theta^{k+1}}{L^2}^2,
	\\
	\abs{G_5}
	&\leq
	{C\norm{\zeta^{k+1}}{L^2}^2 + C\norm{\omega^{k+1}}{L^2}^2 + \alpha \norm{\xi^{k+1}}{L^2}^2 +  \alpha \norm{\Delta_h \theta^{k+1}}{L^2}^2}
	\\
	&\leq
	C\norm{\zeta^{k+1}}{L^2}^2 + {Ch^{2(q+1)}} + \alpha \norm{\xi^{k+1}}{L^2}^2 +  \alpha \norm{\Delta_h \theta^{k+1}}{L^2}^2.
\end{align*}
The term $G_7$ is estimated using Young's inequality as
\begin{align*}
	\abs{G_7}
	&\leq
	C\abs{e^{k+1}}^2 + \alpha \norm{\xi^{k+1}}{L^2}^2
	+
	\alpha \norm{\Delta_h \theta^{k+1}}{L^2}^2.
\end{align*}
To estimate the terms $G_6$ and $G_8$, we apply Lemma~\ref{lem:rf frac est} and Young's inequality to obtain
\begin{align*}
	\abs{G_6}
	&\leq
	C \norm{b_h^k-b^k}{L^2}^2
	+
	\alpha \norm{\xi^{k+1}}{L^2}^2
	+
	\alpha \norm{\Delta_h \theta^{k+1}}{L^2}^2
	\\
	&\leq
	C\norm{\theta^k}{L^2}^2 + Ch^{2(q+1)}
	+
	\alpha \norm{\xi^{k+1}}{L^2}^2
	+
	\alpha \norm{\Delta_h \theta^{k+1}}{L^2}^2,
	\\
	\abs{G_8}
	&\leq
	C \norm{b^k-b^{k+1}}{L^2}^2
	+
	\alpha \norm{\xi^{k+1}}{L^2}^2
	+
	\alpha \norm{\Delta_h \theta^{k+1}}{L^2}^2
	\\
	&\leq
	C\tau^2 
	+ \alpha \norm{\xi^{k+1}}{L^2}^2
	+
	\alpha \norm{\Delta_h \theta^{k+1}}{L^2}^2.
\end{align*}
Altogether, denoting $G:= \sum_{i=1}^8 G_i$, we deduce that for any $\alpha>0$,
\begin{align}\label{equ:est G}
	\abs{G} &\leq
	C\abs{e^k}^2 + C\norm{\theta^k}{L^2}^2 +  C\norm{\theta^{k+1}}{L^2}^2 + C\norm{\zeta^{k+1}}{L^2}^2 + Ch^{2(q+1)} + C\tau^2
	\nonumber\\
	&\quad
	+ \alpha \norm{\xi^{k+1}}{L^2}^2
	+
	\alpha \norm{\Delta_h \theta^{k+1}}{L^2}^2.
\end{align}
\\
\underline{Estimates for $H_1$ to $H_6$}: The term $H_1$ is estimated using Young's inequality, \eqref{equ:dt vn Lp}, and \eqref{equ:dt vn min par v} as
\begin{align*}
	\abs{H_1}
	&\leq
	Ch^{2(q+1)} + C\tau^2 + C\norm{\zeta^{k+1}}{L^2}^2.
\end{align*}
For $H_2$, we apply \eqref{equ:nab disc lap}, H\"older's, and Young's inequalities to obtain for any $\alpha>0$,
\begin{align*}
	\abs{H_2}
	\leq
	C \norm{\nabla\theta^k}{L^2} \norm{\nabla \zeta^{k+1}}{L^2}
	&\leq
	C \norm{\theta^k}{L^2}^{\frac12} \norm{\Delta_h \theta^k}{L^2}^{\frac12} \norm{\nabla \zeta^{k+1}}{L^2}
	\\
	&\leq
	C\norm{\theta^k}{L^2}^2 + \alpha \norm{\Delta_h \theta^k}{L^2}^2 + \alpha \norm{\nabla \zeta^{k+1}}{L^2}^2.
\end{align*}
By Young's inequality and \eqref{equ:dt vn Lp}, for the term $H_3$ we have
\begin{align*}
	\abs{H_3}
	&\leq
	C\tau^2 +\alpha \norm{\nabla \zeta^{k+1}}{L^2}^2.
\end{align*}
The terms $H_4$, $H_5$, and $H_6$ can be estimated in a similar manner as the terms $F_3$, $F_4$, and $F_5$, respectively. Therefore, we obtain
\begin{align*}
	\abs{H_4}
	&\leq
	C\left(1+\norm{\mu_h^{k+1}}{H^1}^2\right) \left(h^{2(q+1)}+ \tau^2+ \norm{\theta^k}{L^2}^2 \right)
	+
	\alpha \norm{\zeta^{k+1}}{H^1}^2,
	\\
	\abs{H_5}
	&\leq
	C \norm{\zeta^{k+1}}{L^2}^2 + C \norm{\theta^k}{L^2}^2 + Ch^{2(q+1)} + C\tau^2,
	\\
	\abs{H_6}
	&\leq
	C\norm{\zeta^{k+1}}{L^2}^2 + Ch^{2(q+1)} + \alpha \norm{\xi^{k+1}}{L^2}^2.
\end{align*}
To summarise, letting $H:= \sum_{i=1}^6 H_i$, we deduce
\begin{align}\label{equ:est H}
	\abs{H} &\leq
	C\left(1+\norm{\mu_h^{k+1}}{H^1}^2\right) \left(h^{2(q+1)}+ \tau^2+ \norm{\theta^k}{L^2}^2 \right)
	+
	C \norm{\theta^{k+1}}{L^2}^2
	+
	C \norm{\zeta^{k+1}}{L^2}^2
	\nonumber\\
	&\quad
	+
	\alpha \norm{\xi^{k+1}}{L^2}^2
	+
	\alpha \norm{\nabla \zeta^{k+1}}{L^2}^2
	+
	\alpha \norm{\Delta_h \theta^k}{L^2}^2.
\end{align}
\\
\underline{Estimates for $I_1$ to $I_6$}: By Young's inequality and \eqref{equ:dt vn min par v}, the term $I_1$ can be estimated as
\begin{align*}
	\abs{I_1} &\leq
	C\tau^2 + C\abs{e^{k+1}}^2.
\end{align*}
For the term $I_2$, by H\"older's and Young's inequality, noting \eqref{equ:f prime xy} and \eqref{equ:stab u Linfty}, we have
\begin{align*}
	\abs{I_2}
	&\leq
	C\abs{e^{k+1}} \norm{f'(u_h^k)-f'(u^k)}{L^2} \norm{\dtt u_h^{k+1}}{L^2}
	\\
	&\leq
	C\abs{e^{k+1}} \left(1+\norm{u_h^k}{L^\infty}^2 + \norm{u^k}{L^\infty}^2\right) \norm{\theta^k+\rho^k}{L^2} \norm{\dtt u_h^{k+1}}{L^2} 
	\\
	&\leq
	C\norm{\dtt u_h^{k+1}}{L^2}^2 \left(h^{2(q+1)} + \norm{\theta^k}{L^2}^2\right) + C \abs{e^{k+1}}^2.
\end{align*}
Next, for the term $I_3$, we apply \eqref{equ:dt theta H-2} and H\"older's inequality to obtain
\begin{align*}
	\abs{I_3}
	&\leq
	C \abs{e^{k+1}} \norm{f'(u^k)}{H^2} \norm{\dtt \theta^{k+1}+ \dtt \rho^{k+1}}{\widetilde{H}^{-2}}
	\\
	&\leq
	C \abs{e^{k+1}} \left(\norm{\theta^k}{H^1}+ \norm{\zeta^{k+1}}{L^2} + \norm{\xi^{k+1}}{L^2} + h^{q+1} + \tau\right)
	\\
	&\leq
	C\abs{e^{k+1}}^2 + C\norm{\theta^k}{L^2}^2 + \alpha \norm{\Delta_h \theta^k}{L^2}^2
	+
	\alpha \norm{\zeta^{k+1}}{L^2}^2
	+
	\alpha \norm{\xi^{k+1}}{L^2}^2
	+
	Ch^{2(q+1)} + C\tau^2,
\end{align*}
where in the last step we also used \eqref{equ:nab disc lap} and Young's inequality. For the term $I_4$, by \eqref{equ:dt vn min par v} and Young's inequality, we infer
\begin{align*}
	\abs{I_4}
	&\leq
	C\abs{e^{k+1}}^2 + C\tau^2.
\end{align*}
For the term $I_5$, by H\"older's and Young's inequality, noting the definition of $\beta_h^k$ and $\beta^k$, we have
\begin{align*}
	\abs{I_5}
	&\leq
	\frac12 \abs{e^{k+1}} \abs{\frac{1}{\sqrt{\beta_h^k}}- \frac{1}{\sqrt{\beta^k}}} \norm{f'(u^k)}{L^2} \norm{\partial_t u^{k+1}}{L^2}
	\\
	&\leq
	C\abs{e^{k+1}} \norm{f(u^k)-f(u_h^k)}{L^1} \norm{f'(u^k)}{L^2} \norm{\partial_t u^{k+1}}{L^2}
	\\
	&\leq
	C \abs{e^{k+1}}^2 + C\norm{\theta^k}{L^2}^2 + Ch^{2(q+1)}.
\end{align*}
Finally, for $I_6$, we apply \eqref{equ:frac rk f min rk1 f} and Young's inequality to obtain
\begin{align*}
	\abs{I_6}
	&\leq
	C \abs{e^{k+1}}^2+ C\tau^2.
\end{align*}
Altogether, letting $I:= \sum_{i=1}^6 I_i$, we have
\begin{align}\label{equ:est I}
	\abs{I} 
	&\leq
	C\abs{e^k}^2 + C\abs{e^{k+1}}^2 
	+
	C\left(1+\norm{\dtt u_h^{k+1}}{L^2}^2 \right) \left(h^{2(q+1)}+ \norm{\theta^k}{L^2}^2\right) 
	+
	C\tau^2
	\nonumber\\
	&\quad
	+
	\alpha \norm{\Delta_h \theta^k}{L^2}^2
	+
	\alpha \norm{\zeta^{k+1}}{L^2}^2
	+
	\alpha \norm{\xi^{k+1}}{L^2}^2.
\end{align}
\\
\underline{{Derivation of \eqref{equ:error theta n eul}}}: Now, putting together the estimates derived in \eqref{equ:est F}, \eqref{equ:est G}, \eqref{equ:est H}, and \eqref{equ:est I}, choosing $\alpha>0$ sufficiently small to absorb relevant terms, we obtain
\begin{align*}
	&\frac{1}{\tau} \left(\norm{\theta^{k+1}}{L^2}^2 - \norm{\theta^k}{L^2}^2 \right)
	+
	\frac{1}{\tau} \left(\norm{\zeta^{k+1}}{L^2}^2 - \norm{\zeta^k}{L^2}^2 \right)
	+
	\frac{1}{\tau} \left(\abs{e^{k+1}}^2 - \abs{e^k}^2 \right) 
	\\
	&\quad
	+
	\norm{\nabla \theta^{k+1}}{L^2}^2
	+
	\norm{\Delta_h \theta^{k+1}}{L^2}^2
	+
	\norm{\xi^{k+1}}{L^2}^2
	+
	\norm{\nabla \zeta^{k+1}}{L^2}^2
	\\
	&\leq
	C\left(1+\norm{\mu_h^{k+1}}{H^1}^2 +\norm{\dtt u_h^{k+1}}{L^2}^2 \right) \left(h^{2(q+1)}+ \tau^2+ \norm{\theta^k}{L^2}^2 \right)
	+
	C \norm{\theta^{k+1}}{L^2}^2
	\nonumber\\
	&\quad
	+ 
	C\abs{e^k}^2 + C\abs{e^{k+1}}^2 + C\norm{\zeta^{k+1}}{L^2}^2 + \alpha \norm{\Delta_h \theta^k}{L^2}^2.
\end{align*}
Summing the above inequality over $j\in \{0,1,\ldots,k-1\}$ and noting $\alpha>0$ can be chosen sufficiently small, we obtain
\begin{align*}
	&\norm{\theta^k}{L^2}^2
	+
	\norm{\zeta^k}{L^2}^2
	+
	\abs{e^k}^2
	+
	\tau \sum_{j=1}^k \left(\norm{\nabla \theta^j}{L^2}^2
	+
	\norm{\Delta_h \theta^j}{L^2}^2
	+
	\norm{\xi^j}{L^2}^2
	+
	\norm{\nabla \zeta^j}{L^2}^2 \right)
	\\
	&\leq
	C\tau \sum_{j=1}^k \left(1+\norm{\mu_h^j}{H^1}^2 +\norm{\dtt u_h^j}{L^2}^2 \right) \left(h^{2(q+1)}+ \tau^2+ \norm{\theta^{j-1}}{L^2}^2 \right)
	+
	C\tau \sum_{j=1}^k \abs{e^{j-1}}^2
	\\
	&\quad
	+
	C_0 \tau \sum_{j=1}^k \left(\norm{\theta^j}{L^2}^2 + \norm{\zeta^j}{L^2}^2 + \abs{e^j}^2 \right),
\end{align*}
where $C_0$ is a constant independent of $k$, $h$, and $\tau$, explicitly highlighted here.
By the discrete Gronwall lemma, noting the estimates \eqref{equ:stab sum mu sigma H1} and \eqref{equ:stab Delta u L2}, we deduce that for $\tau<1/C_0$,
\begin{align*}
	&\norm{\theta^k}{L^2}^2
	+
	\norm{\zeta^k}{L^2}^2
	+
	\abs{e^k}^2
	+
	\tau \sum_{j=1}^k \left(\norm{\nabla \theta^j}{L^2}^2
	+
	\norm{\Delta_h \theta^j}{L^2}^2
	+
	\norm{\xi^j}{L^2}^2
	+
	\norm{\nabla \zeta^j}{L^2}^2 \right)
	\nonumber\\
	&\leq
	C \exp\left[ \tau \sum_{j=1}^k \left(1+\norm{\mu_h^j}{H^1}^2 +\norm{\dtt u_h^j}{L^2}^2 \right) \right] (h^{2(q+1)}+\tau^2),
\end{align*}
which implies \eqref{equ:error theta n eul}, {thanks to \eqref{equ:stab Delta u L2} and \eqref{equ:stab u Linfty}}.

Finally, the error estimate \eqref{equ:error euler L2} follows at once by noting the splitting \eqref{equ:split uk}, \eqref{equ:split n}, \eqref{equ:split rk}, together with the auxiliary error estimate \eqref{equ:error theta n eul} established above and \eqref{equ:Ritz ineq}.
\end{proof}

To derive an estimate for $\norm{\xi^k}{L^2}$, we require the following lemma.

\begin{lemma}
	Let $k\in\{1,2,\ldots,\lfloor T/\tau\rfloor\}$. Let $\beta_h^k$, $b_h^k$, $\beta^k$, and $b^k$ be defined by \eqref{equ:phk bhk} and \eqref{equ:beta k bk}. For any $\alpha>0$, there exists a constant $C_\alpha>0$, independent of $h$, $\tau$, and $k$, such that for all $\psi\in V_h$,
	\begin{align}
		\label{equ:pre est N5}
		&\abs{
			\inpro{
				\frac{1}{\tau}
				\left(
				\frac{r_h^{k+1}}{\sqrt{\beta_h^k}}
				-
				\frac{r_h^k}{\sqrt{\beta_h^{k-1}}}
				\right)
				f'(u_h^k)
			}{\psi}}
		\nonumber\\
		&\qquad
		\leq
		C_\alpha\norm{\psi}{L^2}^2
		+
		\alpha h^{2(q+1)}
		+
		\alpha \tau^2
		+
		\alpha\left(
		\norm{\dtt\theta^{k+1}}{L^2}^2
		+
		\norm{\dtt\theta^k}{L^2}^2
		\right),
	\end{align}
	and
	\begin{align}
		\label{equ:pre est N6}
		&\abs{
			\inpro{
				\left(\frac{r_h^k}{\sqrt{\beta_h^{k-1}}}\right)
				\left(
				\frac{f'(u_h^k)-f'(u_h^{k-1})}{\tau}
				\right)
				-
				\left(
				\frac{r^{k+1}}{\sqrt{\beta^{k+1}}}
				\right)
				f''(u^{k+1})\partial_t u^{k+1}
			}{\psi}}
		\nonumber\\
		&\qquad
		\leq
		C_\alpha \norm{\psi}{L^2}^2
		+
		\alpha h^{2(q+1)}
		+
		\alpha \tau^2
		+
		\alpha\norm{\dtt\theta^k}{L^2}^2.
	\end{align}
\end{lemma}

\begin{proof}
	We first prove \eqref{equ:pre est N5}. By \eqref{equ:fem euler r}, we have
	\begin{align}
		\frac{1}{\tau}
		\left(
		\frac{r_h^{k+1}}{\sqrt{\beta_h^k}}
		-
		\frac{r_h^k}{\sqrt{\beta_h^{k-1}}}
		\right)
		&=
		\frac{\dtt r_h^{k+1}}{\sqrt{\beta_h^k}}
		+
		\frac{1}{\tau}
		\left(
		\frac{1}{\sqrt{\beta_h^k}}
		-
		\frac{1}{\sqrt{\beta_h^{k-1}}}
		\right) r_h^k
		\nonumber\\
		&\quad
		=
		\frac{1}{2\beta_h^k}
		\inpro{f'(u_h^k)}{\dtt u_h^{k+1}}
		-
		\frac{r_h^k}{\sqrt{\beta_h^k\beta_h^{k-1}}
			(\sqrt{\beta_h^k}+\sqrt{\beta_h^{k-1}})}
		\dtt\beta_h^k .
		\label{equ:ratio-diff-rh}
	\end{align}
	Using the integral form of the mean-value theorem, we write
	\[
	\dtt\beta_h^k
	=
	\inpro{\overline f_h^k}{\dtt u_h^k},
	\qquad \text{where}\quad
	\overline f_h^k
	:=
	\int_0^1
	f'\big(u_h^{k-1}+s(u_h^k-u_h^{k-1})\big)\,\ds .
	\]
	Hence \eqref{equ:ratio-diff-rh} becomes
	\begin{align}
		&\frac{1}{\tau}
		\left(
		\frac{r_h^{k+1}}{\sqrt{\beta_h^k}}
		-
		\frac{r_h^k}{\sqrt{\beta_h^{k-1}}}
		\right)
		=
		\frac{1}{2\beta_h^k}
		\inpro{f'(u_h^k)}{\dtt u_h^{k+1}}
		-
		\frac{r_h^k}{\sqrt{\beta_h^k\beta_h^{k-1}}
			(\sqrt{\beta_h^k}+\sqrt{\beta_h^{k-1}})}
		\inpro{\overline f_h^k}{\dtt u_h^k}.
		\label{equ:ratio-diff-rh-2}
	\end{align}
	We now estimate the two terms on the right-hand side of
	\eqref{equ:ratio-diff-rh-2} separately. Adding and subtracting a common exact
	quantity
	\[
	\frac{1}{2\beta^k}
	\inpro{f'(u^k)}{\partial_t u^k},
	\]
	we obtain by the triangle inequality,
	\begin{align}
		&\abs{
			\frac{1}{\tau}
			\left(
			\frac{r_h^{k+1}}{\sqrt{\beta_h^k}}
			-
			\frac{r_h^k}{\sqrt{\beta_h^{k-1}}}
			\right)}
		\nonumber\\
		&\quad
		\leq
		\abs{
			\frac{1}{2\beta_h^k}
			\inpro{f'(u_h^k)}{\dtt u_h^{k+1}}
			-
			\frac{1}{2\beta^k}
			\inpro{f'(u^k)}{\partial_t u^k}
		}
		\nonumber\\
		&\qquad
		+
		\abs{
			\frac{r_h^k}{\sqrt{\beta_h^k\beta_h^{k-1}}
				(\sqrt{\beta_h^k}+\sqrt{\beta_h^{k-1}})}
			\inpro{\overline f_h^k}{\dtt u_h^k}
			-
			\frac{1}{2\beta^k}
			\inpro{f'(u^k)}{\partial_t u^k}
		}
		\nonumber\\
		&\quad
		=: J_1+J_2 .
		\label{equ:J1-J2-ratio}
	\end{align}
	For the first term, we have by adding and subtracting common terms,
	\begin{align}
		J_1
		&\leq
		\abs{
			\frac{1}{2\beta_h^k}
			-
			\frac{1}{2\beta^k}
		}
		\abs{\inpro{f'(u_h^k)}{\dtt u_h^{k+1}}}
		\nonumber\\
		&\quad
		+
		C\abs{
			\inpro{f'(u_h^k)-f'(u^k)}{\dtt u_h^{k+1}}
		}
		+
		C\abs{
			\inpro{f'(u^k)}{\dtt u_h^{k+1}-\partial_t u^k}
		}.
		\label{equ:J1-est-decomp}
	\end{align}
	Now, observe that we have
	\[
	\abs{\beta_h^k-\beta^k} + \norm{f'(u_h^k)-f'(u^k)}{L^2}
	\leq
	C\norm{u_h^k-u^k}{L^2}
	\leq
	C(h^{q+1}+\tau),
	\]
	as well as \eqref{equ:stab u Linfty}, \eqref{equ:stab dtu L2}, \eqref{equ:dt vn Lp}, and \eqref{equ:dt vn min par v}. Thus, we deduce from \eqref{equ:J1-est-decomp} that
	\begin{align}
		J_1
		&\leq
		C(h^{q+1}+\tau)
		+
		C\norm{\dtt u_h^{k+1}-\partial_t u^k}{L^2}
		\nonumber\\
		&\leq
		C(h^{q+1}+\tau)
		+
		C\norm{\dtt\theta^{k+1}+
			\dtt\rho^{k+1}
			+
			\dtt u^{k+1}-\partial_t u^k}{L^2}
		\nonumber\\
		&\leq
		C(h^{q+1}+\tau)
		+
		C\norm{\dtt\theta^{k+1}}{L^2}.
		\label{equ:J1-est}
	\end{align}
	We next estimate $J_2$. First, we note the following fact: the map
	\[
	G(r,a,b):=\frac{r}{\sqrt{ab}(\sqrt a+\sqrt b)}
	\]
	is Lipschitz on the range where $a,b$ are bounded away from zero and $r$ is bounded. We also observe that
	$G(r^k,\beta^k,\beta^k)=1/(2\beta^k)$. Therefore,
	\begin{align}
		&\abs{
			\frac{r_h^k}{\sqrt{\beta_h^k\beta_h^{k-1}}
				(\sqrt{\beta_h^k}+\sqrt{\beta_h^{k-1}})}
			-
			\frac{1}{2\beta^k}
		}
		\nonumber\\
		&=
		\abs{G(r_h^k,\beta_h^k,\beta_h^{k-1})- G(r^k,\beta^k,\beta^k)}
		\nonumber\\
		&
		\leq
		C\abs{r_h^k-r^k}
		+
		C\abs{\beta_h^k-\beta^k}
		+
		C\abs{\beta_h^{k-1}-\beta^{k-1}}
		+
		C\abs{\beta^k-\beta^{k-1}}
		\nonumber\\
		&
		\leq
		C(h^{q+1}+\tau)
		\label{equ:coef-J2}
	\end{align}
	by \eqref{equ:error euler L2}.
	Moreover, by the definition of $\overline f_h^k$ and the mean-value theorem,
	\begin{align}
		\norm{\overline f_h^k-f'(u^k)}{L^2}
		&\leq
		C\norm{u_h^k-u^k}{L^2}
		+
		C\norm{u_h^{k-1}-u^k}{L^2}
		\leq
		C(h^{q+1}+\tau).
		\label{equ:fbar-est}
	\end{align}
	Thus, we have the estimate
	\begin{align}
		J_2
		&\leq
		C(h^{q+1}+\tau)
		\abs{\inpro{\overline f_h^k}{\dtt u_h^k}}
		+
		C\abs{\inpro{\overline f_h^k-f'(u^k)}{\dtt u_h^k}}
		+
		C\abs{\inpro{f'(u^k)}{\dtt u_h^k-\partial_t u^k}}
		\nonumber\\
		&\leq
		C(h^{q+1}+\tau)
		+
		C\norm{\dtt u_h^k-\partial_t u^k}{L^2}
		\nonumber\\
		&\leq
		C(h^{q+1}+\tau)
		+
		C\norm{\dtt\theta^k}{L^2}.
		\label{equ:J2-est}
	\end{align}
	Combining \eqref{equ:J1-J2-ratio}, \eqref{equ:J1-est}, and
	\eqref{equ:J2-est}, we obtain
	\begin{align}
		\label{equ:ratio-pre-est}
		\abs{
			\frac{1}{\tau}
			\left(
			\frac{r_h^{k+1}}{\sqrt{\beta_h^k}}
			-
			\frac{r_h^k}{\sqrt{\beta_h^{k-1}}}
			\right)}
		\leq
		C(h^{q+1}+\tau)
		+
		C\norm{\dtt\theta^{k+1}}{L^2}
		+
		C\norm{\dtt\theta^k}{L^2}.
	\end{align}
	By \eqref{equ:stab u Linfty}, we also have
	$\norm{f'(u_h^k)}{L^\infty}\le C$. Inequality \eqref{equ:pre est N5} then follows by an application of Young's inequality.

	We now prove \eqref{equ:pre est N6}. Again by the integral form of the mean-value theorem,
	\[
	\frac{f'(u_h^k)-f'(u_h^{k-1})}{\tau}
	=
	\overline f_{h,2}^k\, \dtt u_h^k, \qquad \text{where} \quad 
	\overline f_{h,2}^k
	:=
	\int_0^1
	f''\big(u_h^{k-1}+s(u_h^k-u_h^{k-1})\big)\,\ds .
	\]
	We first record the following estimate.
	Adding and subtracting
	$\overline f_{h,2}^k\partial_tu^{k+1}$, we have
	\begin{align*}
		&\norm{
			\frac{f'(u_h^k)-f'(u_h^{k-1})}{\tau}
			-
			f''(u^{k+1})\partial_tu^{k+1}
		}{L^2}
		\nonumber\\
		&=
		\norm{
			\overline f_{h,2}^k \dtt u_h^k
			-
			f''(u^{k+1})\partial_tu^{k+1}
		}{L^2}
		\nonumber\\
		&
		\leq
		\norm{\overline f_{h,2}^k}{L^\infty}
		\norm{\dtt u_h^k-\partial_tu^{k+1}}{L^2}
		+
		\norm{\overline f_{h,2}^k-f''(u^{k+1})}{L^2}
		\norm{\partial_tu^{k+1}}{L^\infty}
		\nonumber\\
		&=: L_1+L_2.
	\end{align*}
	By the uniform $L^\infty$ bound for $u_h^j$ in \eqref{equ:stab u Linfty}, the growth assumption on $f''$, \eqref{equ:dt vn Lp}, and \eqref{equ:dt vn min par v}, we have
	\begin{align}
		L_1
		&\leq
		C\left(\norm{\dtt\theta^k}{L^2}
		+
		\norm{\dtt\rho^k}{L^2}
		+
		\norm{\dtt u^k-\partial_tu^{k+1}}{L^2}\right)
		\leq
		C\norm{\dtt\theta^k}{L^2}
		+
		C(h^{q+1}+\tau).
		\label{equ:N6-dtuh-est}
	\end{align}
	Moreover, by the mean-value theorem, the assumed regularity on $u$, and \eqref{equ:error euler L2}, we have
	\begin{align}
		L_2
		&\leq
		C\norm{u_h^k-u^{k+1}}{L^2}
		+
		C\norm{u_h^{k-1}-u^{k+1}}{L^2}
		\leq
		C(h^{q+1}+\tau).
		\label{equ:N6-fbar2-est}
	\end{align}
	Combining \eqref{equ:N6-dtuh-est}--\eqref{equ:N6-fbar2-est}, we obtain
	\begin{align}
		\label{equ:diff-fprime-est}
		&\norm{
			\frac{f'(u_h^k)-f'(u_h^{k-1})}{\tau}
			-
			f''(u^{k+1})\partial_tu^{k+1}
		}{L^2}
		\leq
		C(h^{q+1}+\tau)
		+
		C\norm{\dtt\theta^k}{L^2}.
	\end{align}
	
	Next, since $r^{k+1}/\sqrt{\beta^{k+1}}=1$ by definition, we have the estimate
	\begin{align}
		\abs{
			\frac{r_h^k}{\sqrt{\beta_h^{k-1}}}
			-
			\frac{r^{k+1}}{\sqrt{\beta^{k+1}}}
		}
		&
		=
		\abs{
			\frac{r_h^k}{\sqrt{\beta_h^{k-1}}}
			-
			\frac{r^k}{\sqrt{\beta^k}}
		}
		\nonumber\\
		&\quad
		\leq
		C\abs{r_h^k-r^k}
		+
		C\abs{\beta_h^{k-1}-\beta^k}
		\nonumber\\
		&\quad
		\leq
		C\abs{r_h^k-r^k}
		+
		C\abs{\beta_h^{k-1}-\beta^{k-1}}
		+
		C\abs{\beta^{k-1}-\beta^k}
		\nonumber\\
		&\quad
		\leq
		C(h^{q+1}+\tau).
		\label{equ:N6-ratio-est}
	\end{align}
	Therefore, by adding and subtracting a common term
	\[
	\frac{r^{k+1}}{\sqrt{\beta^{k+1}}}
	\frac{f'(u_h^k)-f'(u_h^{k-1})}{\tau},
	\]
	we obtain
	\begin{align}
		&\norm{
			\left(\frac{r_h^k}{\sqrt{\beta_h^{k-1}}}\right)
			\left(
			\frac{f'(u_h^k)-f'(u_h^{k-1})}{\tau}
			\right)
			-
			\left(
			\frac{r^{k+1}}{\sqrt{\beta^{k+1}}}
			\right)
			f''(u^{k+1})\partial_tu^{k+1}
		}{L^2}
		\nonumber\\
		&\quad
		\leq
		\abs{
			\frac{r_h^k}{\sqrt{\beta_h^{k-1}}}
			-
			\frac{r^{k+1}}{\sqrt{\beta^{k+1}}}
		}
		\norm{
			\frac{f'(u_h^k)-f'(u_h^{k-1})}{\tau}
		}{L^2}
		\nonumber\\
		&\qquad
		+
		C
		\norm{
			\frac{f'(u_h^k)-f'(u_h^{k-1})}{\tau}
			-
			f''(u^{k+1})\partial_tu^{k+1}
		}{L^2}.
		\label{equ:N6-product-split}
	\end{align}
	Now, we note that by \eqref{equ:stab dtu L2}, the growth assumption on $f''$, and \eqref{equ:stab u Linfty},
	\[
	\norm{
		\frac{f'(u_h^k)-f'(u_h^{k-1})}{\tau}
	}{L^2}
	=
	\norm{\overline f_{h,2}^k\, \dtt u_h^k}{L^2}
	\leq 
	C,
	\]
	Therefore, using \eqref{equ:diff-fprime-est} and \eqref{equ:N6-ratio-est}, we infer from \eqref{equ:N6-product-split} that
	\begin{align*}
		&\norm{
			\left(\frac{r_h^k}{\sqrt{\beta_h^{k-1}}}\right)
			\left(
			\frac{f'(u_h^k)-f'(u_h^{k-1})}{\tau}
			\right)
			-
			\left(
			\frac{r^{k+1}}{\sqrt{\beta^{k+1}}}
			\right)
			f''(u^{k+1})\partial_tu^{k+1}
		}{L^2}
		\leq
		C(h^{q+1}+\tau)
		+
		C\norm{\dtt\theta^k}{L^2}.
	\end{align*}
	We then deduce \eqref{equ:pre est N6} by an application of Young's inequality. The proof is now complete.
\end{proof}

The following proposition shows an auxiliary superconvergence estimate for $\norm{\zeta^k}{H^1}$.

\begin{proposition}
For sufficiently small $h,\tau>0$, we have
\begin{align}\label{equ:error xi k L2}
    \norm{\xi^k}{L^2}^2
    +
    \norm{\nabla \zeta^k}{L^2}^2
    +
    \tau \sum_{j=1}^k \left(\norm{\dtt \theta^j}{L^2}^2 + \norm{\dtt \zeta^j}{L^2}^2 \right)
		\leq C(h^{2(q+1)}+ \tau^2).
\end{align}
In particular, for $d\leq 2$, we have
\begin{align}\label{equ:error zeta Linfty 2d}
    \norm{\zeta^k}{L^\infty}^2 \leq C(h^{2(q+1)}+ \tau^2) \abs{\ln h},
\end{align}
where $C$ depends on $T$, but is independent of $k$, $h$, and $\tau$.
\end{proposition}

\begin{proof}
First, we divide \eqref{equ:mu k1 min mu k} by $\tau$ and subtract the result from the time derivative of \eqref{equ:weak tum mu} at time $t=t_{k+1}$ to obtain
\begin{align}\label{equ:dtt xi kk}
    &\inpro{\dtt \xi^{k+1}+ \dtt \eta^{k+1}+ \dtt \mu^{k+1} - \partial_t \mu^{k+1}}{\psi}
    \nonumber\\
    &=
    \epsilon^2 \inpro{\nabla \dtt \theta^{k+1}+ \nabla \dtt u^{k+1}- \nabla \partial_t u^{k+1}}{\nabla \psi} 
    +
    \lambda \inpro{\dtt \theta^{k+1} + \dtt \rho^{k+1} + \dtt u^{k+1}- \partial_t u^{k+1}}{\psi} 
    \nonumber\\
    &\quad
    -
    \chi_0 \inpro{\dtt \zeta^{k+1}+ \dtt \omega^{k+1} + \dtt n^{k+1}- \partial_t n^{k+1}}{\psi}
    \nonumber\\
    &\quad
    +
    \inpro{\frac{1}{\tau} \left(\frac{r_h^{k+1}}{\sqrt{\beta_h^k}} - \frac{r_h^k}{\sqrt{\beta_h^{k-1}}} \right) f'(u_h^k)}{\psi}
    \nonumber\\
    &\quad
    +
    \inpro{\left(\frac{r_h^k}{\sqrt{\beta_h^{k-1}}}\right) \left(\frac{f'(u_h^k)-f'(u_h^{k-1})}{\tau}\right) - \left(\frac{r^{k+1}}{\sqrt{\beta^{k+1}}} \right) f''(u^{k+1}) \cdot \partial_t u^{k+1}}{\psi}.
\end{align}
Next, we take $\phi=\epsilon^2 \dtt \theta^{k+1}$ in \eqref{equ:fem euler sub u} and $\psi=\xi^{k+1}$ in \eqref{equ:dtt xi kk} to obtain
\begin{align}\label{equ:eps dt theta 2}
    &\epsilon^2 \norm{\dtt \theta^{k+1}}{L^2}^2
    +
    \epsilon^2 \inpro{\nabla \xi^{k+1}}{\nabla \dtt \theta^{k+1}}
    \nonumber\\
    &=
    -\epsilon^2 \inpro{\dtt \rho^{k+1}+\dtt u^{k+1}-\partial_t u^{k+1}}{\dtt \theta^{k+1}}
    +
    \epsilon^2
	\inpro{\big(p_h^k-P(u^{k+1})\big)\cdot (\sigma_h^{k+1}-\mu_h^{k+1})}{\dtt \theta^{k+1}}
	\nonumber\\
	&\quad
	+
	\epsilon^2 \inpro{P(u^{k+1})\cdot \big(\delta^{-1}(\zeta^{k+1}+\omega^{k+1}) - \chi_0(\theta^k+\rho^k+u^k-u^{k+1})\big)}{\dtt \theta^{k+1}}
	\nonumber\\
	&\quad
	- \epsilon^2
    \inpro{P(u^{k+1})\cdot (\xi^{k+1}+ \eta^{k+1})}{\dtt \theta^{k+1}}
    \nonumber\\
    &=: M_1+M_2+M_3+M_4,
\end{align}
and
\begin{align}\label{equ:xi k min xi 2}
    &\frac{1}{2\tau} \left(\norm{\xi^{k+1}}{L^2}^2 - \norm{\xi^k}{L^2}^2 \right)
    +
    \frac{1}{2\tau} \norm{\xi^{k+1}-\xi^k}{L^2}^2
    -
    \epsilon^2 \inpro{\nabla\dtt 
    \theta^{k+1}}{\nabla\xi^{k+1}}
    \nonumber\\
    &=
    -\inpro{\dtt \eta^{k+1}+\dtt \mu^{k+1}- \partial_t \mu^{k+1}}{\xi^{k+1}}
    -
    \epsilon^2 \inpro{\Delta \dtt u^{k+1}-\Delta \partial_t u^{k+1}}{\xi^{k+1}}
    \nonumber\\
    &\quad
    +
    \lambda \inpro{\dtt\theta^{k+1}+ \dtt \rho^{k+1}+\dtt u^{k+1}- \partial_t u^{k+1}}{\xi^{k+1}}
    \nonumber\\
    &\quad
    -
    \chi_0 \inpro{\dtt \zeta^{k+1}+\dtt \omega^{k+1}+ \dtt n^{k+1}- \partial_t n^{k+1}}{\xi^{k+1}}
    \nonumber\\
    &\quad
    +
    \inpro{\frac{1}{\tau} \left(\frac{r_h^{k+1}}{\sqrt{\beta_h^k}} - \frac{r_h^k}{\sqrt{\beta_h^{k-1}}} \right) f'(u_h^k)}{\xi^{k+1}}
    \nonumber\\
    &\quad
    +
    \inpro{\left(\frac{r_h^k}{\sqrt{\beta_h^{k-1}}}\right) \left(\frac{f'(u_h^k)-f'(u_h^{k-1})}{\tau}\right) - \left(\frac{r^{k+1}}{\sqrt{\beta^{k+1}}} \right) f''(u^{k+1}) \cdot \partial_t u^{k+1}}{\xi^{k+1}}
    \nonumber\\
    &=: N_1+N_2+\ldots+ N_6.
\end{align}
Lastly, setting $\varphi=\dtt \zeta^{k+1}$ in \eqref{equ:fem euler sub n}, we obtain
\begin{align}\label{equ:dt zeta min zeta}
    &\norm{\dtt \zeta^{k+1}}{L^2}^2
    +
    \frac{1}{2\delta \tau} \left(\norm{\nabla \zeta^{k+1}}{L^2}^2- \norm{\nabla\zeta^k}{L^2}^2 \right) + 
    \frac{1}{2\delta \tau} \norm{\nabla \zeta^{k+1}-\nabla \zeta^k}{L^2}^2
    \nonumber\\
    &=
    -\inpro{\dtt \omega^{k+1}+\dtt n^{k+1}-\partial_t n^{k+1}}{\dtt \zeta^{k+1}}
    -
    \chi_0 \inpro{\Delta_h \theta^k}{\dtt \zeta^{k+1}}
    -
    \chi_0 \inpro{\Delta u^k-\Delta u^{k+1}}{\dtt \zeta^{k+1}}
    \nonumber\\
    &\quad
    -
    \inpro{\big(p_h^k- p^{k+1}\big) (\sigma_h^{k+1}-\mu_h^{k+1})}{\dtt \zeta^{k+1}}
    \nonumber\\
    &\quad
    -
    \inpro{p^{k+1} \big(\delta^{-1}(\zeta^{k+1}+\omega^{k+1}) - \chi_0(\theta^k+\rho^k+u^k-u^{k+1})\big)}{\dtt \zeta^{k+1}}
    \nonumber\\
    &\quad
	+ \inpro{p^{k+1} (\xi^{k+1}+\eta^{k+1})\big)}{\dtt \zeta^{k+1}}
    \nonumber\\
    &=
    O_1+O_2+\ldots+ O_6.
\end{align}
We now add \eqref{equ:eps dt theta 2}, \eqref{equ:xi k min xi 2}, and \eqref{equ:dt zeta min zeta}, then estimate each term on the right-hand side of these equations.
\\[2ex]
\underline{Estimates for $M_1$ to $M_4$}: By Young's inequality and applying similar argument as in the estimates for $F_1$ in \eqref{equ:12 tau theta u L2}, we have
\begin{align*}
	\abs{M_1} &\leq
	Ch^{2(q+1)}+C\tau^2 + \alpha \norm{\dtt \theta^{k+1}}{L^2}^2.
\end{align*}
For the term $M_2$ we apply Young's inequality, \eqref{equ:disc lapl L infty}, \eqref{equ:Ritz ineq}, as well as the uniform estimate \eqref{equ:stab n H1} to obtain
\begin{align*}
    \abs{M_2} 
    &\leq
    C\norm{\theta^k}{L^\infty} \norm{\sigma_h^{k+1}+\mu_h^{k+1}}{L^2} \norm{\dtt \theta^{k+1}}{L^2}
    \\
    &\quad
    +
    C\norm{\rho^k+u^k-u^{k+1}}{L^4}
    \norm{\sigma_h^{k+1}+\mu_h^{k+1}}{L^4} \norm{\dtt \theta^{k+1}}{L^2}
    \\
    &\leq
    C\norm{\theta^k}{H^1}^2 + \alpha \norm{\Delta_h \theta^k}{L^2}^2
    +
    \alpha \norm{\dtt \theta^{k+1}}{L^2}^2
    +
    C\left(1+\norm{\mu_h^{k+1}}{H^1}^2\right) (h^{2(q+1)}+ \tau^2).
\end{align*}
For the term $M_3$, noting \eqref{equ:P sublin}, by the same argument as in the estimate for $F_3$ in \eqref{equ:12 tau theta u L2}, we have
\begin{align*}
	\abs{M_3}
	&\leq
	C\norm{P(u^{k+1})}{L^\infty} \left(\norm{\zeta^{k+1}+\omega^{k+1}}{L^2} + \norm{\theta^k+\rho^k+u^k-u^{k+1}}{L^2} \right) \norm{\dtt \theta^{k+1}}{L^2}
	\\
	&\leq
	C \norm{\zeta^{k+1}}{L^2}^2 + C \norm{\theta^k}{L^2}^2 + Ch^{2(q+1)} + C\tau^2 + \alpha \norm{\dtt \theta^{k+1}}{L^2}^2.
\end{align*}
Similarly, for the term $M_4$ we infer
\begin{align*}
	\abs{M_4}
	&\leq
	C\norm{\xi^{k+1}}{L^2}^2 + Ch^{2(q+1)} + \alpha \norm{\dtt \theta^{k+1}}{L^2}^2.
\end{align*}
\\
\underline{Estimates for $N_1$ to $N_6$}: For the terms $N_1$ to $N_4$, we apply Young's inequality and follow a standard argument to obtain
\begin{align*}
    \abs{N_1}
    &\leq
    Ch^{2(q+1)}+ C\tau^2 + C\norm{\xi^{k+1}}{L^2}^2,
    \\
    \abs{N_2}
    &\leq
    C\tau^2 + C\norm{\xi^{k+1}}{L^2}^2,
    \\
    \abs{N_3}
    &\leq
    Ch^{2(q+1)}+ C\tau^2 + C\norm{\xi^{k+1}}{L^2}^2 + \alpha \norm{\dtt \theta^{k+1}}{L^2}^2,
    \\
    \abs{N_4}
    &\leq
    Ch^{2(q+1)} + C\tau^2 + C \norm{\xi^{k+1}}{L^2}^2 + \alpha \norm{\dtt \zeta^{k+1}}{L^2}^2.
\end{align*}
The terms $N_5$ and $N_6$ can be bounded by using \eqref{equ:pre est N5} and \eqref{equ:pre est N6}, respectively, giving
\begin{align*}
    \abs{N_5}
    &\leq
    Ch^{2(q+1)} + C\tau^2 + C \norm{\xi^{k+1}}{L^2}^2 + \alpha \norm{\dtt \theta^k}{L^2}^2+ \alpha \norm{\dtt \theta^{k+1}}{L^2}^2,
    \\
    \abs{N_6}
    &\leq
    Ch^{2(q+1)}+ C\tau^2 + C\norm{\xi^{k+1}}{L^2}^2 + \alpha \norm{\dtt \theta^k}{L^2}^2.
\end{align*}
\\
\underline{Estimates for $O_1$ to $O_6$}: These terms can be estimated by applying a standard argument, thus we obtain
\begin{align*}
    \abs{O_1}
    &\leq
    Ch^{2(q+1)} + C\tau^2 + \alpha \norm{\dtt \zeta^{k+1}}{L^2}^2,
    \\
    \abs{O_2}
    &\leq
    C\norm{\Delta_h \theta^k}{L^2}^2
    +
    \alpha \norm{\dtt \zeta^{k+1}}{L^2}^2,
    \\
    \abs{O_3}
    &\leq
    C\tau^2 + \alpha \norm{\dtt \zeta^{k+1}}{L^2}^2,
    \\
    \abs{O_4}
    &\leq
    C\norm{\theta^k}{H^1}^2 + \alpha \norm{\Delta_h \theta^k}{L^2}^2
    +
    \alpha \norm{\dtt \zeta^{k+1}}{L^2}^2
    +
    C\left(1+\norm{\mu_h^{k+1}}{H^1}^2\right) (h^{2(q+1)}+ \tau^2),
    \\
    \abs{O_5}
    &\leq
    C\norm{\zeta^{k+1}}{L^2}^2
    +
    C\norm{\theta^k}{L^2}^2
    +
    Ch^{2(q+1)} + C\tau^2
    +
    \alpha \norm{\dtt \zeta^{k+1}}{L^2}^2,
    \\
    \abs{O_6}
    &\leq
    Ch^{2(q+1)}
    +
    C\norm{\xi^{k+1}}{L^2}^2
    +
    \alpha \norm{\dtt \zeta^{k+1}}{L^2}^2.
\end{align*}
\\
\underline{Conclusion of the proof}: Collecting all the above estimates, choosing $\alpha>0$ sufficiently small to absorb relevant terms, and summing the resulting inequality over $j\in \{0,1,\ldots,k-1\}$, we obtain
\begin{align*}
    &\norm{\xi^k}{L^2}^2
    +
    \norm{\nabla \zeta^k}{L^2}^2
    +
    \tau \sum_{j=1}^k \left(\norm{\dtt \theta^j}{L^2}^2 + \norm{\dtt \zeta^j}{L^2}^2 \right)
    \\
    &\leq
    C(h^{2(q+1)}+\tau^2)\, \tau \sum_{j=1}^k \left(1+\norm{\mu_h^j}{H^1}^2\right)
    +
    C\tau \sum_{j=1}^k \left(\norm{\theta^{j-1}}{H^1}^2 + \norm{\Delta_h \theta^{j-1}}{L^2}^2 + \norm{\xi^j}{L^2}^2
    + \norm{\zeta^j}{L^2}^2 \right)
    \\
    &\leq
    C(h^{2(q+1)}+\tau^2),
\end{align*}
where in the last step we used \eqref{equ:stab sum mu sigma H1} and \eqref{equ:error theta n eul}, thus proving \eqref{equ:error xi k L2}. Finally, \eqref{equ:error zeta Linfty 2d} follows from \eqref{equ:error theta n eul}, \eqref{equ:error xi k L2}, and the discrete Sobolev inequality \eqref{equ:disc sob 2d}.
\end{proof}

In the following proposition, we establish an auxiliary superconvergence estimate for $\norm{\nabla\theta^k}{L^2}$ and $\norm{\Delta_h \theta^k}{L^2}$, which in turn yields an estimate for $\norm{\theta^k}{L^\infty}$.

\begin{proposition}
For sufficiently small $h,\tau>0$, we have
\begin{align}\label{equ:error Delta theta L2}
    \norm{\nabla \theta^k}{L^2}^2
    +
    \norm{\Delta_h \theta^k}{L^2}^2
    +
    \tau \sum_{j=1}^k \norm{\Delta_h \xi^j}{L^2}^2
    \leq
    C(h^{2(q+1)}+ \tau^2).
\end{align}
In particular, we have for $d\leq 3$,
\begin{align}\label{equ:error theta Linfty}
    \norm{\theta^k}{L^\infty}^2 \leq
    C(h^{2(q+1)}+ \tau^2),
\end{align}
where $C$ depends on $T$, but is independent of $k$, $h$, and $\tau$.
\end{proposition}

\begin{proof}
Recall the notations in \eqref{equ:phk bhk} and \eqref{equ:beta k bk}.
We set $\psi=\Delta_h \theta^{k+1}$ in \eqref{equ:fem euler sub mu} to obtain
\begin{align}\label{equ:Delta th L2}
    &\epsilon^2 \norm{\Delta_h \theta^{k+1}}{L^2}^2
    +
    \lambda \norm{\nabla \theta^{k+1}}{L^2}^2
    \nonumber\\
    &=
    - \inpro{\xi^{k+1}+\eta^{k+1}}{\Delta_h \theta^{k+1}}
    +
    \lambda \inpro{\rho^{k+1}}{\Delta_h \theta^{k+1}}
    -
    \chi_0 \inpro{\zeta^{k+1}+\omega^{k+1}}{\Delta_h \theta^{k+1}}
    \nonumber\\
    &\quad
    +
		r_h^{k+1} \inpro{b_h^k-b^k}{\Delta_h \theta^{k+1}}
		+ 
		e^k \inpro{b^k}{\Delta_h \theta^{k+1}}
		+
		r^{k+1} \inpro{b^k-b^{k+1}}{\Delta_h \theta^{k+1}}
    \nonumber\\
    &=P_1+P_2+\ldots +P_6.
\end{align}
The terms $P_1$, $P_2$, and $P_3$ can be estimated by Young's inequality. We have for any $\alpha>0$,
\begin{align*}
    \abs{P_1}+\abs{P_2}+\abs{P_3}
    &\leq
    C\norm{\xi^{k+1}}{L^2}^2
    +
    C\norm{\zeta^{k+1}}{L^2}^2
    +
    Ch^{2(q+1)}
    +
    \alpha \norm{\Delta_h \theta^{k+1}}{L^2}^2
    \\
    &\leq
    C(h^{2(q+1)}+ \tau^2)+\alpha \norm{\Delta_h \theta^{k+1}}{L^2}^2,
\end{align*}
where in the last step we used \eqref{equ:error xi k L2} and \eqref{equ:error theta n eul}. Similarly, for the next three terms, we use \eqref{equ:frac rk f min rk1 f}, \eqref{equ:frac L2 uhk uk}, and Young's inequality to obtain
\begin{align*}
    \abs{P_4}+\abs{P_5}+\abs{P_6}
    &\leq
    C(h^{2(q+1)}+ \tau^2)+\alpha \norm{\Delta_h \theta^{k+1}}{L^2}^2,
\end{align*}
Continuing from \eqref{equ:Delta th L2}, by choosing $\alpha>0$ sufficiently small to absorb the term $\norm{\Delta_h \theta^{k+1}}{L^2}^2$, we show \eqref{equ:error Delta theta L2} for the first two terms on the left-hand side.

To show the inequality for the remaining term, we take $\phi=\tau \Delta_h \xi^{k+1}$ in \eqref{equ:fem euler sub u}, sum over $j\in \{0,1,\ldots,k-1\}$, and apply Young's inequality as before. Since the argument is straightforward by now, further details are omitted. Finally, inequality \eqref{equ:error theta Linfty} follows by \eqref{equ:error Delta theta L2} and \eqref{equ:disc lapl L infty}.
\end{proof}

Next, we show an auxiliary error estimate for $\norm{\nabla \xi^k}{L^2}$.

\begin{proposition}
For sufficiently small $h,\tau>0$, we have
\begin{align}\label{equ:error xi H1}
    \norm{\nabla \xi^k}{L^2}^2
    +
    \tau \sum_{j=1}^k \norm{\nabla \Delta_h \xi^j}{L^2}^2 
		\leq C(h^{2q}+ \tau^2),
\end{align}
where $C$ depends on $T$, but is independent of $k$, $h$, and $\tau$.
\end{proposition}

\begin{proof}
Taking $\phi=\epsilon^2 \Delta_h^2 \xi^{k+1}$ in \eqref{equ:fem euler sub u} and noting~\eqref{equ:disc lap sq}, we have
\begin{align}\label{equ:ep nab Delta xi}
    &\epsilon^2 \norm{\nabla\Delta_h \xi^{k+1}}{L^2}^2
    -
    \epsilon^2 \inpro{\dtt \nabla \theta^{k+1}}{\nabla\Delta_h \xi^{k+1}}
    \nonumber\\
    &=
    -\epsilon^2 \inpro{\dtt \rho^{k+1}+ \dtt u^{k+1}-\partial_t u^{k+1}}{\Delta_h^2 \xi^{k+1}}
    +
	\inpro{\big(p_h^k-P(u^{k+1})\big)\cdot (\sigma_h^{k+1}-\mu_h^{k+1})}{\Delta_h^2 \xi^{k+1}}
		\nonumber\\
		&\quad
		+
		\inpro{P(u^{k+1})\cdot \big(\delta^{-1}(\zeta^{k+1}+\omega^{k+1}) - \chi_0(\theta^k+\rho^k+u^k-u^{k+1})\big)}{\Delta_h^2 \xi^{k+1}}
		\nonumber\\
		&\quad
		- \inpro{P(u^{k+1})\cdot (\xi^{k+1}+\eta^{k+1})}{\Delta_h^2 \xi^{k+1}}
        \nonumber\\
        &=:Q_1+Q_2+Q_3+Q_4.
\end{align}
Next, setting $\psi=-\Delta_h \xi^{k+1}$ in \eqref{equ:dtt xi kk} and similarly noting \eqref{equ:disc lap sq}, we obtain
\begin{align}\label{equ:nab xi L2}
    &\frac12 \left(\norm{\nabla \xi^{k+1}}{L^2}^2 - \norm{\nabla \xi^k}{L^2}^2 \right)
    +
    \frac12 \norm{\nabla \xi^{k+1}-\nabla \xi^k}{L^2}^2
    +
    \epsilon^2 \inpro{\nabla \dtt \theta^{k+1}}{\nabla \Delta_h \xi^{k+1}}
    \nonumber\\
    &=
    \inpro{\dtt \eta^{k+1}+ \dtt \mu^{k+1}- \partial_t \mu^{k+1}}{\Delta_h \xi^{k+1}}
    +
    \lambda \inpro{\dtt \theta^{k+1} + \dtt \rho^{k+1} + \dtt u^{k+1}- \partial u^{k+1}}{\Delta_h \xi^{k+1}} 
    \nonumber\\
    &\quad
    -
    \chi_0 \inpro{\dtt \zeta^{k+1}+ \dtt \omega^{k+1} + \dtt n^{k+1}- \partial_t n^{k+1}}{\Delta_h \xi^{k+1}}
    \nonumber\\
    &\quad
    +
    \inpro{\frac{1}{\tau} \left(\frac{r_h^{k+1}}{\sqrt{\beta_h^k}} - \frac{r_h^k}{\sqrt{\beta_h^{k-1}}} \right) f'(u_h^k)}{\Delta_h \xi^{k+1}}
    \nonumber\\
    &\quad
    +
    \inpro{\left(\frac{r_h^k}{\sqrt{\beta_h^{k-1}}}\right) \left(\frac{f'(u_h^k)-f'(u_h^{k-1})}{\tau}\right) - \left(\frac{r^{k+1}}{\sqrt{\beta^{k+1}}} \right) f''(u^{k+1}) \cdot \partial_t u^{k+1}}{\Delta_h \xi^{k+1}}
    \nonumber\\
    &=:R_1+R_2+\ldots+R_5.
\end{align}
We now add \eqref{equ:ep nab Delta xi} and \eqref{equ:nab xi L2}, then estimate each term on the right-hand side of each equation. 
\\[2ex]
\underline{Estimates for $Q_1$ to $Q_4$}: Firstly, by H\"older's and Young's inequalities, employing \eqref{equ:Ritz ineq} and \eqref{equ:disc lap H-1},  we have for any $\alpha>0$,
\begin{align*}
    \abs{Q_1}
    &\leq
    C \norm{\dtt \rho^{k+1}+\dtt u^{k+1}-\partial_t u^{k+1}}{H^1} \norm{\Delta_h^2 \xi^{k+1}}{\widetilde{H}^{-1}}
    \\
    &\leq
    Ch^{2q}+C\tau^2 + \alpha \norm{\nabla\Delta_h \xi^{k+1}}{L^2}^2.
\end{align*}
Next, by H\"older's and Young's inequalities, noting \eqref{equ:Ritz ineq} and \eqref{equ:disc lap H-1}, as well as error estimates \eqref{equ:error theta n eul} and \eqref{equ:error xi k L2},
\begin{align*}
    \abs{Q_2}
    &\leq
    C\norm{\big(p_h^k-p^{k+1}\big) (\sigma_h^{k+1}-\mu_h^{k+1})}{H^1}  \norm{\Delta_h^2 \xi^{k+1}}{\widetilde{H}^{-1}}
    \\
    &\leq
    C\norm{\theta^k+\rho^k+u^k-u^{k+1}}{H^1} \norm{\sigma_h^{k+1}-\mu_h^{k+1}}{L^\infty} \norm{\Delta_h^2 \xi^{k+1}}{\widetilde{H}^{-1}}
    \\
    &\quad
    +
    C\norm{\theta^k+\rho^k+u^k-u^{k+1}}{L^6} \norm{\nabla\sigma_h^{k+1}- \nabla\mu_h^{k+1}}{L^3} \norm{\Delta_h^2 \xi^{k+1}}{\widetilde{H}^{-1}}
    \\
    &\leq
    C(h^{2q}+\tau^2) \left(1+\norm{\Delta_h \sigma_h^{k+1}}{L^2}^2 + \norm{\Delta_h \mu_h^{k+1}}{L^2}^2 \right) 
    +
    \alpha \norm{\nabla\Delta_h \xi^{k+1}}{L^2}^2,
\end{align*}
where in the last step we also used \eqref{equ:disc lapl L infty}, \eqref{equ:nab vh L3}, \eqref{equ:stab u H1 n L2}, and \eqref{equ:stab n H1}. The terms $Q_3$ and $Q_4$ can be estimated in a standard way, giving
\begin{align*}
    \abs{Q_3}
    &\leq
    \norm{p^{k+1} \big(\delta^{-1}(\zeta^{k+1}+\omega^{k+1}) - \chi_0(\theta^k+\rho^k+u^k-u^{k+1})\big)}{H^1} 
    \norm{\Delta_h^2 \xi^{k+1}}{\widetilde{H}^{-1}}
    \\
    &\leq
    C\left(1+\norm{u^{k+1}}{W^{1,\infty}}\right) \norm{\zeta^{k+1}+\omega^{k+1}+\theta^k+\rho^k+ u^k-u^{k+1}}{H^1} \norm{\Delta_h^2 \xi^{k+1}}{\widetilde{H}^{-1}}
    \\
    &\leq
    C(h^{2q}+\tau^2) + \alpha \norm{\nabla\Delta_h \xi^{k+1}}{L^2}^2,
\end{align*}
where in the last step we also used the regularity of $u$ in \eqref{equ:reg u euler}, as well as error estimates \eqref{equ:error xi k L2} and \eqref{equ:error Delta theta L2}.
Similarly, for the term $Q_4$ we apply Young's inequality and \eqref{equ:error xi k L2} to obtain
\begin{align*}
    \abs{Q_4}
    &\leq
    C\norm{p^{k+1}}{W^{1,\infty}} \norm{\xi^{k+1}+\eta^{k+1}}{H^1} \norm{\Delta_h^2 \xi^{k+1}}{\widetilde{H}^{-1}}
    \\
    &\leq
    C(h^{2q}+\tau^2) + \alpha \norm{\nabla \xi^{k+1}}{L^2}^2 + \alpha \norm{\nabla\Delta_h \xi^{k+1}}{L^2}^2.
\end{align*}
\\
\underline{Estimates for $R_1$ to $R_5$}: For the first three terms, we apply Young's inequality, \eqref{equ:Ritz ineq}, and \eqref{equ:dt vn min par v} to obtain
\begin{align*}
    \abs{R_1}
    &\leq
    Ch^{2(q+1)}+ C\tau^2 + \alpha \norm{\Delta_h \xi^{k+1}}{L^2}^2,
    \\
    \abs{R_2}
    &\leq
    Ch^{2(q+1)}+ C\tau^2 + C\norm{\dtt \theta^{k+1}}{L^2}^2 + \alpha \norm{\Delta_h \xi^{k+1}}{L^2}^2,
    \\
    \abs{R_3}
    &\leq
    Ch^{2(q+1)}+ C\tau^2 + C\norm{\dtt \zeta^{k+1}}{L^2}^2 + \alpha \norm{\Delta_h \xi^{k+1}}{L^2}^2.
\end{align*}
For the terms $R_4$ and $R_5$, by \eqref{equ:pre est N5} and \eqref{equ:pre est N6} we have
\begin{align*}
    \abs{R_4}
    &\leq
    Ch^{2(q+1)} + C\tau^2 + C \norm{\Delta_h \xi^{k+1}}{L^2}^2 + \alpha \norm{\dtt \theta^k}{L^2}^2 + \alpha \norm{\dtt \theta^{k+1}}{L^2}^2,
    \\
    \abs{R_5}
    &\leq
    Ch^{2(q+1)}+ C\tau^2 + C\norm{\Delta_h \xi^{k+1}}{L^2}^2 + \alpha \norm{\dtt \theta^k}{L^2}^2.
\end{align*}
\\
\underline{Conclusion of the proof}: Altogether, choosing $\alpha>0$ sufficiently small, rearranging the terms, and summing over $j\in \{0,1,\ldots,k-1\}$, we infer that
\begin{align*}
    &\norm{\nabla \xi^k}{L^2}^2 
    +
    \tau \sum_{j=1}^k \norm{\nabla\Delta_h \xi^j}{L^2}^2
    \\
    &\leq
    C(h^{2q}+\tau^2) \tau\sum_{j=1}^k \left(1+\norm{\Delta_h \sigma_h^j}{L^2}^2 + \norm{\Delta_h \mu_h^j}{L^2}^2 \right) 
    +
    C\tau \sum_{j=1}^k \left(\norm{\dtt \theta^j}{L^2}^2 + \norm{\dtt \zeta^j}{L^2}^2 + \norm{\Delta_h \xi^j}{L^2}^2 \right)
    \\
    &\leq
    C(h^{2q}+\tau^2),
\end{align*}
where in the last step we used \eqref{equ:stab n H1}, \eqref{equ:stab Delta u L2}, \eqref{equ:error xi k L2}, and \eqref{equ:error Delta theta L2}.
\end{proof}

We are now finally in a position to state our main theorem, which shows optimal order convergence in various norms.

\begin{theorem}\label{the:main}
Let $\{\mathcal{T}_h\}_{h>0}$ be a family of quasi-uniform triangulations of a convex polytopal domain $\mathscr{D}$.
Let $(u_h^k,\mu_h^k,n_h^k,\sigma_h^k,r_h^k)$ be the solution of the fully discrete scheme \eqref{equ:fem euler}, and let $(u,\mu,n,\sigma,r)$ be a solution of \eqref{equ:tum sav} {satisfying the regularity assumption \eqref{equ:reg u euler}}. Suppose that Assumptions~\ref{ass:assum P}--\ref{ass:assum f} and the coercivity condition~\eqref{equ:ass lambda begin} hold. For sufficiently small $h,\tau >0$ and $k\in \{1,2,\ldots,\lfloor T/\tau \rfloor\}$, we have
\begin{align}
\label{equ:error euler Hs}
    \norm{u_h^k-u(t_k)}{H^s}
    +
    \norm{\mu_h^k-\mu(t_k)}{H^s}
    +
    \norm{n_h^k-n(t_k)}{H^s}
    +
    \abs{r_h^k-r(t_k)}
    &\leq
    C(h^{q+1-s}+\tau),
\end{align}
where $s\in \{0,1\}$.
Furthermore,
\begin{equation}\label{equ:error euler Linfty}
\begin{aligned}
    \norm{u_h^k-u(t_k)}{L^\infty}
    &\leq
    C(h^{q+1}+\tau) \abs{\ln h}^{\frac12}, \quad \text{for $d\leq 3$},
    \\
    \norm{n_h^k-n(t_k)}{L^\infty}
    &\leq
    C(h^{q+1}+\tau) \abs{\ln h}^{\frac12}, \quad \text{for $d\leq 2$}.
\end{aligned}
\end{equation}
The energy functional $\widetilde{\mathcal{E}}[u_h^k,n_h^k,r_h^k]$ is a good approximation of $\mathcal{E}[u(t_k),n(t_k)]$ in the sense that
\begin{align}\label{equ:good app ener}
    \abs{\widetilde{\mathcal{E}}[u_h^k,n_h^k,r_h^k] - \mathcal{E}[u(t_k),n(t_k)]} \leq C(h^q+\tau).
\end{align}
The constant $C$ depends on $T$, but is independent of $k$, $h$, and $\tau$.
\end{theorem}

\begin{proof}
The error estimate \eqref{equ:error euler Hs} follows at once by noting the splitting \eqref{equ:split uk}, \eqref{equ:split mu}, \eqref{equ:split n}, \eqref{equ:split rk}, together with the auxiliary error estimates \eqref{equ:error theta n eul}, \eqref{equ:error xi k L2}, \eqref{equ:error Delta theta L2}, \eqref{equ:error xi H1}, and \eqref{equ:Ritz ineq}. Similarly, inequality \eqref{equ:error euler Linfty} follows by the auxiliary error estimates \eqref{equ:error theta Linfty}, \eqref{equ:error zeta Linfty 2d}, and \eqref{equ:Ritz infty}.

Finally, we have
\begin{align*}
    &\abs{\widetilde{\mathcal{E}}[u_h^k,n_h^k,r_h^k] - \mathcal{E}[u(t_k),n(t_k)]}
    \\
    &\leq
    \frac{\epsilon^2}{2} \int_\mathscr{D} \abs{ \abs{\nabla u_h^k}^2 - \abs{\nabla u(t_k)}^2 } \dx 
    +
    \frac{\lambda}{2} \int_\mathscr{D} \abs{ \abs{u_h^k}^2 - \abs{u(t_k)}^2 } \dx 
    +
    \frac{1}{2\delta} \int_\mathscr{D} \abs{ \abs{n_h^k}^2 - \abs{n(t_k)}^2 } \dx 
    \\
    &\quad
    +
    \chi_0 \int_\mathscr{D} \abs{ u_h^k n_h^k - u(t_k) n(t_k) } \dx 
    +
    \abs{ \abs{r_h^k}^2 - \abs{r(t_k)}^2 }
    \\
    &\leq
    C \norm{u_h^k-u(t_k)}{H^1} + \norm{n_h^k-n(t_k)}{L^2} + \abs{r_h^k- r(t_k)},
\end{align*}
where in the last step we used the elementary inequality $|a^2-b^2|\leq \left(|a|+|b|\right) |a-b|$ together with H\"older's inequality, noting the uniform bounds on $\norm{u_h^k}{H^1}$, $\norm{n_h^k}{L^2}$, and $|r_h^k|$. This yields \eqref{equ:good app ener}, in view of \eqref{equ:error euler Hs}, as required.
\end{proof}

\section{Numerical simulations}\label{sec:exp}

We present numerical simulations of the scheme \eqref{equ:fem euler} using the open-source package~\textsc{FEniCS}~\cite{AlnaesEtal15}. Since no exact solution of the underlying equation is available, we verify the convergence order experimentally by means of extrapolation. Let $\left(u_h^k, \mu_h^k, n_h^k \right)$ denote the finite element solution at time-step $k$, with spatial mesh size $h$ and time-step size $\tau$. The extrapolated spatial (resp. temporal) convergence rate is defined as
\begin{equation*}
	h\text{-rate}(\cdot) :=  \log_2 \left[\frac{\max_k \norm{e_{2h}(\cdot)}{X}}{\max_k \norm{e_{h}(\cdot)}{X}}\right], \; \text{ and } \; 
    \tau\text{-rate}(\cdot) :=  \log_2 \left[\frac{ \norm{e_{2\tau}(\cdot)}{X}}{\norm{e_{\tau}(\cdot)}{X}}\right],
\end{equation*}
for $X \in \{L^2, L^\infty, H^1\}$. Here, $e_h(v) := v_{h}^k - v_{h/2}^k$ and $e_\tau(v)$ is the difference between the numerical solutions at time $T$ computed with time-step size $\tau$ and $\tau/2$. By choosing the time-step size $\tau$ sufficiently small, we ensure that the spatial discretisation error dominates, thereby allowing us to assess the spatial convergence rate. Similarly, fixing $h$ sufficiently small while refining $\tau$ allows us to verify the temporal convergence rate.

\subsection{Simulation 1 (growth of a single tumour)}

Let the domain $\mathscr{D}=[-1,1]^2$. We fix $B=4$, and take the parameters to be $\kappa=0.25$, $\lambda=0.001$, $\delta=0.4$, $p_0=50.0$, $\chi_0=0.02$, and $\epsilon=0.02$. The initial data is specified to be
\[
    u_0(x,y)=\frac12 + \frac12 \tanh\left(\frac{0.2-\sqrt{x^2+y^2}}{\epsilon\sqrt{2}} + 1\right),
    \;\text{ and }\; 
    n_0(x,y)=1-u_0(x,y).
\]
Snapshots of the tumour volume fraction $u$ and the nutrient volume fraction $n$ with mesh-size $h=1/64$ at selected times are shown in Figure~\ref{fig:snapshots u 2d 1} and Figure~\ref{fig:snapshots n 2d 1}, respectively. The colour indicates the relative value of the quantity. Plots of $e_h(u)$, $e_h(\mu)$, and $e_h(n)$ against $1/h$, with time-step size fixed at $\tau=5\times 10^{-4}$, are shown in Figures~\ref{fig:order u 1}, \ref{fig:order mu 1}, and \ref{fig:order n 1}, respectively.

The growth of a single circular tumour with a diffuse interface is simulated in this experiment under a large proliferation growth parameter and small chemotaxis parameter. The tumour is observed to grow steadily, maintaining its shape. Total combined mass of the tumour and the nutrient is conserved throughout the simulation. Both the free energy and the modified SAV energy of the system are observed to decay in accordance with the theory, as seen in Figure~\ref{fig:mass energy exp1}.

\begin{figure}[!htb]
	\centering
	\begin{subfigure}[b]{0.26\textwidth}
		\centering
		\includegraphics[width=\textwidth]{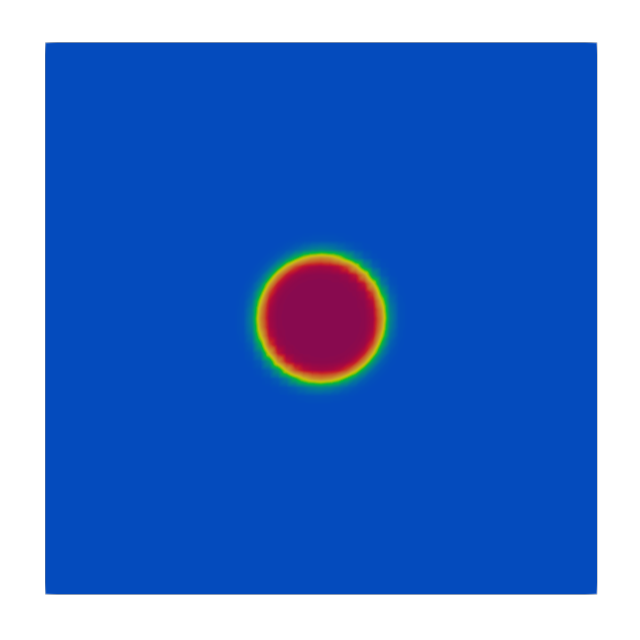}
		\caption{$t=0$}
	\end{subfigure}
	\begin{subfigure}[b]{0.26\textwidth}
		\centering
		\includegraphics[width=\textwidth]{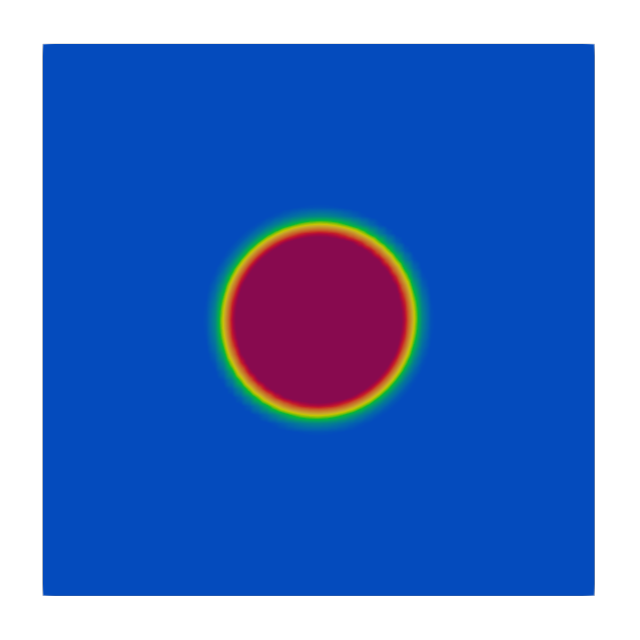}
		\caption{$t=0.04$}
	\end{subfigure}
	\begin{subfigure}[b]{0.26\textwidth}
		\centering
		\includegraphics[width=\textwidth]{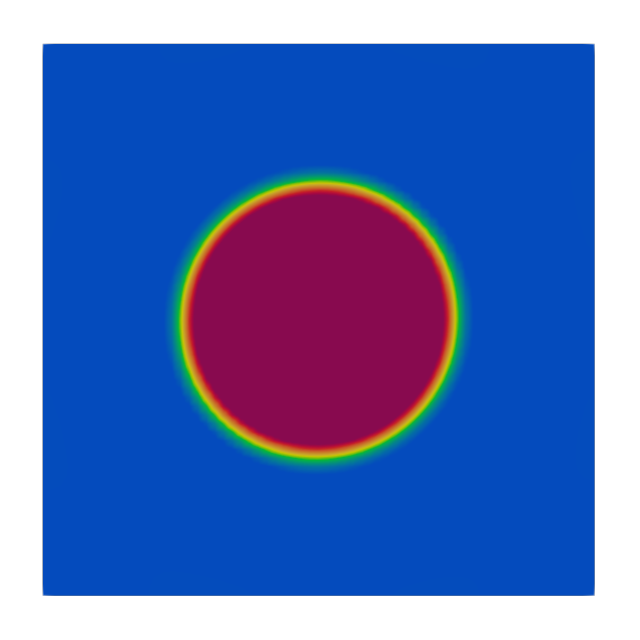}
		\caption{$t=0.08$}
	\end{subfigure}
	\begin{subfigure}[b]{0.1\textwidth}
		\centering
		\includegraphics[width=\textwidth]{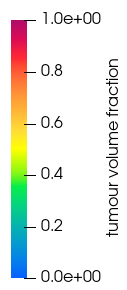}
	\end{subfigure}
	\begin{subfigure}[b]{0.26\textwidth}
		\centering
		\includegraphics[width=\textwidth]{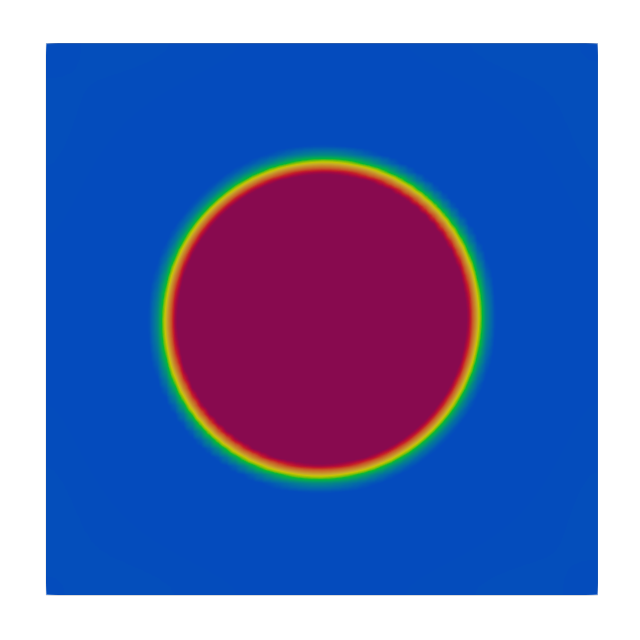}
		\caption{$t=0.1$}
	\end{subfigure}
	\begin{subfigure}[b]{0.26\textwidth}
		\centering
		\includegraphics[width=\textwidth]{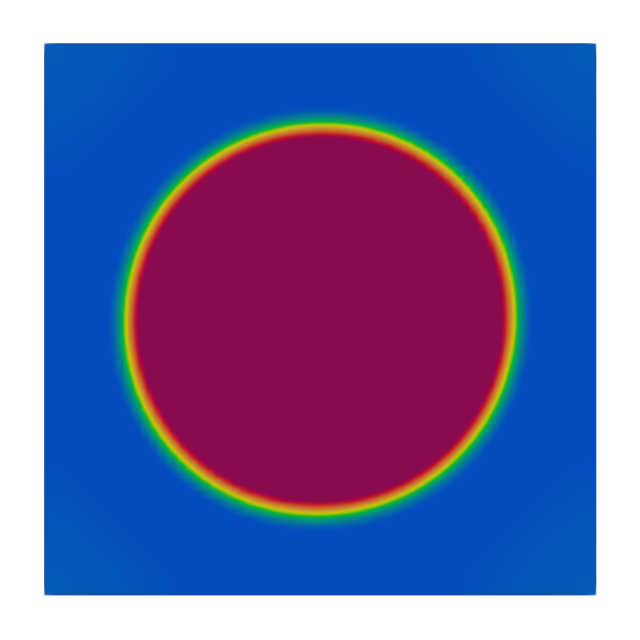}
		\caption{$t=0.14$}
	\end{subfigure}
	\begin{subfigure}[b]{0.26\textwidth}
		\centering
		\includegraphics[width=\textwidth]{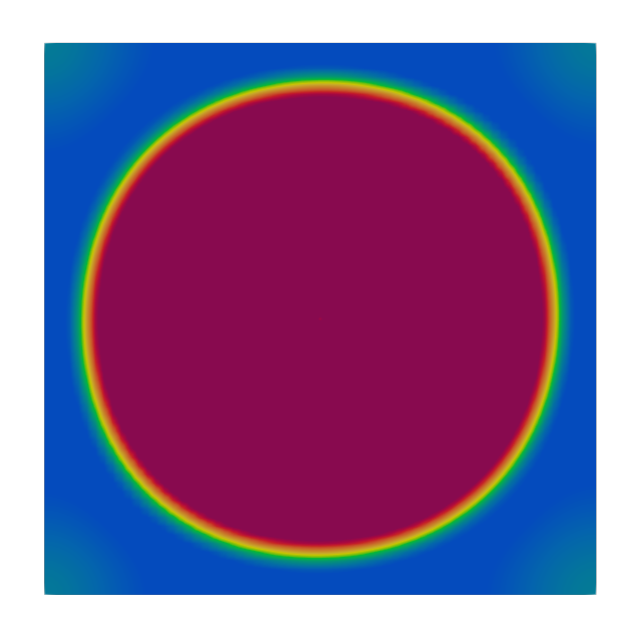}
		\caption{$t=0.2$}
	\end{subfigure}
	\begin{subfigure}[b]{0.1\textwidth}
		\centering
		\includegraphics[width=\textwidth]{u_exp1_legend.png}
	\end{subfigure}
	\caption{Snapshots of the tumour volume fraction $u$ in simulation 1.}
	\label{fig:snapshots u 2d 1}
\end{figure}

\begin{figure}[!htb]
	\centering
	\begin{subfigure}[b]{0.26\textwidth}
		\centering
		\includegraphics[width=\textwidth]{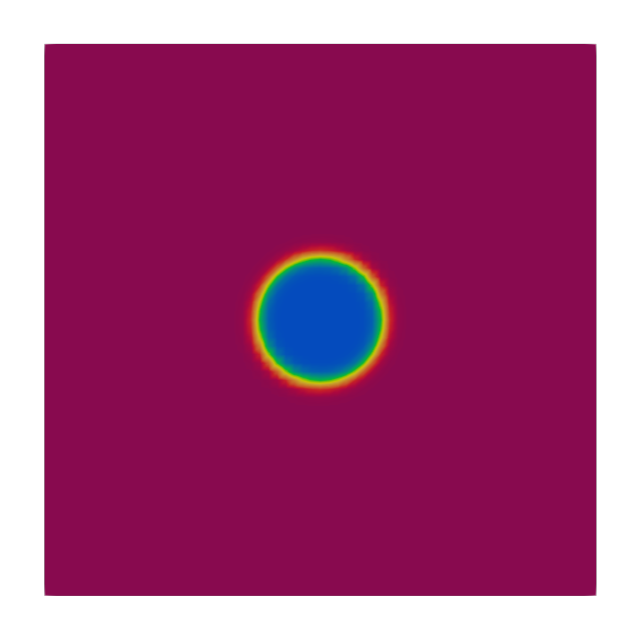}
		\caption{$t=0$}
	\end{subfigure}
	\begin{subfigure}[b]{0.26\textwidth}
		\centering
		\includegraphics[width=\textwidth]{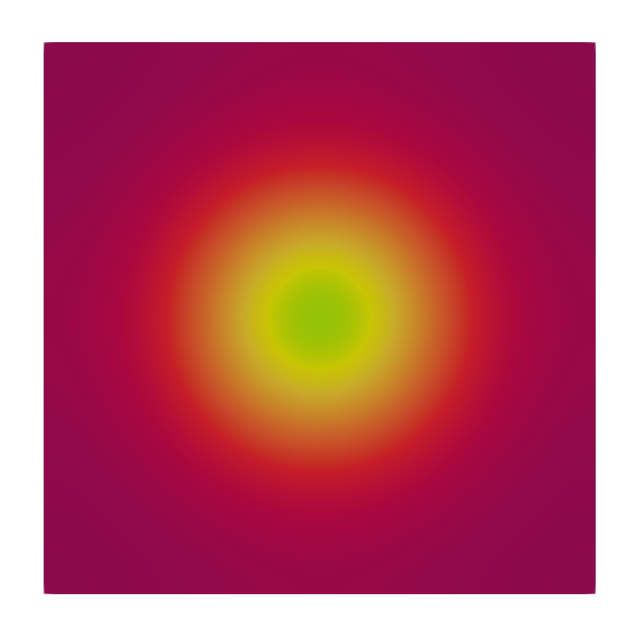}
		\caption{$t=0.04$}
	\end{subfigure}
	\begin{subfigure}[b]{0.26\textwidth}
		\centering
		\includegraphics[width=\textwidth]{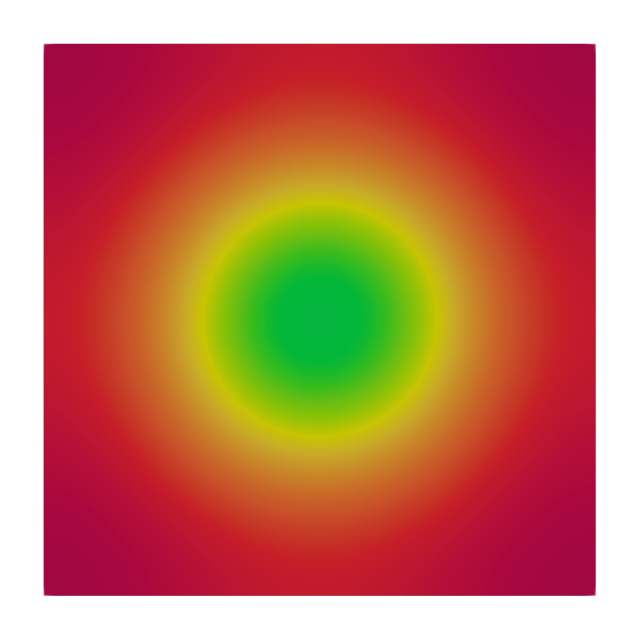}
		\caption{$t=0.08$}
	\end{subfigure}
	\begin{subfigure}[b]{0.1\textwidth}
		\centering
		\includegraphics[width=\textwidth]{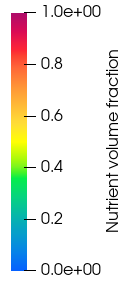}
	\end{subfigure}
	\begin{subfigure}[b]{0.26\textwidth}
		\centering
		\includegraphics[width=\textwidth]{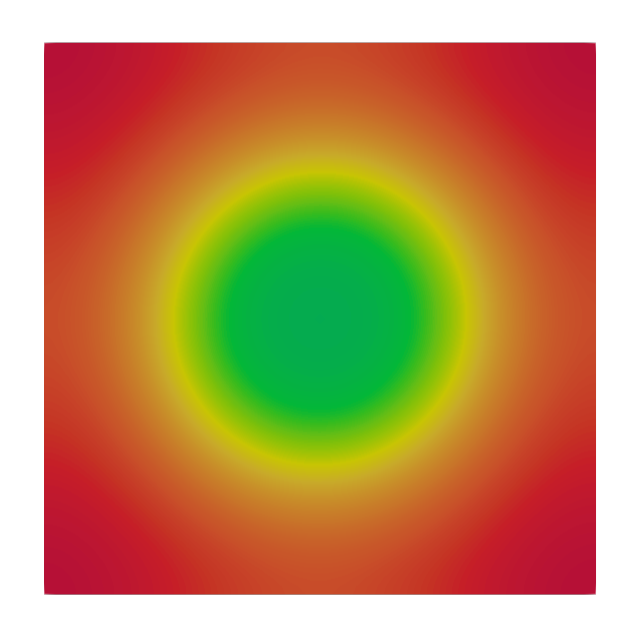}
		\caption{$t=0.1$}
	\end{subfigure}
	\begin{subfigure}[b]{0.26\textwidth}
		\centering
		\includegraphics[width=\textwidth]{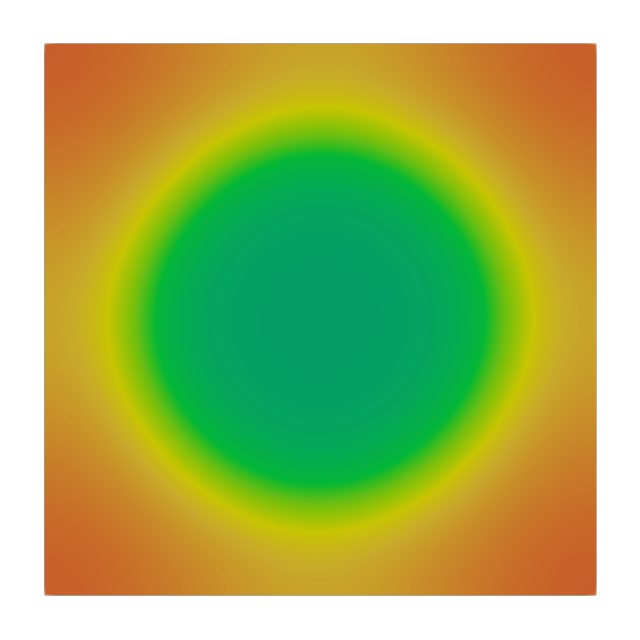}
		\caption{$t=0.14$}
	\end{subfigure}
	\begin{subfigure}[b]{0.26\textwidth}
		\centering
		\includegraphics[width=\textwidth]{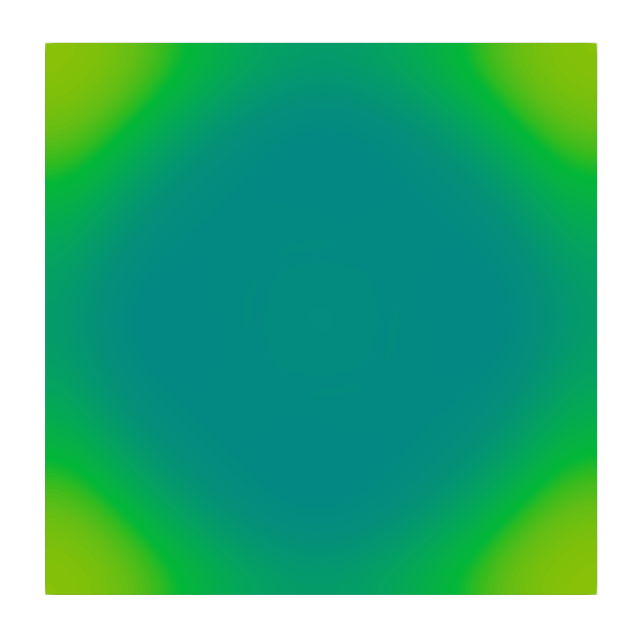}
		\caption{$t=0.2$}
	\end{subfigure}
	\begin{subfigure}[b]{0.1\textwidth}
		\centering
		\includegraphics[width=\textwidth]{n_exp1_legend.png}
	\end{subfigure}
	\caption{Snapshots of the nutrient volume fraction $n$ in simulation 1.}
	\label{fig:snapshots n 2d 1}
\end{figure}

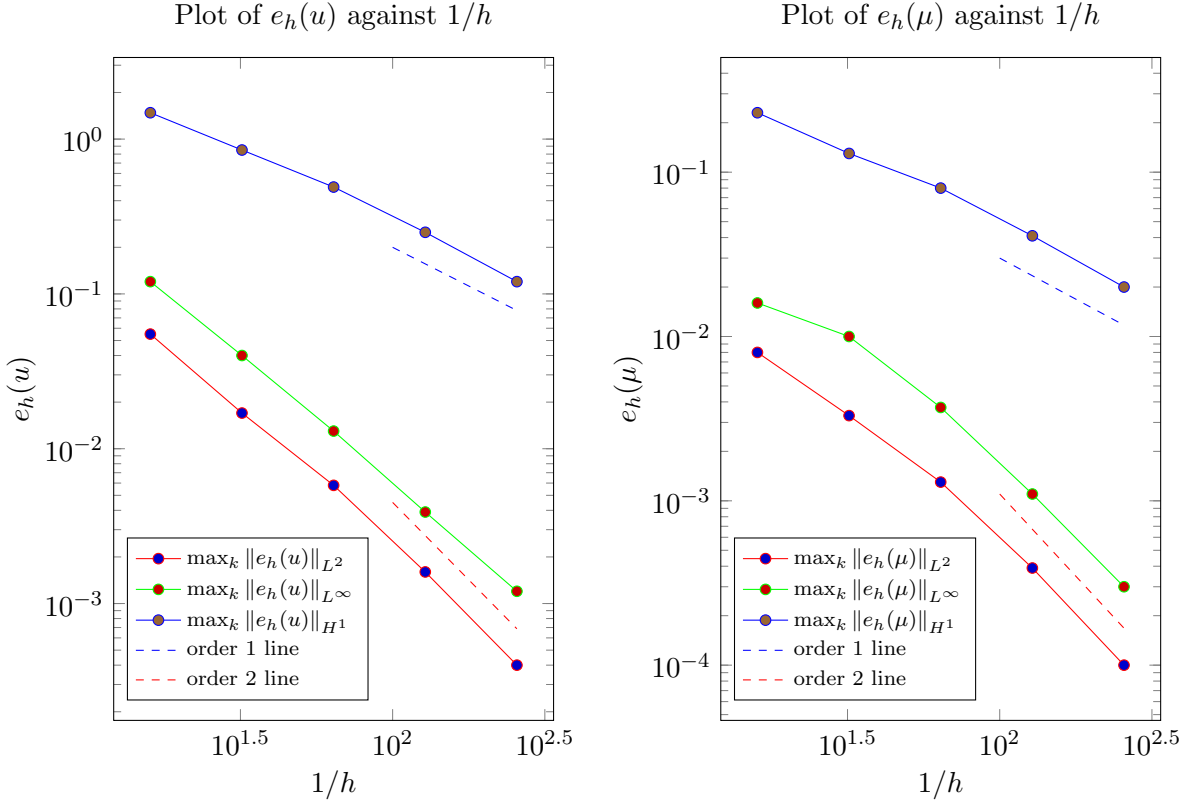
\begin{figure}[!htb]
	\begin{subfigure}[b]{0.45\textwidth}
		\centering
		\begin{tikzpicture}
			\begin{axis}[
				title=Plot of $e_h(u)$ against $1/h$,
				height=1.4\textwidth,
				width=1\textwidth,
				xlabel= $1/h$,
				ylabel= $e_h(u)$,
				xmode=log,
				ymode=log,
				legend pos=south west,
				legend cell align=left,
				]
				\addplot+[mark=*,red] coordinates {(16,0.055)(32,0.017)(64,0.0058)(128,0.0016)(256,0.0004)};
				\addplot+[mark=*,green] coordinates {(16,0.12)(32,0.04)(64,0.013)(128,0.0039)(256,0.0012)};
                \addplot+[mark=*,blue] coordinates {(16,1.48)(32,0.85)(64,0.49)(128,0.25)(256,0.12)};
				\addplot+[dashed,no marks,blue,domain=100:256]{20/x};
				\addplot+[dashed,no marks,red,domain=100:256]{45/x^2};
				\legend{\scriptsize{$\max_k \norm{e_h(u)}{L^2}$}, \scriptsize{$\max_k \norm{e_h(u)}{L^\infty}$}, \scriptsize{$\max_k \norm{e_h(u)}{H^1}$}, \scriptsize{order 1 line}, \scriptsize{order 2 line}}
			\end{axis}
		\end{tikzpicture}
		\caption{Spatial convergence order of $u$.}
		\label{fig:order u 1}
	\end{subfigure}
    \hspace{1em}
	\begin{subfigure}[b]{0.45\textwidth}
		\centering
		\begin{tikzpicture}
			\begin{axis}[
				title=Plot of $e_h(\mu)$ against $1/h$,
				height=1.4\textwidth,
				width=1\textwidth,
				xlabel= $1/h$,
				ylabel= $e_h(\mu)$,
				xmode=log,
				ymode=log,
				legend pos=south west,
				legend cell align=left,
				]
				\addplot+[mark=*,red] coordinates {(16,0.008)(32,0.0033)(64,0.0013)(128,0.00039)(256,0.0001)};
				\addplot+[mark=*,green] coordinates {(16,0.016)(32,0.01)(64,0.0037)(128,0.0011)(256,0.0003)};
                \addplot+[mark=*,blue] coordinates {(16,0.23)(32,0.13)(64,0.08)(128,0.041)(256,0.02)};
				\addplot+[dashed,no marks,blue,domain=100:256]{3/x};
				\addplot+[dashed,no marks,red,domain=100:256]{11/x^2};
				\legend{\scriptsize{$\max_k \norm{e_h(\mu)}{L^2}$}, \scriptsize{$\max_k \norm{e_h(\mu)}{L^\infty}$}, \scriptsize{$\max_k \norm{e_h(\mu)}{H^1}$}, \scriptsize{order 1 line}, \scriptsize{order 2 line}}
			\end{axis}
		\end{tikzpicture}
		\caption{Spatial convergence order of $\mu$.}
		\label{fig:order mu 1}
	\end{subfigure}
    \caption{Spatial convergence order of tumour volume fraction and chemical potential in simulation 1.}
\end{figure}

\begin{figure}[!htb]
		\begin{tikzpicture}
			\begin{axis}[
				title=Plot of $e_h(n)$ against $1/h$,
				height=0.63\textwidth,
				width=0.48\textwidth,
				xlabel= $1/h$,
				ylabel= $e_h(n)$,
				xmode=log,
				ymode=log,
				legend pos=south west,
				legend cell align=left,
				]
				\addplot+[mark=*,red] coordinates {(16,0.031)(32,0.0081)(64,0.0024)(128,0.00064)(256,0.00016)};
				\addplot+[mark=*,green] coordinates {(16,0.05)(32,0.015)(64,0.0048)(128,0.0014)(256,0.0004)};
                \addplot+[mark=*,blue] coordinates {(16,0.55)(32,0.3)(64,0.16)(128,0.078)(256,0.04)};
				\addplot+[dashed,no marks,blue,domain=100:256]{6/x};
				\addplot+[dashed,no marks,red,domain=100:256]{17/x^2};
				\legend{\scriptsize{$\max_k \norm{e_h(n)}{L^2}$}, \scriptsize{$\max_k \norm{e_h(n)}{L^\infty}$}, \scriptsize{$\max_k \norm{e_h(n)}{H^1}$}, \scriptsize{order 1 line}, \scriptsize{order 2 line}}
			\end{axis}
		\end{tikzpicture}
		\caption{Spatial convergence order of nutrient volume fraction $n$ in simulation 1.}
		\label{fig:order n 1}
\end{figure}
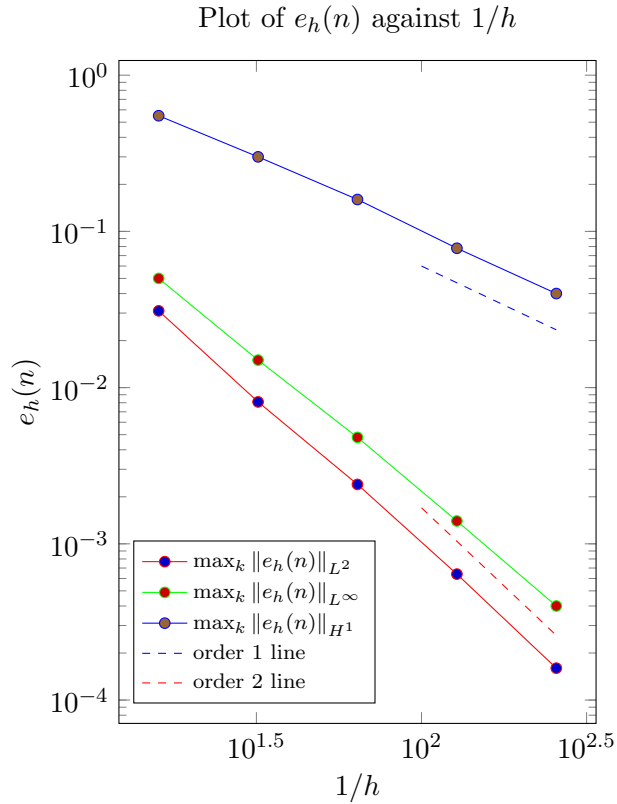

\begin{figure}[!htb]
	\centering
	\begin{subfigure}[b]{0.47\textwidth}
		\centering
		\includegraphics[width=\textwidth]{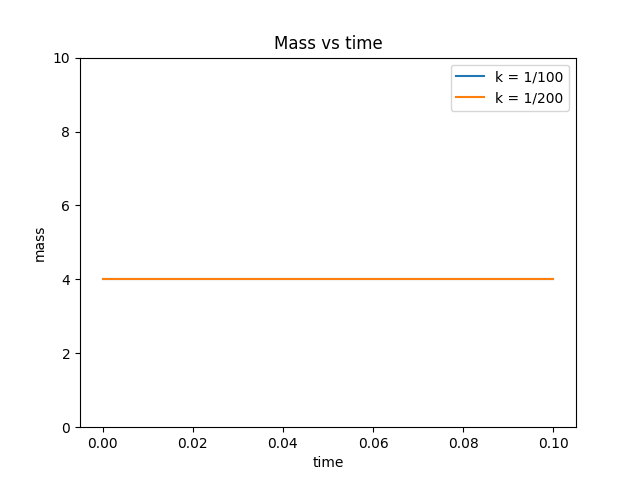}
		\caption{Graph of combined tumour and nutrient mass vs time with $\tau=1/100$ and $\tau=1/200$.}
	\end{subfigure}
    \hspace{1em}
	\begin{subfigure}[b]{0.47\textwidth}
		\centering
		\includegraphics[width=\textwidth]{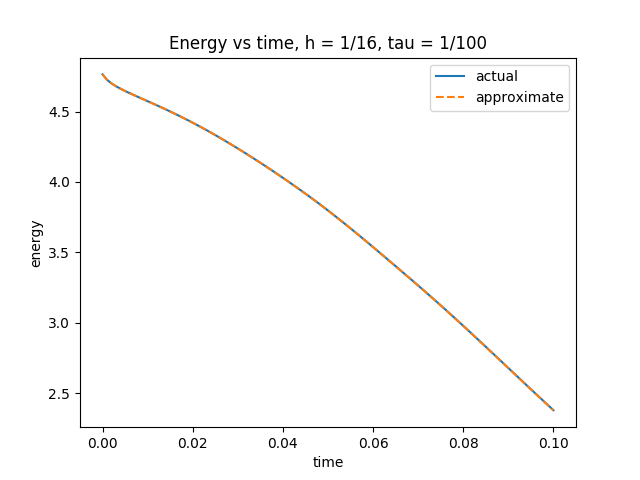}
		\caption{Graph of energy and approximate energy vs time with $h=1/16$ and $\tau=1/100$.}
	\end{subfigure}
        \caption{Conservation of mass and energy dissipation in simulation 1.}
	\label{fig:mass energy exp1}
\end{figure}

\subsection{Simulation 2 (aggregation of multiple tumours)}

We take the same parameters as those in Simulation 1. The initial data is specified to be
\[
    u_0(x,y)= 1+ \frac12 \tanh\left(\frac{0.2-\sqrt{(x+0.3)^2+y^2}}{\epsilon\sqrt{2}} +1\right) + \frac12 \tanh\left(\frac{0.2-\sqrt{(x-0.3)^2+y^2}}{\epsilon\sqrt{2}} +1\right),
\]
and $n_0(x,y)=1-u_0(x,y)$,
corresponding to two neighbouring tumours of equal sizes.
Snapshots of the tumour volume fraction $u$ and the nutrient volume fraction $n$ at selected times are shown in Figure~\ref{fig:snapshots u 2d 2} and Figure~\ref{fig:snapshots n 2d 2}, respectively. The colour indicates the relative value of the quantity. Plots of $e_h(u)$, $e_h(\mu)$, and $e_h(n)$ against $1/h$, with time-step size fixed at $\tau=5\times 10^{-4}$, are shown in Figures~\ref{fig:order u 2}, \ref{fig:order mu 2}, and \ref{fig:order n 2}, respectively.

The growth of two neighbouring tumours is simulated in this experiment under a large proliferation growth parameter and a small chemotaxis parameter. The tumours are observed to grow steadily and eventually merge to form a larger mass. Total combined mass of the tumour and the nutrient is conserved throughout the simulation. Both the free energy and the modified SAV energy of the system are observed to decay in accordance with the theory, as seen in Figure~\ref{fig:mass energy exp2}.

\begin{figure}[!htb]
	\centering
	\begin{subfigure}[b]{0.26\textwidth}
		\centering
		\includegraphics[width=\textwidth]{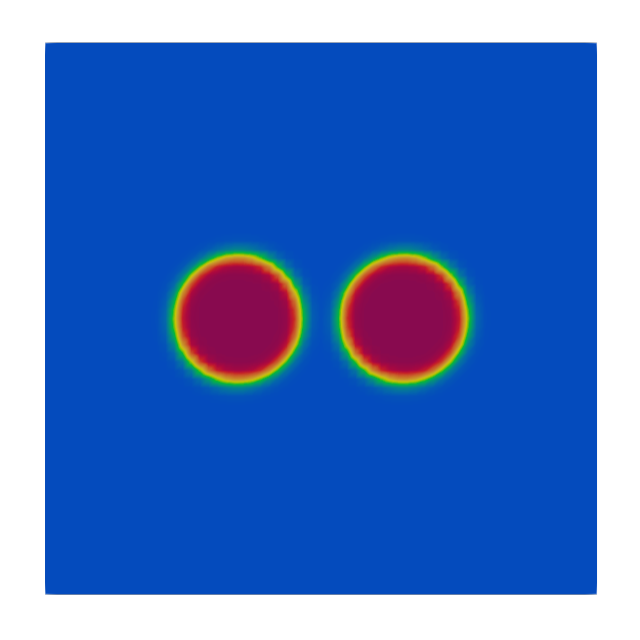}
		\caption{$t=0$}
	\end{subfigure}
	\begin{subfigure}[b]{0.26\textwidth}
		\centering
		\includegraphics[width=\textwidth]{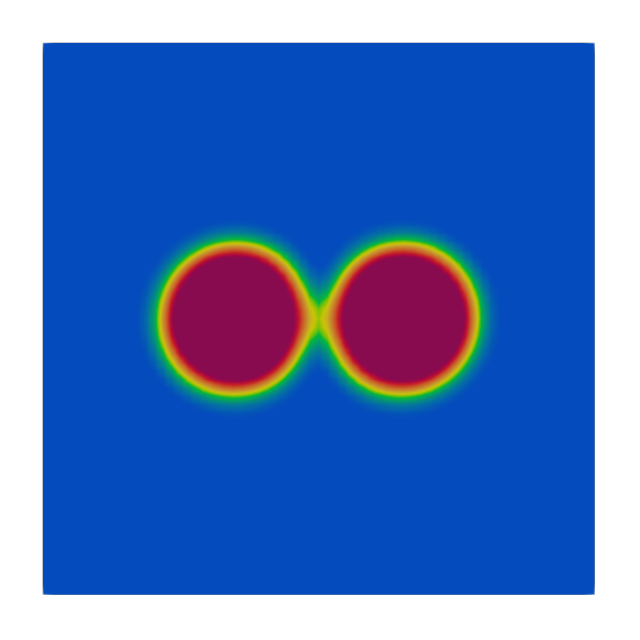}
		\caption{$t=0.02$}
	\end{subfigure}
	\begin{subfigure}[b]{0.26\textwidth}
		\centering
		\includegraphics[width=\textwidth]{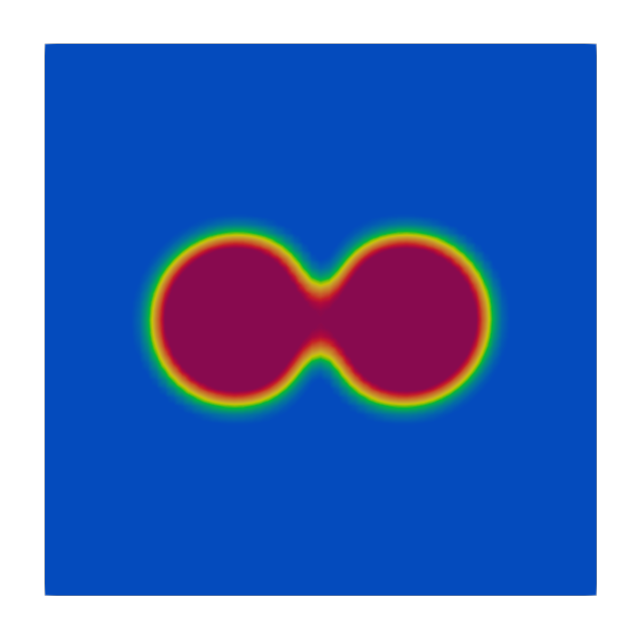}
		\caption{$t=0.03$}
	\end{subfigure}
	\begin{subfigure}[b]{0.1\textwidth}
		\centering
		\includegraphics[width=\textwidth]{u_exp1_legend.png}
	\end{subfigure}
	\begin{subfigure}[b]{0.26\textwidth}
		\centering
		\includegraphics[width=\textwidth]{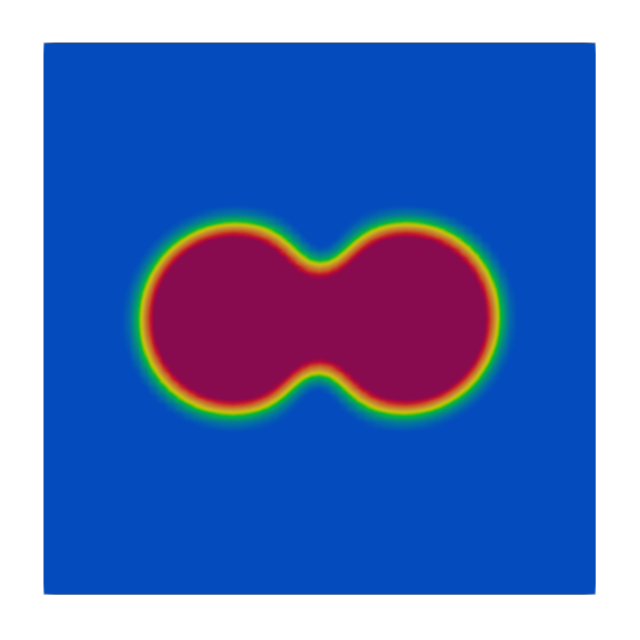}
		\caption{$t=0.04$}
	\end{subfigure}
	\begin{subfigure}[b]{0.26\textwidth}
		\centering
		\includegraphics[width=\textwidth]{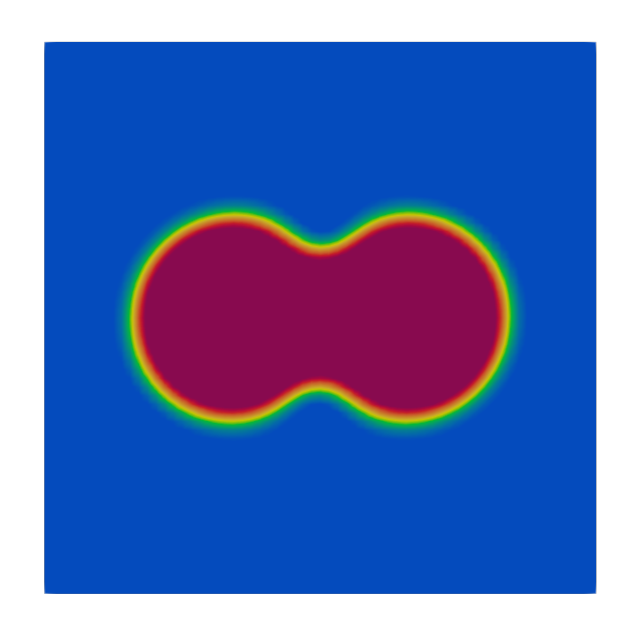}
		\caption{$t=0.05$}
	\end{subfigure}
	\begin{subfigure}[b]{0.26\textwidth}
		\centering
		\includegraphics[width=\textwidth]{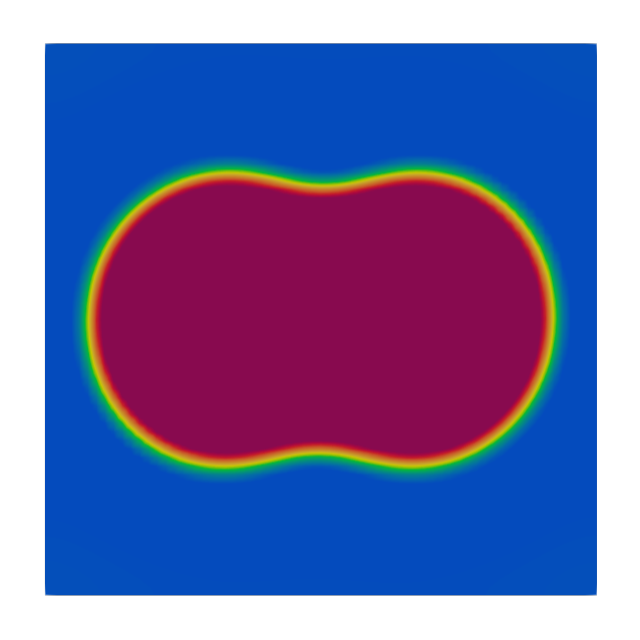}
		\caption{$t=0.1$}
	\end{subfigure}
	\begin{subfigure}[b]{0.1\textwidth}
		\centering
		\includegraphics[width=\textwidth]{u_exp1_legend.png}
	\end{subfigure}
	\caption{Snapshots of the tumour volume fraction $u$ in simulation 2.}
	\label{fig:snapshots u 2d 2}
\end{figure}

\begin{figure}[!htb]
	\centering
	\begin{subfigure}[b]{0.26\textwidth}
		\centering
		\includegraphics[width=\textwidth]{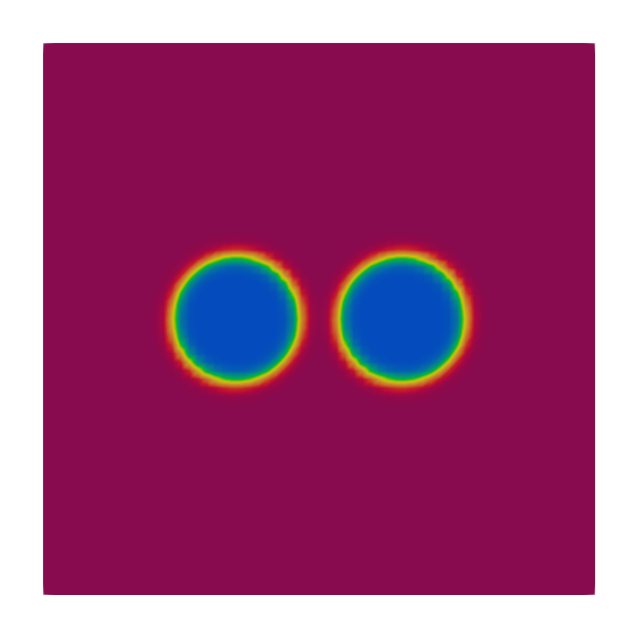}
		\caption{$t=0$}
	\end{subfigure}
	\begin{subfigure}[b]{0.26\textwidth}
		\centering
		\includegraphics[width=\textwidth]{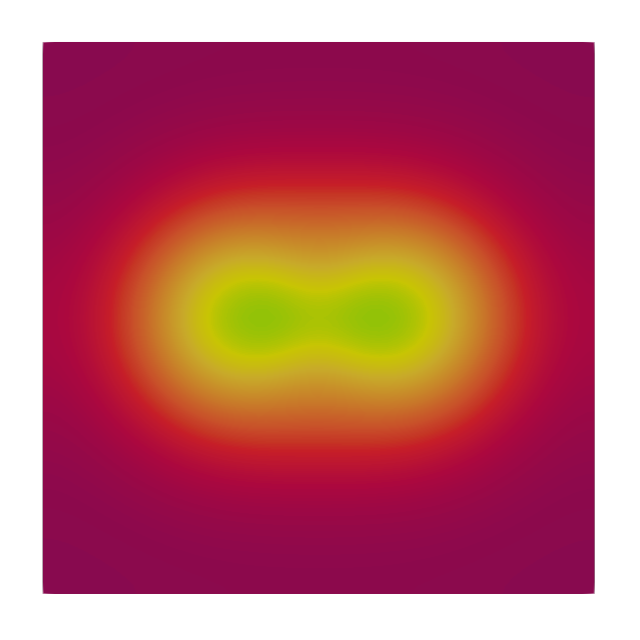}
		\caption{$t=0.02$}
	\end{subfigure}
	\begin{subfigure}[b]{0.26\textwidth}
		\centering
		\includegraphics[width=\textwidth]{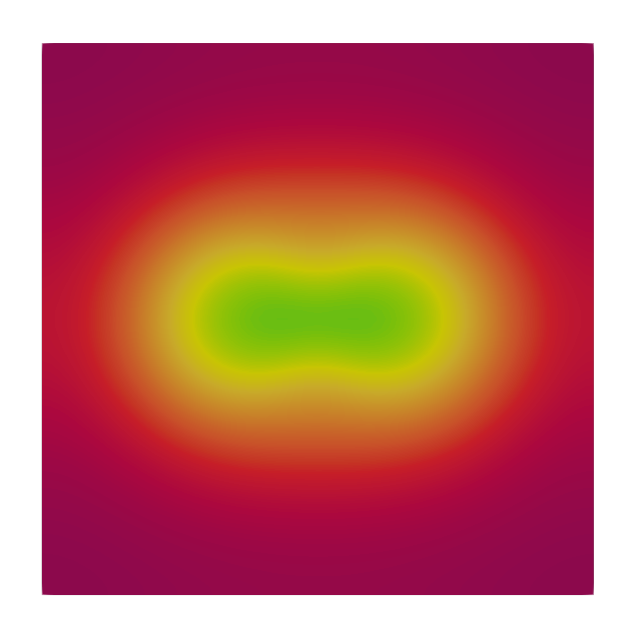}
		\caption{$t=0.03$}
	\end{subfigure}
	\begin{subfigure}[b]{0.1\textwidth}
		\centering
		\includegraphics[width=\textwidth]{n_exp1_legend.png}
	\end{subfigure}
	\begin{subfigure}[b]{0.26\textwidth}
		\centering
		\includegraphics[width=\textwidth]{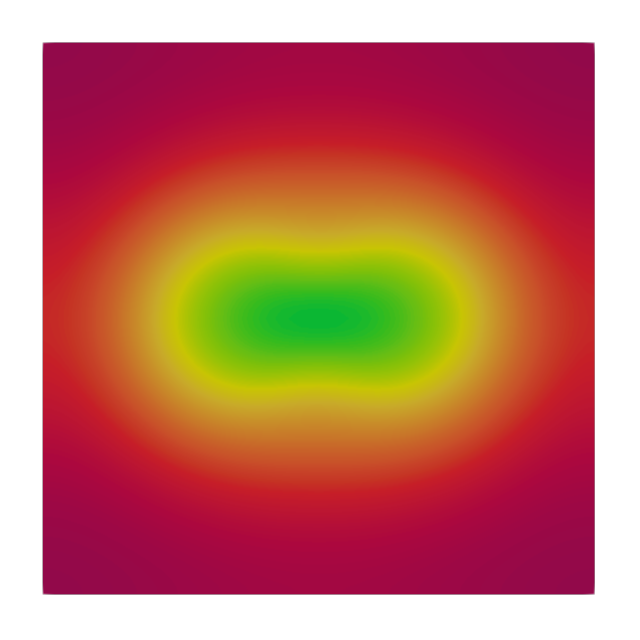}
		\caption{$t=0.04$}
	\end{subfigure}
	\begin{subfigure}[b]{0.26\textwidth}
		\centering
		\includegraphics[width=\textwidth]{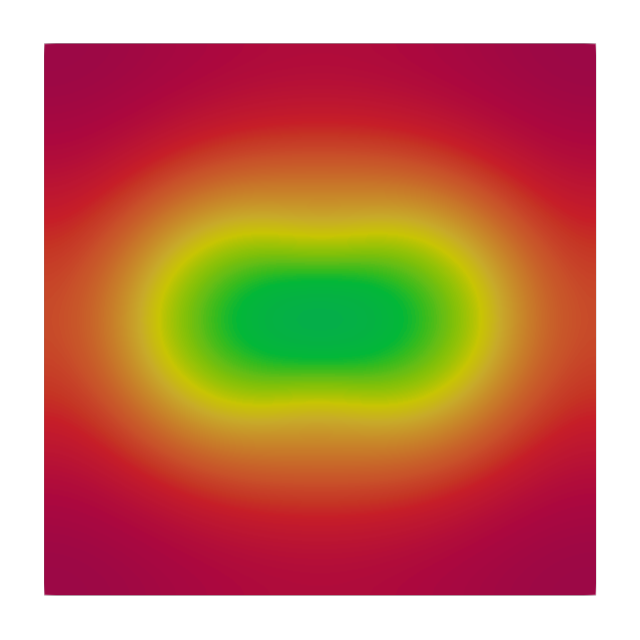}
		\caption{$t=0.05$}
	\end{subfigure}
	\begin{subfigure}[b]{0.26\textwidth}
		\centering
		\includegraphics[width=\textwidth]{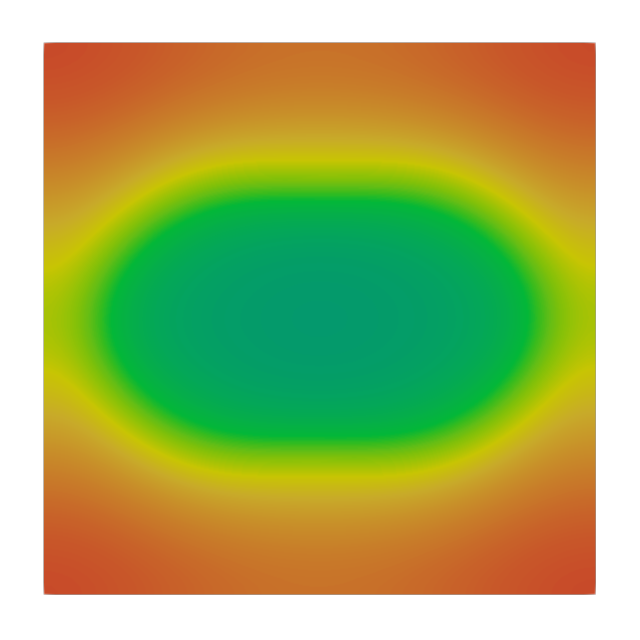}
		\caption{$t=0.1$}
	\end{subfigure}
	\begin{subfigure}[b]{0.1\textwidth}
		\centering
		\includegraphics[width=\textwidth]{n_exp1_legend.png}
	\end{subfigure}
	\caption{Snapshots of the nutrient volume fraction $n$ in simulation 2.}
	\label{fig:snapshots n 2d 2}
\end{figure}

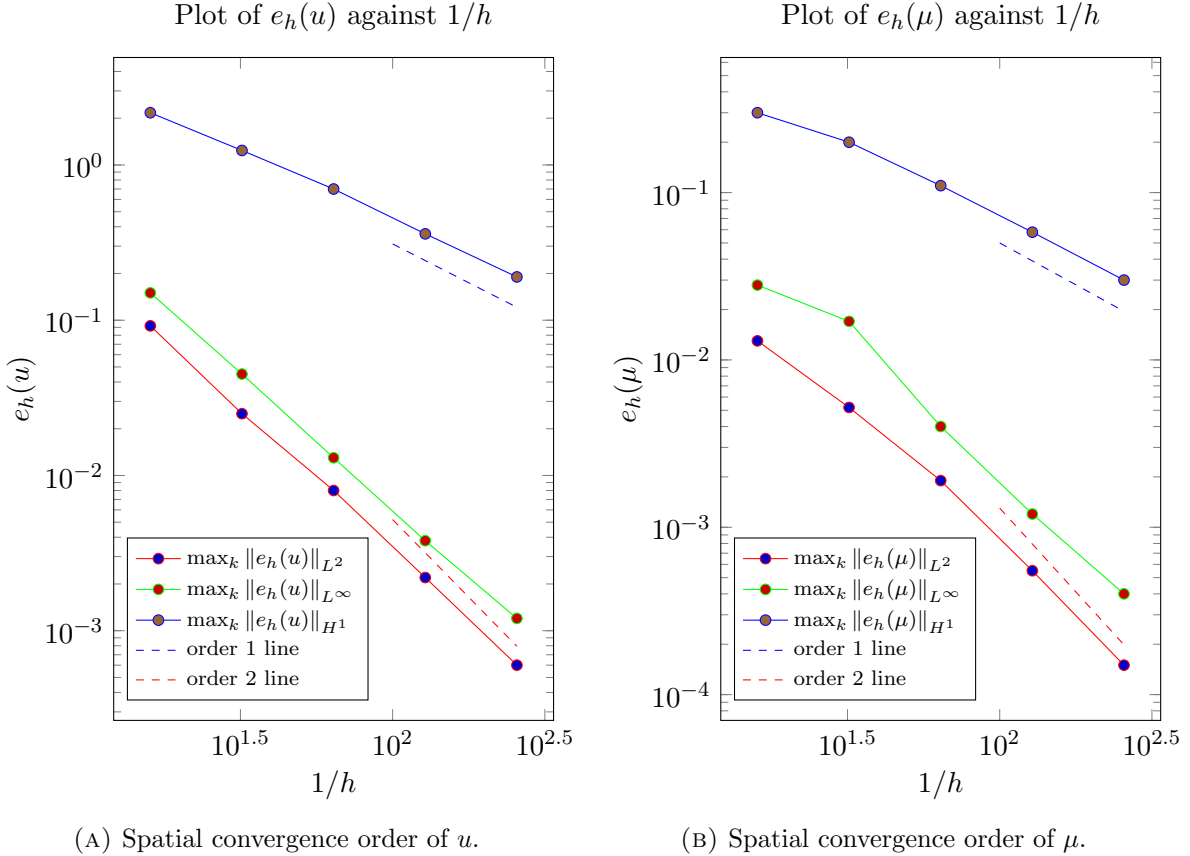
\begin{figure}[!htb]
	\begin{subfigure}[b]{0.45\textwidth}
		\centering
		\begin{tikzpicture}
			\begin{axis}[
				title=Plot of $e_h(u)$ against $1/h$,
				height=1.4\textwidth,
				width=1\textwidth,
				xlabel= $1/h$,
				ylabel= $e_h(u)$,
				xmode=log,
				ymode=log,
				legend pos=south west,
				legend cell align=left,
				]
				\addplot+[mark=*,red] coordinates {(16,0.092)(32,0.025)(64,0.008)(128,0.0022)(256,0.0006)};
				\addplot+[mark=*,green] coordinates {(16,0.15)(32,0.045)(64,0.013)(128,0.0038)(256,0.0012)};
                \addplot+[mark=*,blue] coordinates {(16,2.17)(32,1.24)(64,0.7)(128,0.36)(256,0.19)};
				\addplot+[dashed,no marks,blue,domain=100:256]{31/x};
				\addplot+[dashed,no marks,red,domain=100:256]{52/x^2};
				\legend{\scriptsize{$\max_k \norm{e_h(u)}{L^2}$}, \scriptsize{$\max_k \norm{e_h(u)}{L^\infty}$}, \scriptsize{$\max_k \norm{e_h(u)}{H^1}$}, \scriptsize{order 1 line}, \scriptsize{order 2 line}}
			\end{axis}
		\end{tikzpicture}
		\caption{Spatial convergence order of $u$.}
		\label{fig:order u 2}
	\end{subfigure}
    \hspace{1em}
	\begin{subfigure}[b]{0.45\textwidth}
		\centering
		\begin{tikzpicture}
			\begin{axis}[
				title=Plot of $e_h(\mu)$ against $1/h$,
				height=1.4\textwidth,
				width=1\textwidth,
				xlabel= $1/h$,
				ylabel= $e_h(\mu)$,
				xmode=log,
				ymode=log,
				legend pos=south west,
				legend cell align=left,
				]
				\addplot+[mark=*,red] coordinates {(16,0.013)(32,0.0052)(64,0.0019)(128,0.00055)(256,0.00015)};
				\addplot+[mark=*,green] coordinates {(16,0.028)(32,0.017)(64,0.004)(128,0.0012)(256,0.0004)};
                \addplot+[mark=*,blue] coordinates {(16,0.3)(32,0.2)(64,0.11)(128,0.058)(256,0.03)};
				\addplot+[dashed,no marks,blue,domain=100:256]{5/x};
				\addplot+[dashed,no marks,red,domain=100:256]{13/x^2};
				\legend{\scriptsize{$\max_k \norm{e_h(\mu)}{L^2}$}, \scriptsize{$\max_k \norm{e_h(\mu)}{L^\infty}$}, \scriptsize{$\max_k \norm{e_h(\mu)}{H^1}$}, \scriptsize{order 1 line}, \scriptsize{order 2 line}}
			\end{axis}
		\end{tikzpicture}
		\caption{Spatial convergence order of $\mu$.}
		\label{fig:order mu 2}
	\end{subfigure}
    \caption{Spatial convergence order of tumour volume fraction and chemical potential in simulation 2.}
\end{figure}

\begin{figure}[!htb]
		\begin{tikzpicture}
			\begin{axis}[
				title=Plot of $e_h(n)$ against $1/h$,
				height=0.63\textwidth,
				width=0.48\textwidth,
				xlabel= $1/h$,
				ylabel= $e_h(n)$,
				xmode=log,
				ymode=log,
				legend pos=south west,
				legend cell align=left,
				]
				\addplot+[mark=*,red] coordinates {(16,0.049)(32,0.012)(64,0.0035)(128,0.00092)(256,0.00024)};
				\addplot+[mark=*,green] coordinates {(16,0.089)(32,0.016)(64,0.0051)(128,0.0014)(256,0.0004)};
                \addplot+[mark=*,blue] coordinates {(16,0.83)(32,0.43)(64,0.22)(128,0.11)(256,0.05)};
				\addplot+[dashed,no marks,blue,domain=100:256]{8/x};
				\addplot+[dashed,no marks,red,domain=100:256]{9/x^2};
				\legend{\scriptsize{$\max_k \norm{e_h(n)}{L^2}$}, \scriptsize{$\max_k \norm{e_h(n)}{L^\infty}$}, \scriptsize{$\max_k \norm{e_h(n)}{H^1}$}, \scriptsize{order 1 line}, \scriptsize{order 2 line}}
			\end{axis}
		\end{tikzpicture}
		\caption{Spatial convergence order of nutrient volume fraction $n$ in simulation 2.}
		\label{fig:order n 2}
\end{figure}
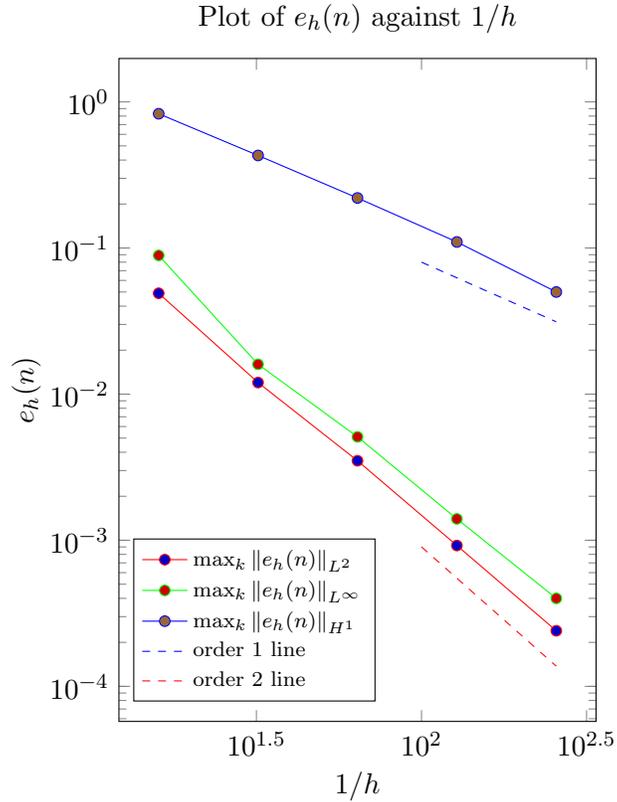

\begin{figure}[!htb]
	\centering
	\begin{subfigure}[b]{0.47\textwidth}
		\centering
		\includegraphics[width=\textwidth]{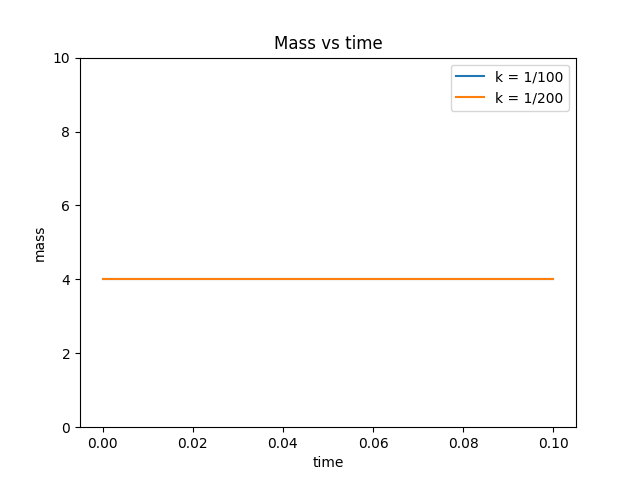}
		\caption{Graph of combined tumour and nutrient mass vs time with $\tau=1/100$ and $\tau=1/200$.}
	\end{subfigure}
    \hspace{1em}
	\begin{subfigure}[b]{0.47\textwidth}
		\centering
		\includegraphics[width=\textwidth]{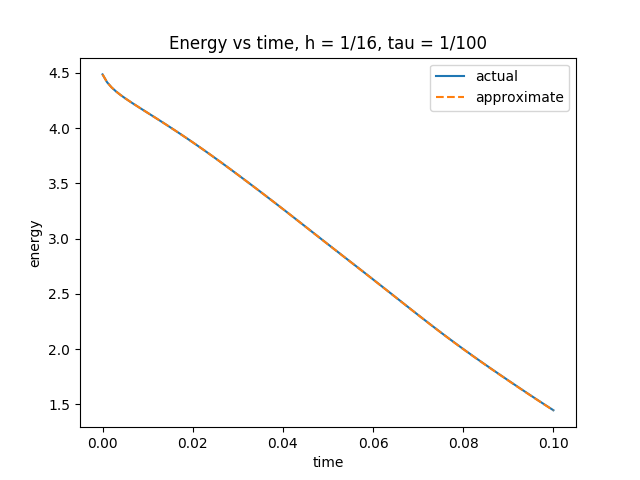}
		\caption{Graph of energy and approximate energy vs time with $h=1/16$ and $\tau=1/100$.}
	\end{subfigure}
        \caption{Conservation of mass and energy dissipation in simulation 2.}
	\label{fig:mass energy exp2}
\end{figure}

\subsection{Simulation 3 (temporal rates of convergence)}

We perform a numerical experiment on the domain $\mathscr{D}=[-0.5,0.5]^2$ to verify the temporal convergence rates of the scheme. The parameters are chosen as follows: $B=5.0$, $\kappa=0.25$, $\lambda=0.2$, $\delta=0.4$, $p_0=20.0$, $\chi_0=0.2$, and $\epsilon=0.02$. The initial conditions are given by
\[
u_0(x,y)=\big| \sin(2\pi x) \sin(2\pi y) \big|,
    \;\text{ and }\; 
    n_0(x,y)=1-u_0(x,y),
\]
which represent less regular initial data, as $u_0$ and $n_0$ only belong to $H^1$, but not $H^2$.

Plots of $e_\tau(u)$, $e_\tau(\mu)$, and $e_\tau(n)$ against $1/\tau$ are displayed in Figures~\ref{fig:order u temp}, \ref{fig:order mu temp}, and \ref{fig:order n temp}, respectively. The results indicate that the scheme achieves first-order temporal convergence in practice, even for less regular initial data beyond the assumptions of Theorem~\ref{the:main}.

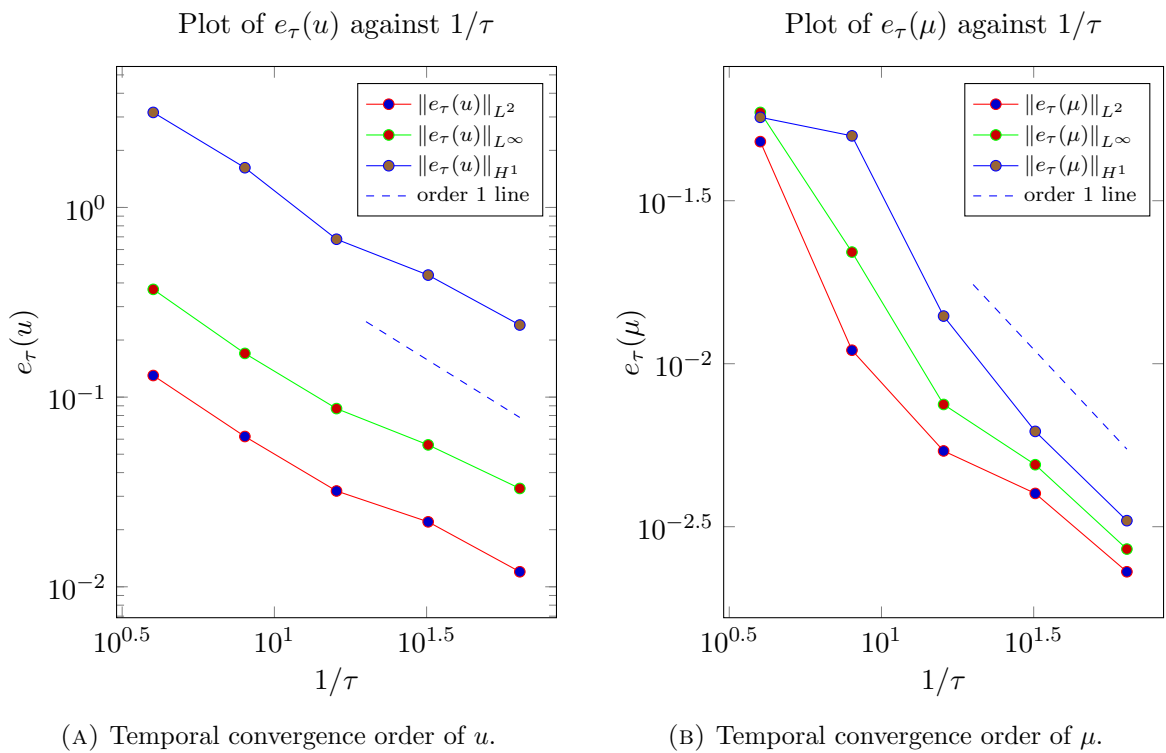
\begin{figure}[!htb]
	\begin{subfigure}[b]{0.45\textwidth}
		\centering
		\begin{tikzpicture}
			\begin{axis}[
				title=Plot of $e_\tau(u)$ against $1/\tau$,
				height=1.2\textwidth,
				width=1\textwidth,
				xlabel= $1/\tau$,
				ylabel= $e_\tau(u)$,
				xmode=log,
				ymode=log,
				legend pos=north east,
				legend cell align=left,
				]
				\addplot+[mark=*,red] coordinates {(4,0.13)(8,0.062)(16,0.032)(32,0.022)(64,0.012)};
				\addplot+[mark=*,green] coordinates {(4,0.37)(8,0.17)(16,0.087)(32,0.056)(64,0.033)};
                \addplot+[mark=*,blue] coordinates {(4,3.17)(8,1.62)(16,0.68)(32,0.44)(64,0.24)};
				\addplot+[dashed,no marks,blue,domain=20:64]{5/x};
				\legend{\scriptsize{$\norm{e_\tau(u)}{L^2}$}, \scriptsize{$\norm{e_\tau(u)}{L^\infty}$}, \scriptsize{$\norm{e_\tau(u)}{H^1}$}, \scriptsize{order 1 line}}
			\end{axis}
		\end{tikzpicture}
		\caption{Temporal convergence order of $u$.}
		\label{fig:order u temp}
	\end{subfigure}
    \hspace{1em}
	\begin{subfigure}[b]{0.45\textwidth}
		\centering
		\begin{tikzpicture}
			\begin{axis}[
				title=Plot of $e_\tau(\mu)$ against $1/\tau$,
				height=1.2\textwidth,
				width=1\textwidth,
				xlabel= $1/\tau$,
				ylabel= $e_\tau(\mu)$,
				xmode=log,
				ymode=log,
				legend pos=north east,
				legend cell align=left,
				]
				\addplot+[mark=*,red] coordinates {(4,0.048)(8,0.011)(16,0.0054)(32,0.004)(64,0.0023)};
				\addplot+[mark=*,green] coordinates {(4,0.059)(8,0.022)(16,0.0075)(32,0.0049)(64,0.0027)};
                \addplot+[mark=*,blue] coordinates {(4,0.057)(8,0.05)(16,0.014)(32,0.0062)(64,0.0033)};
				\addplot+[dashed,no marks,blue,domain=20:64]{0.35/x};
				\legend{\scriptsize{$\norm{e_\tau(\mu)}{L^2}$}, \scriptsize{$\norm{e_\tau(\mu)}{L^\infty}$}, \scriptsize{$\norm{e_\tau(\mu)}{H^1}$}, \scriptsize{order 1 line}}
			\end{axis}
		\end{tikzpicture}
		\caption{Temporal convergence order of $\mu$.}
		\label{fig:order mu temp}
	\end{subfigure}
    \caption{Temporal convergence of tumour volume fraction and chemical potential in simulation 3.}
\end{figure}

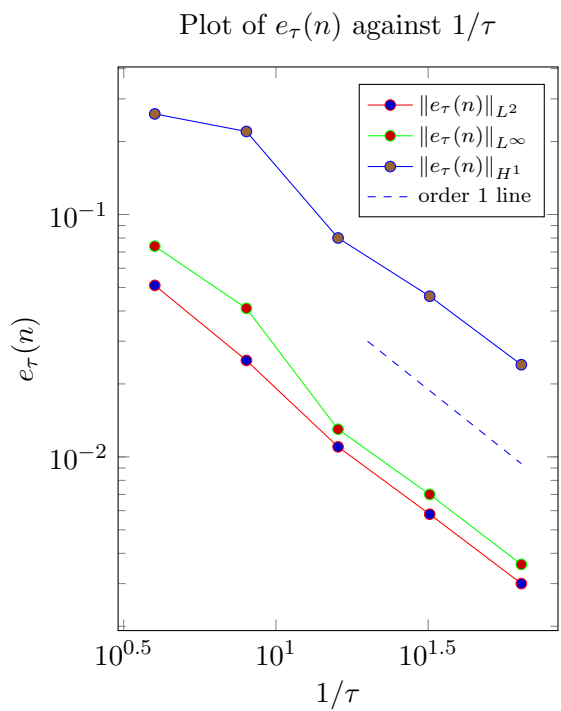
\begin{figure}[!htb]
		\begin{tikzpicture}
			\begin{axis}[
				title=Plot of $e_\tau(n)$ against $1/\tau$,
				height=0.55\textwidth,
				width=0.45\textwidth,
				xlabel= $1/\tau$,
				ylabel= $e_\tau(n)$,
				xmode=log,
				ymode=log,
				legend pos=north east,
				legend cell align=left,
				]
				\addplot+[mark=*,red] coordinates {(4,0.051)(8,0.025)(16,0.011)(32,0.0058)(64,0.003)};
				\addplot+[mark=*,green] coordinates {(4,0.074)(8,0.041)(16,0.013)(32,0.0070)(64,0.0036)};
                \addplot+[mark=*,blue] coordinates {(4,0.26)(8,0.22)(16,0.08)(32,0.046)(64,0.024)};
				\addplot+[dashed,no marks,blue,domain=20:64]{0.6/x};
				\legend{\scriptsize{$\norm{e_\tau(n)}{L^2}$}, \scriptsize{$\norm{e_\tau(n)}{L^\infty}$}, \scriptsize{$\norm{e_\tau(n)}{H^1}$}, \scriptsize{order 1 line}}
			\end{axis}
		\end{tikzpicture}
		\caption{Temporal convergence of nutrient volume fraction $n$ in simulation 3.}
		\label{fig:order n temp}
\end{figure}

\subsection{Simulation 4 (chemotactic growth of a three-dimensional tumour)}

For this simulation, we take the domain $\mathscr{D}=[-1,1]^3$. We perform a simulation to illustrate growth of a spherical tumour which exhibits chemotactic growth.  The parameters are taken as follows: $B=20.0, \kappa=0.25, \lambda=0.5, \delta=0.4, p_0=25.0, \chi_0=1.6$, and $\epsilon=0.02$. For this choice of parameters, $\lambda<\chi_0^2 \delta$, so the sufficient coercivity condition \eqref{equ:ass lambda ener} used in the analysis is not satisfied. Nevertheless, the simulation remains stable and captures the expected chemotactic tumour growth.

The initial data for the tumour volume fraction and the nutrient volume fraction are, respectively,
\begin{align*}
    u_0(x,y,z)
    &=
    \frac12 + \frac12 \tanh\left(\frac{0.2-\sqrt{x^2+y^2+z^2}}{\epsilon\sqrt{2}} + 1\right),
\end{align*}
which represents a slightly perturbed spherical tumour, and
\begin{align*}
    n_0(x,y,z)
    &=
    \frac32 + \frac12 \tanh\left(\frac{0.2-\sqrt{(x+0.3)^2+(y+0.3)^2+z^2}}{\epsilon\sqrt{2}} + 1\right)
    \\
    &\quad
    +
    \frac12 \tanh\left(\frac{0.2-\sqrt{(x-0.3)^2+(y-0.3)^2+z^2}}{\epsilon\sqrt{2}} + 1\right)
    \\
    &\quad
    +
    \frac14 \tanh\left(\frac{0.2-\sqrt{(x-0.3)^2+(y+0.3)^2+z^2}}{\epsilon\sqrt{2}} + 1\right)
    \\
    &\quad
    +
    \frac14 \tanh\left(\frac{0.2-\sqrt{(x+0.3)^2+(y-0.3)^2+z^2}}{\epsilon\sqrt{2}} + 1\right),
\end{align*}
which represents four spherical nutrient sources with varying densities. Snapshots of the tumour volume fraction $u$ and the nutrient volume fraction $n$ with mesh-size $h=1/50$ at selected times are shown in Figure~\ref{fig:snapshots u 3d 3} and Figure~\ref{fig:snapshots n 3d 3}, respectively. The colour indicates the relative value of the quantity. 

In this experiment, we simulate the growth of a nearly spherical tumour under a large chemotaxis parameter. The tumour exhibits chemotactic behaviour, with growth concentrated in regions of high nutrient density, before eventually coalescing to form a large mass~\cite{PulSimSea15}. Throughout the simulation, the total mass of both the tumour and the nutrient is conserved. Moreover, both the free energy and its numerical approximation are observed to decay consistently with the theoretical predictions, as illustrated in Figure~\ref{fig:mass energy exp3}.

\begin{figure}[!htb]
	\centering
	\begin{subfigure}[b]{0.26\textwidth}
		\centering
		\includegraphics[width=\textwidth]{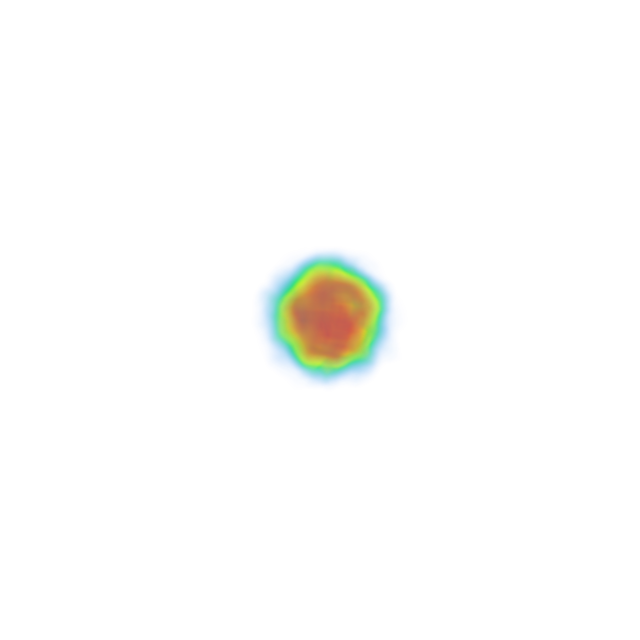}
		\caption{$t=0$}
	\end{subfigure}
	\begin{subfigure}[b]{0.26\textwidth}
		\centering
		\includegraphics[width=\textwidth]{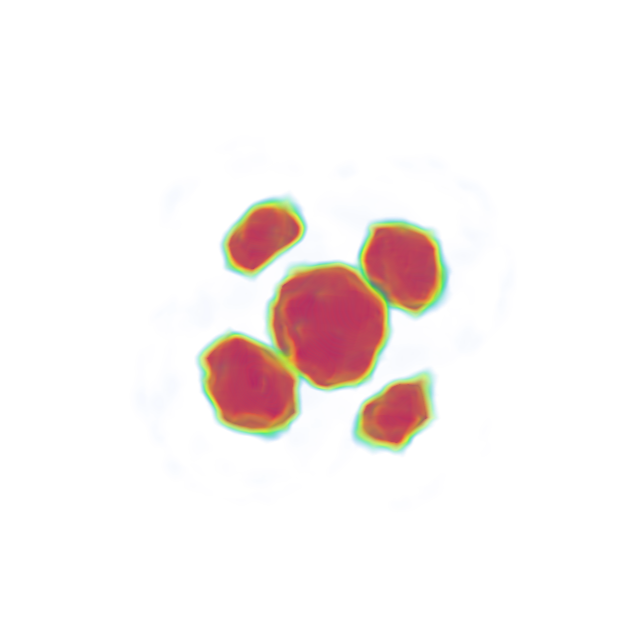}
		\caption{$t=0.05$}
	\end{subfigure}
	\begin{subfigure}[b]{0.26\textwidth}
		\centering
		\includegraphics[width=\textwidth]{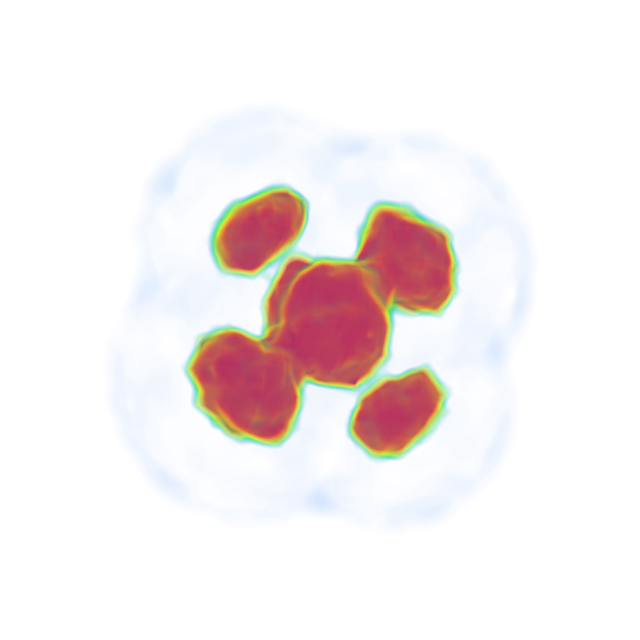}
		\caption{$t=0.08$}
	\end{subfigure}
	\begin{subfigure}[b]{0.1\textwidth}
		\centering
		\includegraphics[width=\textwidth]{u_exp1_legend.png}
	\end{subfigure}
	\begin{subfigure}[b]{0.26\textwidth}
		\centering
		\includegraphics[width=\textwidth]{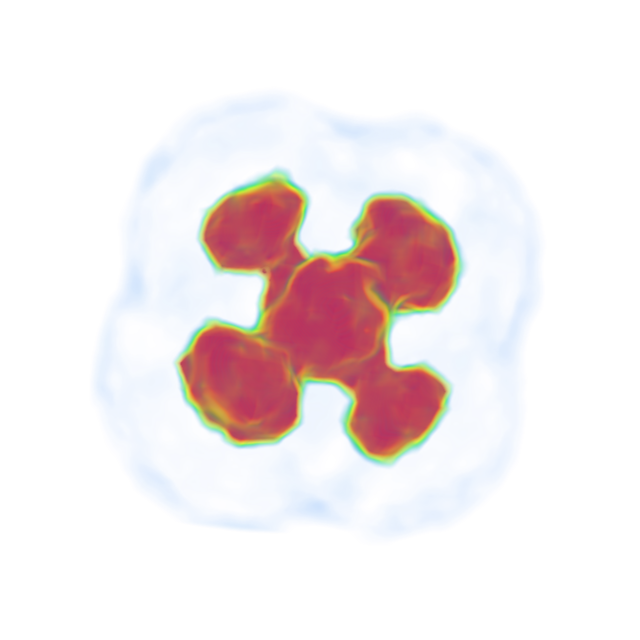}
		\caption{$t=0.1$}
	\end{subfigure}
	\begin{subfigure}[b]{0.26\textwidth}
		\centering
		\includegraphics[width=\textwidth]{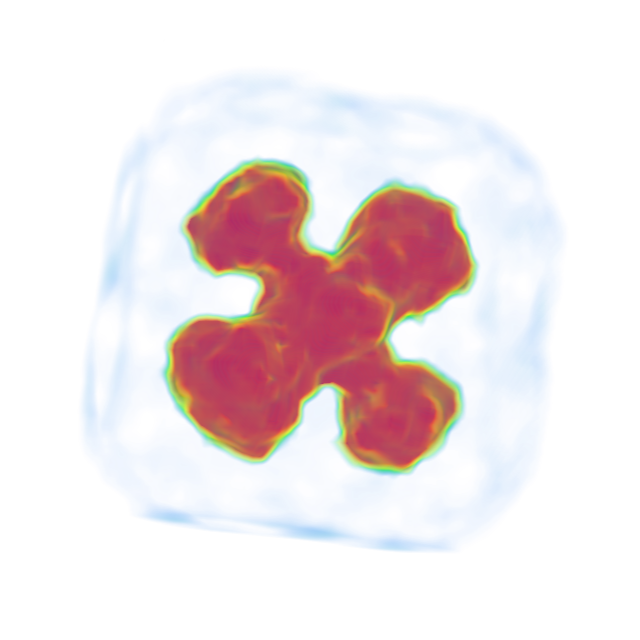}
		\caption{$t=0.15$}
	\end{subfigure}
	\begin{subfigure}[b]{0.26\textwidth}
		\centering
		\includegraphics[width=\textwidth]{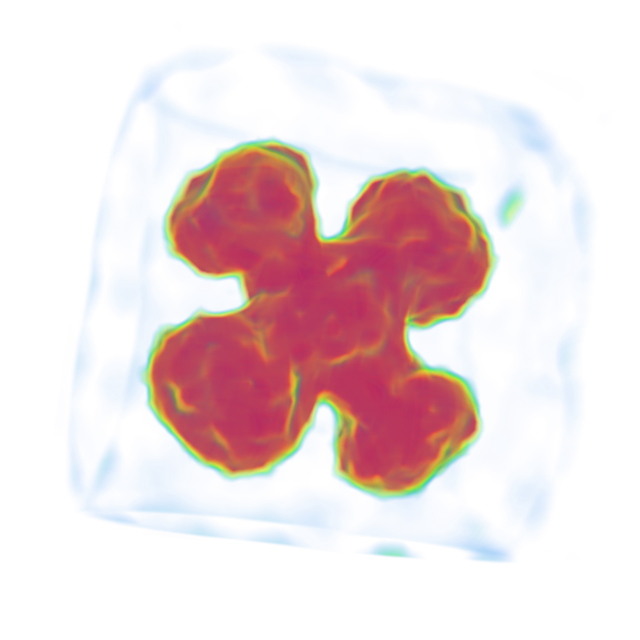}
		\caption{$t=0.25$}
	\end{subfigure}
	\begin{subfigure}[b]{0.1\textwidth}
		\centering
		\includegraphics[width=\textwidth]{u_exp1_legend.png}
	\end{subfigure}
	\caption{Snapshots of the tumour volume fraction $u$ in simulation 4.}
	\label{fig:snapshots u 3d 3}
\end{figure}

\begin{figure}[!htb]
	\centering
	\begin{subfigure}[b]{0.26\textwidth}
		\centering
		\includegraphics[width=\textwidth]{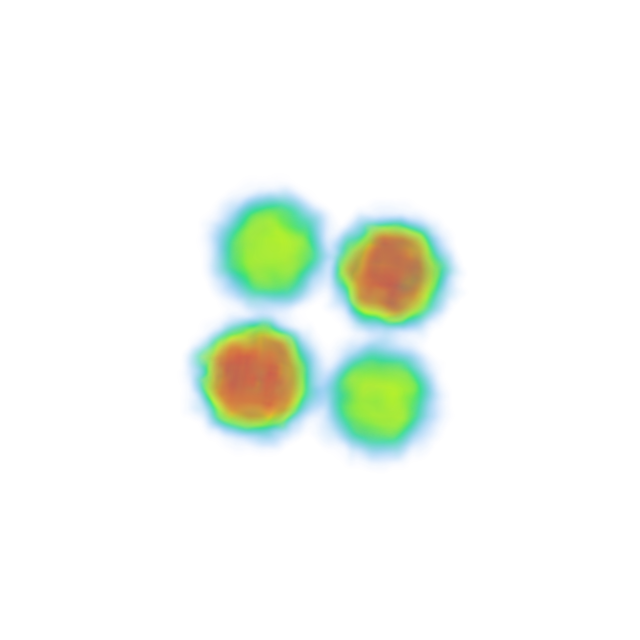}
		\caption{$t=0$}
	\end{subfigure}
	\begin{subfigure}[b]{0.26\textwidth}
		\centering
		\includegraphics[width=\textwidth]{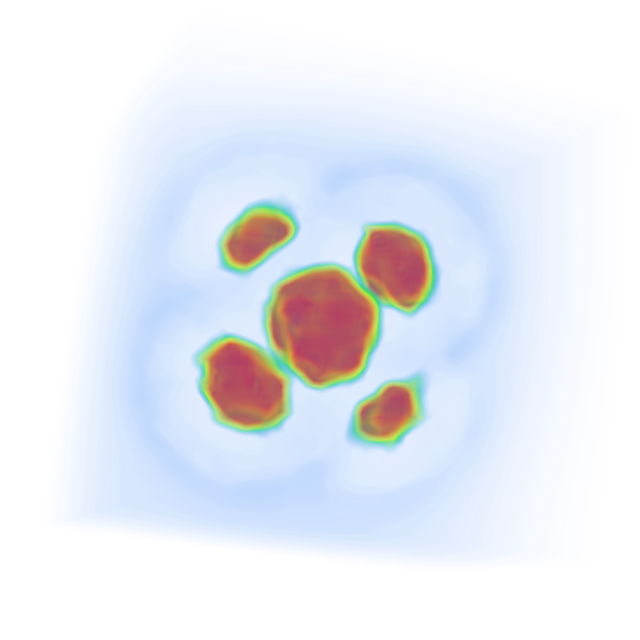}
		\caption{$t=0.05$}
	\end{subfigure}
	\begin{subfigure}[b]{0.26\textwidth}
		\centering
		\includegraphics[width=\textwidth]{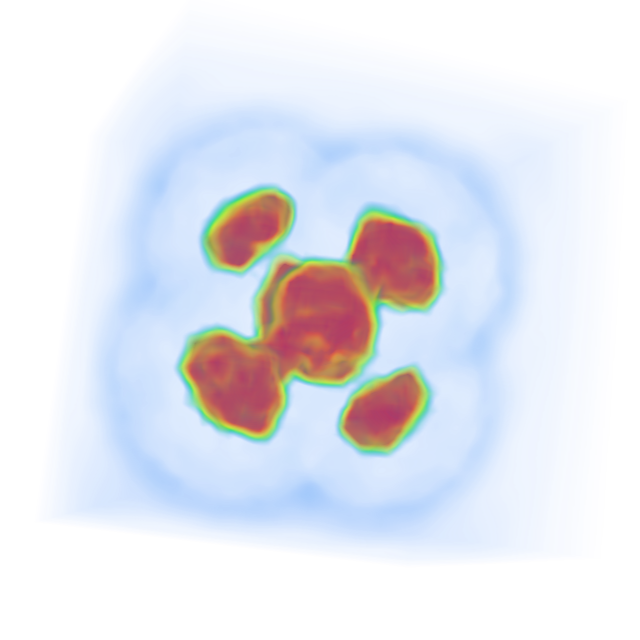}
		\caption{$t=0.08$}
	\end{subfigure}
	\begin{subfigure}[b]{0.1\textwidth}
		\centering
		\includegraphics[width=\textwidth]{n_exp1_legend.png}
	\end{subfigure}
	\begin{subfigure}[b]{0.26\textwidth}
		\centering
		\includegraphics[width=\textwidth]{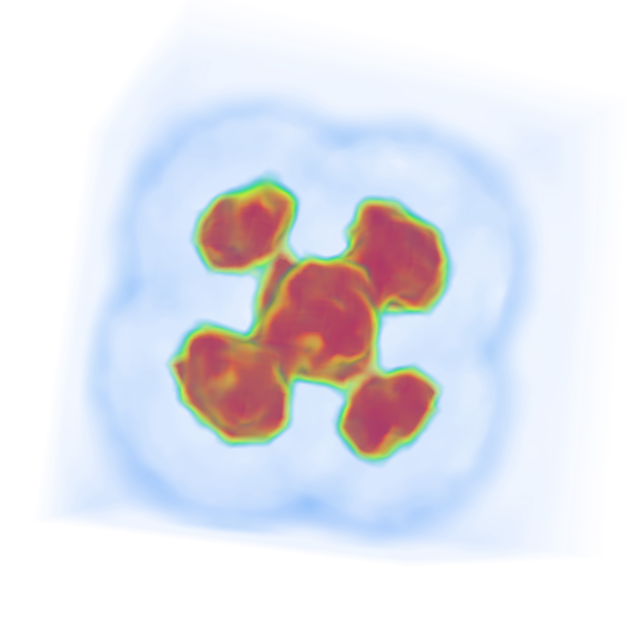}
		\caption{$t=0.1$}
	\end{subfigure}
	\begin{subfigure}[b]{0.26\textwidth}
		\centering
		\includegraphics[width=\textwidth]{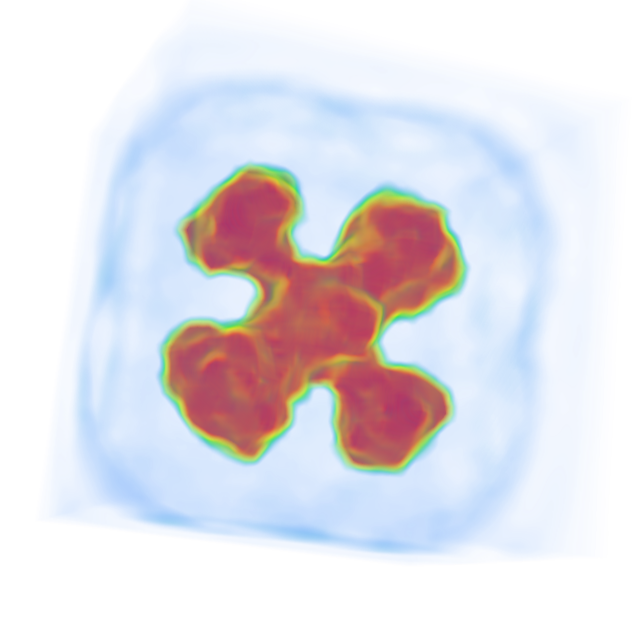}
		\caption{$t=0.15$}
	\end{subfigure}
	\begin{subfigure}[b]{0.26\textwidth}
		\centering
		\includegraphics[width=\textwidth]{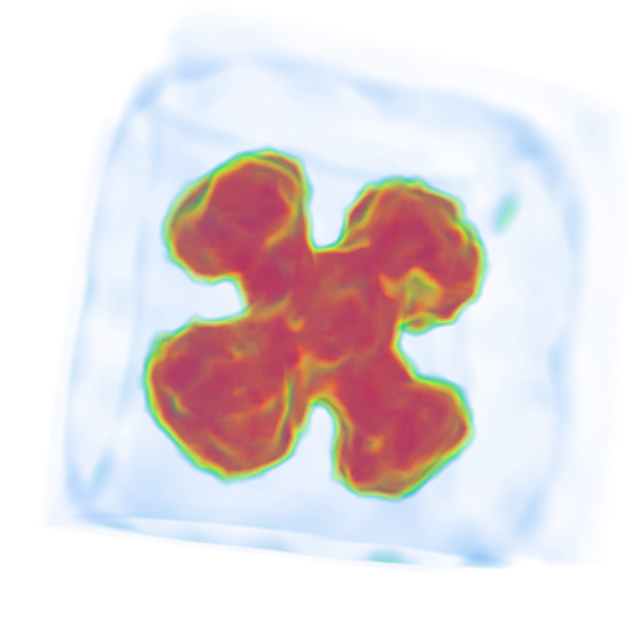}
		\caption{$t=0.25$}
	\end{subfigure}
	\begin{subfigure}[b]{0.1\textwidth}
		\centering
		\includegraphics[width=\textwidth]{n_exp1_legend.png}
	\end{subfigure}
	\caption{Snapshots of the nutrient volume fraction $n$ in simulation 4.}
	\label{fig:snapshots n 3d 3}
\end{figure}

\begin{figure}[!htb]
	\centering
	\begin{subfigure}[b]{0.47\textwidth}
		\centering
		\includegraphics[width=\textwidth]{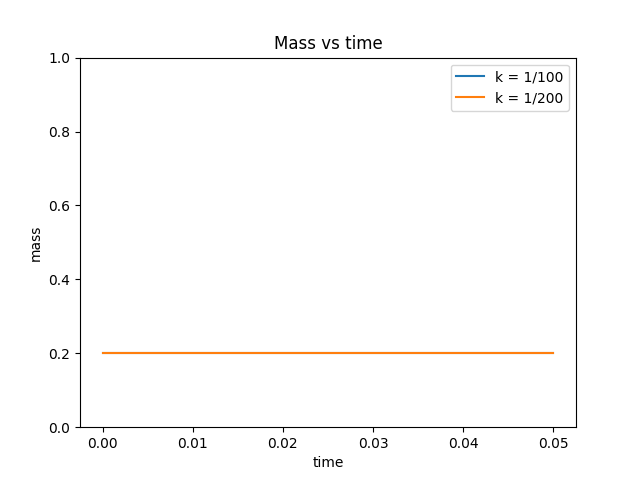}
		\caption{Graph of combined tumour and nutrient mass vs time with $\tau=1/100$ and $\tau=1/200$.}
	\end{subfigure}
    \hspace{1em}
	\begin{subfigure}[b]{0.47\textwidth}
		\centering
		\includegraphics[width=\textwidth]{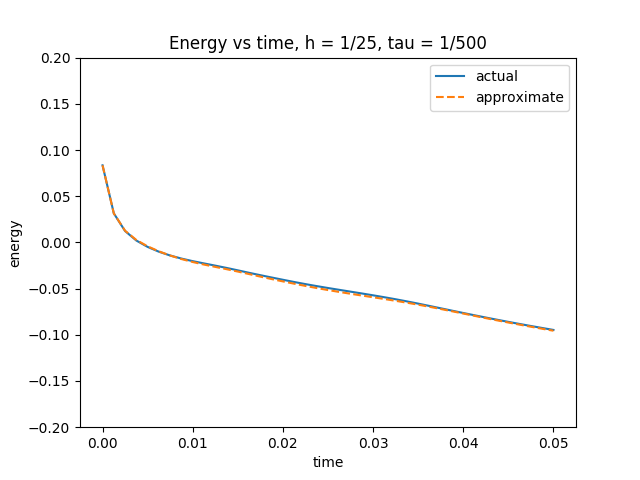}
		\caption{Graph of energy and approximate energy vs time with $h=1/25$ and $\tau=1/500$.}
	\end{subfigure}
        \caption{Conservation of mass and energy dissipation in simulation 4.}
	\label{fig:mass energy exp3}
\end{figure}

\section{Conclusion}\label{sec:conclusion}

We have presented a linear fully discrete structure-preserving finite element method for a diffuse-interface model of tumour growth, combining a scalar auxiliary variable (SAV) formulation with a stabilised mixed FEM for the Cahn--Hilliard part and standard conforming FEM for the reaction-diffusion part, along with a first-order time-stepping scheme. The proposed method preserves mass, satisfies a discrete energy dissipation law (albeit for an approximate energy functional), and requires solving only linear systems at each time step. Under suitable regularity assumptions, we establish rigorous error estimates in the $L^2$, $H^1$, and $L^\infty$ norms, achieving first-order temporal and optimal spatial convergence. Numerical experiments support the theoretical results and demonstrate the robustness of the scheme, even for less regular initial data, in capturing characteristic aggregation, as well as diffusion-driven or chemotactic tumour growth.

Future work includes extensions to higher-order temporal schemes and adaptive spatial discretisations. It would also be interesting to couple the model with Navier--Stokes~\cite{LamWu18} or Brinkman-type equations~\cite{EbeGar19} to account for convective effects, or with Keller--Segel-type equations to better capture chemotactic behaviour~\cite{RocSchSig23}. Additionally, the use of a singular logarithmic potential is necessary to maintain the phase-field variable within physically relevant interval~\cite{ColGilSigSpr23}, and developing efficient numerical methods to handle this is part of our ongoing research.

\section*{Acknowledgements}
The authors acknowledge financial support from the Edinburgh Mathematical Society Research Fund (EMS-RSF) under Grant E2503-LIN. A. Soenjaya is also supported by the Australian Government through the Research Training Program (RTP) Scholarship awarded at the University of New South Wales, Sydney. T. Tran is partially supported by the Australian Research Council, Grant number DP220101811.

Part of this work was completed during A. Soenjaya's visit to the University of Dundee. The hospitality of the Department of Mathematics at the University of Dundee and the host (Ping Lin) is acknowledged.

The authors thank the anonymous referees for their careful reading and constructive comments, which helped improve the presentation of the paper.

\newcommand{\noopsort}[1]{}\def\cprime{$'$}
\def\soft#1{\leavevmode\setbox0=\hbox{h}\dimen7=\ht0\advance \dimen7
	by-1ex\relax\if t#1\relax\rlap{\raise.6\dimen7
		\hbox{\kern.3ex\char'47}}#1\relax\else\if T#1\relax
	\rlap{\raise.5\dimen7\hbox{\kern1.3ex\char'47}}#1\relax \else\if
	d#1\relax\rlap{\raise.5\dimen7\hbox{\kern.9ex \char'47}}#1\relax\else\if
	D#1\relax\rlap{\raise.5\dimen7 \hbox{\kern1.4ex\char'47}}#1\relax\else\if
	l#1\relax \rlap{\raise.5\dimen7\hbox{\kern.4ex\char'47}}#1\relax \else\if
	L#1\relax\rlap{\raise.5\dimen7\hbox{\kern.7ex
			\char'47}}#1\relax\else\message{accent \string\soft \space #1 not
		defined!}#1\relax\fi\fi\fi\fi\fi\fi}

\end{document}